\documentclass[12pt]{amsart}
\usepackage{amsmath, amssymb, amsthm, latexsym}
\usepackage{tikz-cd}
\usepackage{enumitem}
\usepackage{amssymb,amscd,amsmath}
\usepackage{latexsym}
 \usepackage[bbgreekl]{mathbbol}
\usepackage{ dsfont }
\usepackage[center]{caption}
\usepackage{tikz}
\usepackage{tikz}
\usepackage{mathrsfs}

\usepackage{ytableau}
\usetikzlibrary{decorations.markings}
\newcounter{braid}
\newcounter{strands}

\DeclareMathAlphabet{\bsf}{OT1}{cmss}{bx}{n}

\pgfkeyssetvalue{/tikz/braid height}{1cm}
\pgfkeyssetvalue{/tikz/braid width}{1cm}
\pgfkeyssetvalue{/tikz/braid start}{(0,0)}
\pgfkeyssetvalue{/tikz/braid colour}{black}
\pgfkeys{/tikz/strands/.code={\setcounter{strands}{#1}}}

\makeatletter
\def\cross{%
  \@ifnextchar^{\message{Got sup}\cross@sup}{\cross@sub}}

\def\cross@sup^#1_#2{\render@cross{#2}{#1}}

\def\cross@sub_#1{\@ifnextchar^{\cross@@sub{#1}}{\render@cross{#1}{1}}}

\def\cross@@sub#1^#2{\render@cross{#1}{#2}}

\def\render@cross#1#2{
  \def\strand{#1}
  \def\crossing{#2}
  \pgfmathsetmacro{\cross@y}{-\value{braid}*\braid@h}
  \pgfmathtruncatemacro{\nextstrand}{#1+1}
  \foreach \thread in {1,...,\value{strands}}
  {
    \pgfmathsetmacro{\strand@x}{\thread * \braid@w}
    \ifnum\thread=\strand
    \pgfmathsetmacro{\over@x}{\strand * \braid@w + .5*(1 - \crossing) * \braid@w}
    \pgfmathsetmacro{\under@x}{\strand * \braid@w + .5*(1 + \crossing) * \braid@w}
    \draw[braid] \pgfkeysvalueof{/tikz/braid start} +(\under@x pt,\cross@y pt) to[out=-90,in=90] +(\over@x pt,\cross@y pt -\braid@h);
    \draw[braid] \pgfkeysvalueof{/tikz/braid start} +(\over@x pt,\cross@y pt) to[out=-90,in=90] +(\under@x pt,\cross@y pt -\braid@h);
    \else
    \ifnum\thread=\nextstrand
    \else
     \draw[braid] \pgfkeysvalueof{/tikz/braid start} ++(\strand@x pt,\cross@y pt) -- ++(0,-\braid@h);
    \fi
   \fi
  }
  \stepcounter{braid}
}

\tikzset{braid/.style={double=\pgfkeysvalueof{/tikz/braid colour},double distance=1pt,line width=2pt,white}}

\newcommand{\braid}[2][]{%
  \begingroup
  \pgfkeys{/tikz/strands=2}
  \tikzset{#1}
  \pgfkeysgetvalue{/tikz/braid width}{\braid@w}
  \pgfkeysgetvalue{/tikz/braid height}{\braid@h}
  \setcounter{braid}{0}
  \let\sigma=\cross
  #2
  \endgroup
}
\makeatother

\input xypic
\newtheorem{theorem}{Theorem}[subsection]

\newtheorem{SectionTheorem}{Theorem}
\newtheorem{proposition}[theorem]{Proposition}
\newtheorem{SectionCorollary}{Corollary}

\newtheorem{lemma}[theorem]{Lemma}

\newtheorem{corollary}{Corollary}
\newtheorem{definition}[theorem]{Definition}


\def\Z{\mathbb{Z}}

\def\C{\mathbb{C}}

\def\C{\mathbb{C}}

\def\N{\mathbb{N}}

\def\F{\mathbb{F}}

\def\md{\mathcal{D}}

\def\Zpk{\mathbb{Z}/p^{k}}
\def\Zpk1{\mathbb{Z}/p^{k-1}}

\newcommand{\rref}[1]{(\ref{#1})}

\newcommand{\beg}[2]{\begin{equation}\label{#1}#2\end{equation}}

\def\F{\mathbb{F}}

\def\sl2{\widetilde{SL_{2}(\Z)}}

\def\md
\def\rank{\operatorname{rank}}

\title[]{Howe Duality over Finite Fields III: Full computation and the Gurevich-Howe conjectures}

\author{Sophie Kriz}
\thanks{The author was supported by a 2023 National Science Foundation
Graduate Research Fellowship, no. 2023350430}

\begin{document}

\maketitle

\vspace{-5mm}

\begin{abstract}

In this third paper in a series on type I Howe duality for finite fields,
we give a complete description of the restriction of the oscillator representation
over a finite field to products of dual pairs of symplectic and orthogonal groups in all cases
that occur. 
We also provide a dictionary with the notation of S.-Y. Pan, who identified
which tensor products of irreducible representations occur with non-zero multiplicity.
As an application, we give a recursive construction of all irreducible
complex representations of finite symplectic and orthogonal groups
and a recursive formula for the characters of unipotent cuspidal representations. We also give a proof
of the Gurevich-Howe rank and exhaustion conjectures for type C groups.

\end{abstract}

\vspace{-5mm}

\tableofcontents

\vspace{-7mm}

\section{Introduction}

The purpose of this paper is to complete the explicit description,
started in the previous papers \cite{TotalHoweI, TotalHoweII}, of the decomposition
of the oscillator representation of a symplectic group over a finite field 
of characteristic not equal to $2$ into
irreducible representations of a product of a dual pair of type I,
which consists of a symplectic and an orthogonal group.

This problem was suggested by R. Howe in \cite{HoweFiniteFields}.
The first approaches this question considered the behavior of {\em unipotent} representations appearing in a restricted oscillator representation. J. Adams and A. Moy \cite{AdamsMoy} proved unipotent cuspidal representations paired with each other at their first occurrence in some oscillator representation.
In combination with their calculation of the decomposition of the unipotent part
of an oscillator representation restricted to type II dual pairs (consisting of general linear groups
or unitary groups), A.-M. Aubert, J. Michel, and R. Rouquier \cite{AubertMichelRouquier} formed a conjecture
on the decomposition of the unipotent part in the case of type I dual pairs.

Later, S. Gurevich and R. Howe \cite{HoweGurevich, HoweGurevichBook} 
approched the problem of finite field Howe duality from a different perspective,
by considering whole oscillator representations, restricted to dual pairs in a certain {\em stable range}.
For each such restriction, they constructed one-to-one correspondences sending each representation
of the orthogonal subgroup in the pair to a newly occurring irreducible representation of the symplectic
subgroup. This construction, which they called the {\em eta correspondence},
is related to a concept of {\em rank} which also carries significance for certain questions
about the group dynamics of finite groups of Lie type.

Returning to the study of the unipotent part of a restricted oscillator representation,
in \cite{Pan1}, S.-Y. Pan proved the type I conjecture of Aubert, Michel, and Rouquier by computing
the projection of its character to the class functions spanned by virtual Deligne-Lusztig characters.
In \cite{LiuWang}, D. Liu and Z. Wang also used this calculation to extend the results of Adams and Moy.
Finally, in \cite{Pan2}, Pan proved a compatibility of the results of \cite{Pan1} with Lusztig's
parametrization of irreducible characters to conclude a classification of the irreducible
tensor products appearing with non-zero multiplicity in the restriction of an oscillator representation
to a type I dual pair.

\vspace{2mm}

In the present paper and \cite{TotalHoweI, TotalHoweII},
we approach the question of finite field Howe duality
using a different method, which directly computes the decomposition of
the restricted oscillator representation and verifies that each distinct
tensor product appears with multiplicity one. In this paper,
we use the structure of endomorphism
algebras and interpolated category theory to extend results for stable ranges
to every possible case of type I dual pair.

As a direct application of the methods used in this paper, we also give an inductive construction of all
irreducible representations of symplectic and orthogonal groups over finite fields,
which in particular gives a recursive formula for the characters of the unipotent cuspidal
representations. Finally, we give a dictionary of our notation with the notation
of S.-Y. Pan \cite{Pan2}, who identified which tensor products of irreducible representations
of type I dual subgroups occur with non-zero multiplicity in the restricted oscillator representation,
and we prove that the multiplicity is always $1$.
As another application, we prove the Gurevich-Howe rank conjecture
\cite{HoweGurevich}, Conjecture 0.3.8 on the coincidence of tensor rank and $U$-rank,
and the exhaustion conjecture \cite{HoweGurevich}, Conjecture 0.4.12, which follows.

The part of this program completed in \cite{TotalHoweI, TotalHoweII}
was to treat the so-called {\em stable ranges}, where the rank of one of the
groups in the pair is much greater than the other.
In \cite{TotalHoweI}, the general form of that case of the decomposition was established, in terms
of certain correspondences between the sets of irreducible
representations of symplectic and orthogonal groups,
and, in \cite{TotalHoweII}, it was completely described in terms of the classification of irreducible
representations of finite groups of Lie type obtained from Lusztig's paramterization
of the irreducible characters.

The strategy of this paper is to break up the remaining cases into two {\em metastable ranges}.
In the metastable ranges, the stable picture breaks down in two ways, both of which
are related to the occurence of {\em generalized Lusztig symbols},
which relax some of the defining conditions of a Lusztig symbol.
One type of generalized symbol predicts a $0$-dimensional representation -
those terms are simply omitted. However, one can also encounter
alternating sums of
induction terms coupled with generalized Lusztig symbols of the same dimension.
These are first shown to be genuine (as opposed to virtual) representations.
This is done in Section \ref{InterpolationSection}.
The other step is to compute them completely, which is done in Section \ref{AlternatingSumSection}.

\vspace{3mm}

To describe our results more concretely, we need some notation.
Consider a finite field $\F_q$ of characteristic not equal to $2$, a symplectic
$\F_q$-vector space $V$, and an $\F_q$-vector space $W$ with a non-degenerate
symmetric bilinear form $B$. In this finite field context,
S. Gurevich and R. Howe (see e.g. \cite{HoweGurevich})
proposed the problem of describing explicitly the restriction of the oscillator representation
of $\text{Sp}(V\otimes W)$ to the subgroup $\text{Sp}(V) \times \text{O}(W,B)$.
The previous two papers in this series \cite{TotalHoweI, TotalHoweII}
described this decomposition explicitly in the {\em symplectic stable range}
(i.e. where $dim (V) \geq 2 dim (W)$) and the {\em orthogonal stable range}
(where $dim (V)$ is less than or equal to the dimension of the maximal isotropic
subspace of $W$).

Denote by $\widehat{G}$ the set of isomorphism classes of irreducible complex
representations of a finite group $G$. Then in the symplectic resp. orthogonal
stable range, there are correspondences
\beg{EtaSympStIntro}{\eta^V_{W,B}:\widehat{\text{O}(W,B)} \hookrightarrow \widehat{\text{Sp}(V)}}
\beg{ZetaOrthoStIntro}{\zeta^{W,B}_V: \widehat{\text{Sp}(V)} \hookrightarrow \widehat{\text{O}(W,B)}}
(explicitly constructed in \cite{TotalHoweI} and described in \cite{TotalHoweII} in terms of 
the classification of irreducible representations
obtained from the Lusztig parametrization of the irreducible characters of finite groups of Lie type, see e.g. \cite{Lusztig})
so that the Gurevich-Howe decomposition is a direct sum of terms of the form
\beg{EtaTopBits}{\rho \otimes \eta^V_{W,B} (\rho)}
resp.
\beg{ZetaTopBits}{\zeta^{W,B}_V (\pi) \otimes \pi}
and ``degenerate terms," which can be explicitly descried as tensor
products analogous to \rref{EtaTopBits}, \rref{ZetaTopBits} involving
Harish-Chandra induced modules (i.e. parabolic Verma modules)
obtained from smaller choices of $V$, $W$.

The purpose of the present paper is to describe the Gurevich-Howe decomposition in
the remaining cases, broken up into two {\em metastable ranges}
which, together with the stable ranges, cover all the cases of $V$, $W$. The precise
definition of the metastable ranges is technical and will be given in Subsection \ref{MetastableSubSect}
below (Definition \ref{MetaStablDefn}).

Our main result can now be stated in broad terms as follows:

\begin{SectionTheorem}\label{IntroductionTheorem}
Consider a type I reductive dual pair $(\text{Sp}(V), \text{O}(W,B))$.
Then in the symplectic resp. orthogonal metastable ranges, there are
correspondences
\beg{EtaMetaThmIntro}{\eta^V_{W,B}: \widehat{\text{O}(W,B)} \rightarrow \widehat{\text{Sp}(V)} \cup \{0\}}
\beg{ZetaMetaThmIntro}{\zeta^{W,B}_V: \widehat{\text{Sp}(V)} \rightarrow \widehat{\text{O}(W,B)} \cup \{0\},}
explicitly described in terms of the classification of the irreducible representations of
$\text{O}(W,B)$ and $\text{Sp} (V)$, such that the restriction of the
oscillator representation of $\text{Sp}(V\otimes W)$ to $\text{Sp}(V) \times \text{O}(W,B)$ is a direct
sum of terms of the form \rref{EtaTopBits} resp. \rref{ZetaTopBits} and explicitly described
as tensor products analogous to \rref{EtaTopBits} resp. \rref{ZetaTopBits} involving alternating
sums of Harish-Chandra induced modules for smaller choices of $V, W$, each of which adds
up to a linear combination of irreducible representations with positive integral coefficients.
The alternating sums are explicitly resolved as
sums of irreducible representations in terms of Lusztig's paramterization of irreducible characters.
\end{SectionTheorem}

\vspace{1mm}

\noindent We restate this more precisely after we have introduced the necessary notation, in
Theorems \ref{ExplicitTheoremSymp} and \ref{ExplicitTheoremOrtho} below. The
description of the alternating sum coefficients appearing with the eta and zeta correspondence
is given in Theorem \ref{AlternatingSumIntroPropStatement}.

\vspace{1mm}

\noindent {\bf Remark:}
This result extends the results of
S.-Y. Pan \cite{Pan1, Pan2}. We discuss this in Appendix \ref{PanAppendix} where
we see how our decompositions of the restricted oscillator representation recover Pan's identification
of the sumands which occur with non-zero multiplicity. In fact, our proof explicitly
shows that they all must appear with multiplicity $1$ and that the sum of the products
of the occuring dimensions does really add up to
$q^{dim (V) \cdot dim (W)/2} = dim (\omega [ V\otimes W])$, using a somewhat
subtle combinatorial
argument. (The cases of rank $\leq 6$ were verified on a computer using
Maple and the GAP package CHEVIE.)

\vspace{3mm}

Further, our organization of these summands in terms of one-to-one correspondences between sets of irreducible representations of symplectic groups and those of orthogonal groups reconciles
Pan's result with the original program of Howe 
\cite{HoweGurevich, HoweGurevichBook, HoweFiniteFields, HoweKobayashi}
(see Appendix A). 

Every irreducible representation of a fixed symplectic group
is obtained by applying the eta correspondence to an irreducible representation
of an orthogonal group forming a symplectic metastable reductive dual pair.
Similarly, every irreducible representation of an orthogonal group is obtained from a signed
zeta correspondence. In particular, recursively applying the eta and zeta correspondences
gives an inductive
construction of every irreducible representation from certain ``terminal representations"
(see Subsection \ref{InductiveConstruction}).

Even more concretely, by organizing the decomposition of the restricted oscillator representation
into the eta and zeta correspondence, we can isolate 
describe the ``top part" consisting of tensor products of irreducible representations
of orthogonal or symplectic groups
with their images under the eta or zeta correpsondence, respectively.
The character of this top part is approximable using the characters of the oscillator representations
(which are computable from the Schr\"{o}dinger model), and orthogonality of characters
can then be used to obtain a formula for the character of the eta or zeta correspondence
applied to a representation, in terms of the characters of the original input
representation and the characters of the oscillator representation
(see Subsection \ref{CharaAlgorithmSubSect}). The unipotent cuspidal representations
are obtainable by recursively applying the eta and zeta correspondences to trivial
and sign representations, meaning this process gives a concrete construction
their characters.


\vspace{3mm}

One important application of our description of the eta correspondence
is that it can be used to answer questions related the
character theory of finite symplectic groups
(see \cite{HoweGurevich, HoweGurevichBook}).
In particular, in Section \ref{GurevichHoweSection},
we prove the Gurevich-Howe {\em rank conjecture}, which predict the equality
of a two kinds of ``ranks" for certain representations.
The first notion of rank defined by Gurevich and Howe
is called {\em $U$-rank} and is defined as the maximal
``rank" of a character in the restriction of a representation to the Siegel unipotent subgroup of a
symplectic group. We denote it by $rk_U$.
The second notion of rank is called {\em tensor rank}
and is defined to be the minimal natural number $k$ such that every irreducible
summand
of the input representation is contained in a tensor product of less than
or equal to $k$ oscillator representations. We denote it by $rk_\otimes$. The main result of Section
\ref{GurevichHoweSection} is
that Conjecture 0.3.8 of \cite{HoweGurevich} holds:

\begin{SectionTheorem}[The Gurevich-Howe Rank Conjecture, \cite{HoweGurevich}, Conjecture 0.3.8] \label{MainMatchRes}
For an irreducible representation $\rho$ of a finite symplectic group $\text{Sp}_{2N} (\F_q)$ (for $q$ an odd prime power), if the $U$-rank of $\rho$ is strictly less than $N$, then it agrees with the tensor rank of $\rho$:
\beg{NeededMatchingThm}{rk_U (\rho ) = rk_\otimes (\rho).}
\end{SectionTheorem}

\vspace{3mm}

\noindent We then also have:

\begin{SectionCorollary}[The Gurevich-Howe Exhaustion Conjecture, \cite{HoweGurevich}, Conjecture 0.4.12]
For every choice of $0< n<N$, every irreducible $\text{Sp}_{2N} (\F_q)$-representation
of $U$-rank $n$ is produced in the image of an eta correspondence $\eta^{\F_q^{2N}}_{W,B}$
for one of the two choices of orthogonal spaces $(W,B)$ of dimension $n$.
\end{SectionCorollary}

\begin{proof}
In \cite{HoweGurevich}, Theorem 0.4.13, it is proved that
every irreducible representation of $\text{Sp}_{2N} (\F_q)$ with tensor rank $n$ is produced
in the image of such an eta correspondence. Thus,
since Theorem \ref{MainMatchRes} states that, in this range, tensor rank
is always equal to $U$-rank, the exhaustion conjecture follows immediately.
\end{proof}

\vspace{1mm}

\noindent {\bf Remark:} Similar results for groups of type A were proved by
R. M. Guralnick, M. Larsen, and P. H. Tiep \cite{PiGuralLarsTiep} and for groups
of type B and D by M. Larsen and P. H. Tiep \cite{LarsenTiep}.

\vspace{1mm}

The correspondences \rref{EtaMetaThmIntro}, \rref{ZetaMetaThmIntro}
are not formal extensions of the correspondences \rref{EtaSympStIntro}, \rref{ZetaOrthoStIntro}
to the metastable range. If we extended the correspondences \rref{EtaSympStIntro},
\rref{ZetaOrthoStIntro} formally, we would obtain ``generalized Lusztig symbols" which could
either contain repeated terms (these translate to the term $\{0\}$ in
\rref{EtaMetaThmIntro}, \rref{ZetaMetaThmIntro}) or non-monotone terms;
those are replaced by different genuine Lusztig symbols.
The reason for working in the two metastable ranges is to avoid the appearance
of illegal Lusztig symbols with negative terms, which are more complicated to resolve.

Our method for proving Theorem \ref{IntroductionTheorem} is interpolation
of semisimple pre-Tannakian categories \cite{DeligneSymetrique, DeligneTensor, HarmanSnowden, VectorDelannoy, OscillatorRepsFull, InterpolatedSchemes}.
There are other possible methods for approaching this
problem. Notably, analogously as the general linear group is embedded into the
multiplicative semigroup of all square matrices, the symplectic group
is embedded into the {\em oscillaor semigroup} discovered by Howe
\cite{HoweFiniteFields, HoweGurevich, HoweGurevichBook, HoweKobayashi}.
Using this semigroup, one can also obtain results on the Gurevich-Howe decomposition
beyond the stable ranges (see e.g. Proposition 8.2.1 of \cite{HoweGurevich}).

Since \rref{EtaMetaThmIntro} and \rref{ZetaMetaThmIntro} cover all the cases
of $V$ and $W$, it is possible to use Theorem \ref{IntroductionTheorem}
for an inductive construction of all irreducible representations of $\text{Sp}(V)$, $\text{O}(W,B)$
(see Subsection \ref{InductiveConstruction}).

\vspace{3mm}

This paper is organized as follows: In Section \ref{BackgroundSection}, we
discuss some background of the oscillator representations and the 
classification of irreducible representations of symplectic and orthogonal groups
obtained from Lusztig's paramterization of irreducible characters.
In Section \ref{StatementSection},
we describe the metastable ranges, construct an eta or zeta correspondence
for every case of type I reductive dual pair, and discuss the alternating sums replacing parabolic
induction in the metastable range. This establishes the necessary notation to restate
Theorem \ref{IntroductionTheorem} in concrete terms.
In Section \ref{InterpolationSection}, we discuss interpolated representation categories and
the analogues of the results of \cite{TotalHoweI, TotalHoweII}, prove
Theorem \ref{IntroductionTheorem}, and discuss the representation-theoretical implications. 
In Section \ref{GurevichHoweSection}, we discuss the Gurevich-Howe
rank conjecture, proving Theorem \ref{MainMatchRes}.
In Section \ref{AlternatingSumSection}, we resolve the alternating sums
appearing as coefficients
in the decomposition of the restricted oscillator representation.

In Appendix \ref{PanAppendix}, we discuss a dictionary between the notation and results
of this series of papers with S.-Y. Pan's identification of the pairs
of irreducible representations whose tensor products appear with non-zero multiplicity
in the restricted oscillator representation.

\section{Background}\label{BackgroundSection}

We begin by recalling the results of \cite{TotalHoweI, TotalHoweII} more precisely.
First, we fix notation: For a symplectic group $\text{Sp}(\mathbf{V})$, we write $\omega_a [ \mathbf{V}]$ to denote the {\em oscillator representation} arising from the Weil-Shale representation of the Heisenberg group on $\mathbf{V}$ with central (non-trivial, additive) character in $\mathbb{F}_q$ corresponding to $a \in \mathbb{F}_q^\times$ under a fixed identification of $\mathbb{F}_q$ with its Pontrjagin dual. In the case of $a= 1$, we omit the subscript and write $\omega [ \mathbf{V}] = \omega_1 [ \mathbf{V}]$.

Now let us consider a type I reductive dual pair of subgroups of $\text{Sp}(\mathbf{V})$, which must be of the form
\beg{GeneralIntroTypeIRedDualPair}{(\text{Sp}(V), \text{O}(W,B)) \subseteq \text{Sp} (\mathbf{V}),}
for $\mathbb{F}_q$-spaces $V$ and $W$, say with symplectic and symmetric bilinear forms $S$ and $B$ respectively, so that $\mathbf{V} = V \otimes W$, and we consider its symplectic form to be $S\otimes B$(see, for example, \cite{HoweFiniteFields}). Tensoring matrices gives an inclusion of the product $\text{Sp}(V) \times \text{O}(W,B)$ into $\text{Sp}(\mathbf{V})$. Note that if there exists a $k$-dimensional isotropic subspace
of $W$ (resp. $V$), then $B$ (resp. $S$) can be expressed as a direct sum of $k$ copies of a hyperbolic
$2$-dimensional symmetric bilinear (resp. symplectic) form with a form $B[-k]$ (resp. $S[-k]$)
on a $dim (W) - 2k$-dimensional $W[-k]$ (resp.
$dim (V) -2k$-dimensional space $V[-k]$). If $k$ is a dimension of an isotropic subspace,
we write $P_k^B$ (resp. $P_k^V$) for the maximal parabolic subgroup
of $\text{O}(W,B)$ (resp. $\text{Sp}(V)$) with Levi factor $\text{GL}_k (\F_q) \times \text{O}(W[-k], B[-k])$ (resp.
$\text{GL}_k (\F_q) \times \text{Sp}(V[-k])$).

We will again write, for a
subgroup $H\subseteq G$, $\text{Ind}_H^G (\rho)$
for the induction of an $H$-representation $\rho$ to $G$.

In \cite{TotalHoweI}, we defined two stable ranges of such type I reductive dual pairs: We say a
pair \rref{GeneralIntroTypeIRedDualPair} is in the {\em symplectic stable range} if $dim (W) \leq dim (V)$, and similarly, we say it is in the {\em orthogonal stable range} if the dimension of $V$ is less than or equal to the dimension of a maximal isotropic subspace of $W$ with respect $B$. Let us denote
by $h_W$ the dimension of a maximal isotropic subspace of $W$.
We proved that, for $(\text{Sp}(V), \text{O}(W,B))$ in the symplectic stable range, the restriction of
$\omega [V\otimes W]$ to a $\text{Sp}(V) \times O(W,B)$-representation is
\beg{TotalHoweIStSymp}{ \bigoplus_{k=0}^{h_W}
\bigoplus_{\rho \in \widehat{O(W[-k], B[-k])}} \eta^V (\rho) \otimes
\text{Ind}_{P_k^B} (\rho \otimes \epsilon (\text{det})) }
for a system of mutually disjoint injections
$$\eta^V_{W,B} : \widehat{\text{O}(W,B)} \hookrightarrow \widehat{\text{Sp}(V)}$$
called the {\em eta correspondence} (ommiting the subscript when the source is
determined). See also, for this case, the original
papers of S. Gurevich and R. Howe
finding the eta correspondence \cite{HoweGurevich, HoweGurevichBook}
and an approach.
Similarly, for $(\text{Sp}(V), \text{O}(W,B))$ in the orthogonal stable range, writing $dim (V) = 2N$,
the restriction of $\omega [ V\otimes W]$ to $\text{Sp}(V) \times \text{O}(W,B)$ decomposes as
\beg{TotalHoweIStOrtho}{\bigoplus_{k=0}^{N}
\bigoplus_{\rho \in \widehat{\text{Sp}(V[-k])}} 
\text{Ind}_{P_k^V} (\rho \otimes \epsilon (\text{det})) \otimes \zeta^{W,B} (\rho)}
for a system of mutually disjoint injections
$$\zeta^{W,B}_V: \widehat{\text{Sp}(V)} \hookrightarrow \widehat{\text{O}(W,B)}$$
called the {\em zeta correspondence} (again, ommiting the subscript when the source
is determined).

\vspace{3mm}

Therefore, in either stable range, the problem of Howe duality reduces to explicitly
computing the eta and zeta correspondences, which we was the main result \cite{TotalHoweII},
using the classification of irreducible representations obtained from Lusztig's parametrization
of irreducible characters.
Recall that, broadly, an irreducible representation of a finite group of Lie type
is classified by data consisting of a conjugacy class of a semisimple element $(s)$ in the
dual group $G^*$ (the ``semisimple part"), a unipotent representation of (the dual of) the 
identity component of the centralizer
of $s$ in $G^*$,
and possible ``central sign
data" when $Z(G)$ is disconnected; we discuss this in more detail in Section \ref{BackgroundSection} below.
For $(\text{Sp}(V), \text{O}(W,B))$ in the symplectic stable range, our computation
of the eta correspondence
$$\eta^V_{W,B}: \widehat{\text{O}(W,B)} \rightarrow \widehat{\text{Sp}(V)} $$
can be summarized as transforming the Lusztig data of an irreducible representation of
$\text{O}(W,B)$ into Lusztig data specifying an irreducible representation of $\text{Sp}(V)$ by
\begin{itemize}
\item Adding an appropriate number of $-1$ eigenvalues (and a single $1$ eigenvalue, with position
depending on the action of $\Z/2 \subseteq O(W,B)$) to
the semisimple part if $dim (W)$ is odd,
and $1$ eigenvalues if $dim (W)$ is even.

\vspace{1mm}

\item Altering the unipotent part by considering the single changed
factor of the the 
identity component of centralizer of the new semisimple part
(corresponding to $-1$ eigenvalues
if $dim (W)$ is odd and $1$ eigenvalues of the semisimple part if $dim (W)$) and adding a single coordinate to the Lusztig symbol to get the appropriate new rank and defect.

\vspace{1mm}

\item Central sign data is determined by the quadratic character applied to the semisimple part (as a torus element) multiplied by the discriminant of $B$ when $dim (W)$ is odd, and the central sign data of
$SO(W,B)$ when $dim (W)$ is even.
\end{itemize}

Similarly,
for $(\text{Sp} (V), \text{O}(W,B))$ in the orthogonal stable range, our construction of the zeta correspondence
$$\zeta^{W,B}_V: \widehat{\text{Sp}(V)} \rightarrow \widehat{\text{O}(W,B)} $$
can be summarized by altering the Lusztig data of an irreducible representation of $\widehat{\text{Sp}(V)}$
by adding $-1$ eigenvalues to the semisimple part if $dim (W)$ is odd, $1$ eigenvalues if $dim (W)$ is even,
altering the affected factor of the unipotent part by adding a single appropriate coordinate, and assigning
central sign data determined by the quadratic character of the original semisimple part and $disc(B)$
or the original central sign data.

To be even more precise, we need to also discuss
the classification of irreducible representations. Let us consider a reductive group $G$ over $\F_q$, specifically with a focus on the
symplectic and special and full orthogonal group cases.
We recall that the set of irreducible representations of $G$ can be partitioned into disjoint
subsets corresponding to
geometric conjugacy classes of the data of a maximal torus $T$ in $G$ and an
irreducible character $\theta$ of $T(\F_q)$ (this partition is the one in
whose Deligne-Lusztig induction $R_T (\theta)$ a given irreducible $G$-representation
appears in, see \cite{DeligneLusztig, LDLDisconn, DeligneLusztigDisconnected, DigneMichel,  Lusztig}).
By considering the dual group $G^*$, the data of $(T, \theta)$ defines a
conjugacy class of a semisimple element $(s) \in G^* (\F_q)$, allowing one
to consider a partition of $\widehat{G}$ into subsets called the {\em Lusztig series} corresponding
to $(s)$. The irreducible $G$-representations in the Lusztig series corresponding to $(1) \in G^*(\F_q)$
are the {\em unipotent} irreducible representations.
To consider the case of $G= \text{Sp}_{2r}$ or $\text{SO}_{2r}^\pm$ with disconnected
center $Z(G) = \mu_2$, the series corresponding to $s$ can be further
partioned according to the elements $Z_{G^* (\F_q)} (s)/ Z_{G^* (\F_q)} (s)^\circ$
when this quotient is non-trivial (see \cite{LDLDisconn, DeligneLusztigDisconnected,DigneMichel}). Note that in the cases we consider here,
it may only be trivial or $\mu_2 = \{\pm 1\}$.
We call this extra sign the {\em central sign data} corresponding to a representation
when it corresponds to such an $(s)$.

The Lusztig series corresponding to a conjugacy class $(s)$ of a semisimple
element in $G^*(\F_q)$ (and a choice of central sign data in the case when $Z_{G^* (\F_q)} (s)/ Z_{G^* (\F_q)} (s)^\circ \cong \mu_2$)
is in bijective correspondence with the set of isomorphism classes of
irreducible unipotent
representations $u$ of the dual of the 
identity component of $s$'s centralizer
$$Z_{G^*(\F_q)} (s)^\circ .$$

This data of $(s) \in G^* (\F_q)$, possible central sign data $\alpha$,
and a unipotent irreducible representation $u$ of
the (dual of the) identity component of $s$'s centralizer
determines a $G(\F_q)$-representation which we denote by $r^G [(s), u , \alpha]$
or $r^G [ (s), u]$ depending on whether central sign data is present or not.
The dimension of this representation is
$$\frac{|G(\F_q)|_{q'}}{|(Z_{G^*(\F_q)} (s))^*|_{q'}} \text{dim} (u),$$
(where $|?|_{q'}$ denotes the prime to $q$ part of a group order).
This can also be expressed in our cases as
$$\begin{array}{c}
\text{dim} (r^G [(s), u] ) = \frac{|G|_{q'}}{|Z_{G^*(\F_q)} (s)^\circ|_{q'}} \text{dim}(u), \\[1ex]
 \text{dim} (r^G [(s), u, \alpha] ) = \frac{|G(\F_q)|_{q'}}{2 |Z_{G^*(\F_q)} (s)^\circ|_{q'}} \text{dim}(u).
\end{array}$$
(Recall that the order of a reductive group over a finite field is equal to the order of its dual.)

For a choice of data $(s)$ where the quotient of its centralizer by its identity
component is $\mu_2$, and for an irreducible unipotent representation $u$ of
(the dual) of this identity component, we may also consider the representation
\beg{SplittingWithPM1s}{r^{\text{Sp}_{2N} (\F_q)}[(s),u] = r^{\text{Sp}_{2N} (\F_q)}[(s), u, +1] \oplus r^{\text{Sp}_{2N} (\F_q)}[(s), u, -1],}
giving
\beg{SplittingWithPM1sHalfDim}{ dim (r^{\text{Sp}_{2N} (\F_q)}[(s), u, \pm 1]) = \frac{dim (r^{\text{Sp}_{2N} (\F_q)}[(s), u])}{2}.}
For the symplectic group, it turns out that
central sign data occurs precisely when $s$ has $-1$ as an eigenvalue
and the tensor factor of $u$ corresponding to the factor of
the identity component of $s$'s centralizer contributed
by $-1$ eigenvalues is non-degenerate.

We also consider the case of disconnected $G$, particularly the orthogonal
groups $\text{O}_{2m+1} (\F_q)$ and $\text{O}_{2m}^\pm (\F_q)$. In these cases,
we classify the irreducible representation according to data defining an irreducible representation
of the corresponding special orthogonal group and certain {\em extension data}
specifying a certain summand of the induction when it splits. In the case of odd orthogonal groups,
the extension data always consists of a single sign specifying the
action of the $\mu_2$ center on an irreducible representation, since we have a splitting
$\text{O}_{2m+1} (\F_q)=\mu_2 \times \text{SO}_{2m+1} (\F_q)$. In the case
of even orthogonal groups, the extension data is more complex and less canonical.

\subsection{The semisimple data}

The data of a conjugacy class of a semisimple element in a finite group of Lie type
is equivalent to the data
of its eigenvalues (under the action of the Weyl group).
In any symplectic or special orthogonal group, every maximal torus is isomorphic
to a product of $SO_2^\pm$ factors, possibly on field extensions of $\mathbb{F}_q$
\beg{TorusMadeOfSO2s}{T \cong \prod_{i=1}^k SO_2^\pm (\F_{q^{n_i}})}
such that $n_1+ \dots + n_k$ adds up to the total rank, and in the case of tori in 
even special orthogonal groups, the sign $\pm$ denoting whether or not the full group is equal to the
product of signs appearing in each $SO_2^\pm $ factor. Note that each factor is cyclic
$SO_2^\pm (\F_{q^n}) \cong \mu_{q^n \mp 1}$,
giving an identification of $T$ with its Pontrjagin
dual. 

Given a choice of semisimple part of the classification data
consisting of a conjugacy class $(s) \in G^* (\F_q)$ before proceeding further
with describing the rest of the classification data, we must write
down the 
identity component of its centralizer $Z_{G^* (\F_q)} (s)$.
To compute this, consider $s$ as an element of some maximal
torus $T$ as in \rref{TorusMadeOfSO2s}, minimizing $n_i$ where possible.
Coordinates $\pm 1 \neq \lambda \in SO_2^\pm (\F_{q^n})$ of multiplicity
$j$ give a factor of the form $U_j^\pm (\F_{q^n})$ (where we use the notation
$U^+_j= GL_j$).
Eigenvalues $\pm 1 \in SO_2^\pm (\F_q)$ give factors depending
on the specific choice of $G^*$ and $G$.

If we consider $G^* (\F_q)= \text{Sp}_{2r} (\F_q)$ (to describe representations of
$G(\F_q)= SO_{2r+1} (\F_q)$), the centralizer
of a semisimple $s\in G$ is always connected and of the form
\beg{ZSp2rFqs}{Z_{\text{Sp}_{2r} (\F_q)} (s) = \prod_{i=1}^k U^\pm_{j_i} (\F_{q^{n_i}}) \times \text{Sp}_{2p} (\F_q)
\times \text{Sp}_{2\ell} (\F_q)}
where $s$ has $1$ as an eigenvalue of multiplicity $2p$ and $-1$
as an eigenvalue of multiplicity $2\ell$.

If $G^*= SO_{2r+1} (\F_q)$ (to describe representations of
$G= \text{Sp}_{2r} (\F_q)$), the identity component of the
cenralizer of a semisimple element $s \in G$ is of the form
\beg{ZSO2r+1Fqs}{Z_{SO_{2r+1} (\F_q)} (s)^\circ = \prod_{i=1}^k U^\pm_{j_i} (\F_{q^{n_i}}) \times SO_{2p+1} (\F_q)
\times SO_{2\ell}^\pm (\F_q)}
if $s$ has $1$ as an eigenvalue of multiplicity $2p+1$ and $-1$ as an eigenvalue of multiplicity
$2\ell$ (recall that a single $1$ eigenvalue is automatic in this case, since it must
be added to embed $T \subset SO_{2r+1} (\F_q)$).
The full centralizer is obtained by replacing $\text{SO}_{2\ell}^\pm (\F_q)$ by
$\text{O}_{2\ell}^\pm (\F_q)$, meaning that the quotient of the centralizer
of $s$ by its identity component is $\mu_2$ precisely when $\ell>0$ and it trivial when
$\ell = 0$.
The sign of the final factor $SO_{2\ell}^\pm (\F_q)$ corresponds to whether
we pick $T\subset SO_{2r}^+ (\F_q)$ or $SO_{2r}^- (\F_q)$.
We specifically will later consider the semisimple conjugacy classes $(\sigma_{r}^+)$ and
$(\sigma_r^-)$ which both have $-1$ as an eigenvalue of multiplicity $2r$ and a
single (automatic) $1$ eigenvalue, placed so that $\sigma_r^\pm$ is in a torus of the form
\rref{TorusMadeOfSO2s} such that the product of all involved signs of $SO_2$ is equal to $\pm$.
These elements satisfy $Z_{\text{SO}_{2r+1} (\F_q)} (\sigma_r^\pm) = \text{O}_{2r}^\pm (\F_q)$, giving
$$Z_{\text{SO}_{2r+1} (\F_q)} (\sigma_r^\pm )^\circ = \text{SO}_{2r}^\pm (\F_q).$$

If $G^* = SO_{2r}^\pm (\F_q)$ (to describe representations of
$G= SO_{2r}^\pm (\F_q)$), the identity component
of the centralizer of a semisimple element $s\in G$ is of the form
\beg{ZSO2rpmFqs}{Z_{SO_{2r}^\pm (\F_q)} (s)^\circ = \prod_{i=1}^k U^\pm_{j_i} (\F_{q^{n_i}}) \times SO_{2p}^\pm (\F_q)
\times SO_{2\ell}^\pm (\F_q)}
with signs chosen so that their product is the total sign of $G$,
where $s$ has $1$ as an eigenvalue of multiplicity $2p$ and $-1$
as an eigenvalue of multiplicity $2\ell$. For more details, see \cite{Carter}.

\subsection{The unipotent data}
Given a choice of $(s) \in G^*$, the next part of the classification data is a
unipotent representation $u$ of the dual of the 
identity component of the centralizer of $s$
$(Z_{G^*} (s)^\circ )^*$, which by \rref{ZSp2rFqs}, \rref{ZSO2r+1Fqs}, \rref{ZSO2rpmFqs},
we can write down as a product of unitary groups and a pair of factors of
$B$, $C$, $D$ or ${}^2D$-type. We may consider $u$ as a tensor product
of irreducible unipotent representations of each factor.
It will turn out (recalling the constructions
in \cite{TotalHoweII}) that only the tensor factor of $u$ corresponding
to one of these final two factors will be ``altered" in the description of $\eta$ or $\zeta$.
Therefore, we describe in this subsection the classification of unipotent representations of
$\text{Sp}_{2r} (\F_q)$, $SO_{2r+1} (\F_q)$, and $SO_{2r}^\pm (\F_q)$, using symbols.

The classification of irreducible unipotent representations of $\text{Sp}_{2r} (\F_q)$ and $SO_{2r+1}( \F_q)$
are the same, since they are dual groups. A {\em symbol of $B$- or $C$-type}
and rank $r$ is defined to be a pair of increasing sequences
$${\lambda_1< \dots < \lambda_a \choose \mu_1< \dots < \mu_b}$$
for $\lambda_i, \mu_i \in \Z_{\geq 0}$ such that $(\lambda_1, \mu_1) \neq (0,0)$,
$a-b$ is odd (the ``defect condition"), and
$$\sum_{i=1}^a \lambda_i + \sum_{i=1}^b \mu_i = r+ \frac{(a+b-1)^2}{4}$$
(the ``rank condition"). We take switching rows to give the same symbol.
The irreducible unipotents are in bijective correspondence with this
combinatorial data (and we denote the representation corresponding to a symbol by the
symbol itself).

Similarly, for the case of $SO_{2r}^+ (\F_q)$ (resp. $SO_{2r}^- (\F_q)$), irreducible
unipotent representations correspond to symbols of $D$- (resp. $^2 D$-type) and rank $r$,
which are defined to consist of pairs of increasing sequences ${\lambda_1< \dots < \lambda_a
\choose \mu_1< \dots < \mu_b}$
for $\lambda_i, \mu_i \in \Z_{\geq 0 }$ such that $(\lambda_1, \mu_1) \neq (0,0)$,
the ``defect condition" $a-b \equiv 0$ mod $4$ (resp. $\equiv 2$ mod $4$, for $SO_{2r}^- (\F_q)$)
and the ``rank condition"
$$\sum_{i=1}^a \lambda_i + \sum_{i=1}^b \mu_i = r + \frac{ (a+b) (a+b-2)}{4}$$
(the same rank condition is used for $SO_{2r}^- (\F_q)$). Again, we denote
a unipotent representation the same as its corresponding symbol.

Further, the dimension of a unipotent representation ${\lambda_1< \dots < \lambda_a
\choose \mu_1< \dots < \mu_b}$ of a symplectic
or special orthogonal group $G$ can be calculated as the following formula
$$\frac{\displaystyle\prod_{1\leq i< j \leq a} (q^{\lambda_j} - q^{\lambda_i}) \cdot \prod_{1\leq i< j \leq b} (q^{\mu_j} - q^{\mu_i}) \cdot \prod_{1\leq i \leq a, 1\leq j \leq b} (q^{\lambda_i} + q^{\mu_j}) \cdot
|G|_{q'}
}{\displaystyle 2^{c(a,b)} \cdot \prod_{1 \leq i \leq a} \prod_{j=1}^{\lambda_i}  (q^{2j}-1)
\cdot \prod_{1 \leq i \leq b} \prod_{j=1}^{\mu_i}  (q^{2j}-1) \cdot q^{d(a,b)}}$$
where $c(a,b)=\lfloor (a+b-1)/2\rfloor$, and we write
$d (a, b) = \sum_{i=1}^{\lfloor (a+b)/2 \rfloor}  {a+b-2i \choose 2}$.

\subsection{Central and extension sign data}

As described in the beginning of the section, since the
center of $\text{Sp}_{2r} (\F_q)$ is $\mu_2$, 
the representations are indexed by choices of 
semisimple $(s) \in \text{SO}_{2r+1} (\F_q)$, possible central sign
data consisting of an element $\pm 1$ of the quotient
of $s$'s centralizer by its identity component if it is non-trivial,
and unipotent representation $u$ of (the dual of) 
$Z_{\text{SO}_{2r+1} (\F_q)} (s)^\circ$. 

Next, consider the odd orthogonal groups $\text{O}_{2r+1} (\F_q)$. 
In this case, 
the center splits $\text{O}_{2r+1} (\F_q) = \mu_2 \times \text{SO}_{2r+1} (\F_q)$.
For $\text{SO}_{2r+1} (\F_q)$ the center is trivial, so the classification data
of a semisimple conjugacy class $(s) \in \text{Sp}_{2r} (\F_q)$ and an irreducible
unipotent representation of the 
connected centralizer $Z_{\text{Sp}_{2r} (\F_q)} (s)$ determines
all the irreducible representations $r^{\text{SO}_{2r+1} (\F_q)} [(s),u]$.
The induction of such a representation always splits
$$\begin{array}{c}
\text{Ind}_{\text{SO}_{2r+1} (\F_q)}^{\text{O}_{2r+1} (\F_q)} (r^{\text{SO}_{2r+1} (\F_q)} [(s),u])=\\[1ex]
(+1) \otimes r^{\text{SO}_{2r+1} (\F_q)} [(s),u] \oplus (-1) \otimes r^{\text{SO}_{2r+1} (\F_q)} [(s),u].
\end{array}$$
where $(\pm 1)$ denote the trivial or sign representation of the $\mu_2$ central factor of
$\text{O}_{2r+1} (\F_q)$. Call this sign the $\text{O}_{2m+1} (\F_q)$-extension data.
Therefore, writing
$$r^{\text{O}_{2r+1} (\F_q)} [(s),u]^{\pm 1} := (\pm 1) \otimes r^{\text{SO}_{2r+1} (\F_q)} [(s),u],$$
we see that the irreducible representations of $\text{O}_{2r+1} (\F_q)$ are classified by
the extended $\text{O}_{2r+1} (\F_q)$-classification data of a semisimple
conjugacy class $(s) \in \text{Sp}_{2r} (\F_q)$, an irreducible unipotent representation $u$
of the dual of the 
identity component of its centralizer, and extension sign data $(\pm 1)$.

Let us now consider $O_{2r}^\pm (\F_q)$-representations.
Suppose $\pi \in \widehat{O_{2r}^\pm (\F_q)}$. Then $\rho$ 
can be enumerated according to following two effects. Suppose the representation $\pi$ appears as the summand of an induction 
$\text{Ind}_{\text{SO}_{2r}^\pm
(\F_q)}^{\text{O}_{2r}^\pm (\F_q)} (r^{\text{SO}_{2r}^\pm (\F_q)} [(s),u])$.
\begin{enumerate}
\item 
If $s$ has eigenvalues not equal to $\pm 1$, there will be other choices of semisimple data
$(s')\in \text{SO}_{2r}^\pm (\F_q)$ such that
$$\text{Ind}_{\text{SO}_{2r}^\pm (\F_q)}^{\text{O}_{2r}^\pm (\F_q)} ( r^{\text{SO}_{2r}^\pm (\F_q)} [(s), u])
\cong \text{Ind}_{\text{SO}_{2r}^\pm (\F_q)}^{\text{O}_{2r}^\pm (\F_q)} ( r^{\text{SO}_{2r}^\pm (\F_q)} [(s'), u]).
$$
These $s'$ are precisely those where $s$ and $s'$ are not conjugate in $\text{SO}_{2r}^\pm (\F_q)$
but are conjugate in $\text{O}_{2r}^\pm (\F_q)$ due to the larger Weyl group. To avoid overcounting,
we consider semisimple data describing $\pi$ to be the $\text{O}_{2r}^\pm (\F_q)$
conjugacy class of an semisimple element $s\in \text{SO}_{2r}^\pm (\F_q)$.


%
\item The induction of an irreducible unipotent $\text{SO}_{2r}^\pm (\F_q)$-representation
corresponding to a symbol ${\lambda_1< \dots <\lambda_a\choose \mu_1< \dots < \mu_b}$
to $\text{O}_{2r}^\pm (\F_q)$ splits into two non-isomorphic equidimensional pieces
${\lambda_1< \dots <\lambda_a\choose \mu_1< \dots < \mu_b}^+ \oplus {\lambda_1< \dots <\lambda_a\choose \mu_1< \dots < \mu_b}^-$.
(In the split case of $\text{SO}_{2r}^+ (\F_q)$, for both irreducible summands
of corresponding to a degenerate symbol ${\lambda_1< \dots < \lambda_a \choose
\lambda_1< \dots < \lambda_a}$, their inductions to the split orthogonal group
are irreducible and isomorphic.)
\end{enumerate}

In summary, then, 
an irreducible representation of $O_{2r}^\pm (\F_q)$ corresponds to the data of
the $\text{O}_{2r}^\pm (\F_q)$-conjugacy class
$(s)$ of a semisimple element in $\text{SO}_{2r}^\pm (\F_q)$,
the unipotent representation $u$ of $(Z_{\text{SO}_{2r}^\pm} (s)^\circ)^*$,
and the choice of $\text{O}_{2r}^\pm (\F_q)$-extension data consiting of 
a choice of sign $\pm 1$ if $s$ has $1$ eigenvalues and the corresponding
factor of $u$ is non-degenerate and an
(indepedently chosen) sign $\pm 1$ if $s$ has $-1$ eigenvalues and the corresponding factor of $u$
is non-degenerate. 
Call this the classification data of an irreducible representation of $O_{2r}^\pm (\F_q)$.
Denote the corresponding irreducible representation by
$$r^{\text{O}_{2r}^\pm (\F_q)}[(s), u, (\pm 1, \pm1)],$$
writing in $(\pm 1, \pm 1)$ the sign arising from $+1$ eigenvalues of $s$ first
and the sign arising from $-1$ eigenvalues of $s$ second, and removing either
if $s$ has no such eigenvalues.

\section{The general statements}\label{StatementSection}

Finally, in this section we precisely state the constructions of the 
extended eta and zeta correspondences,
define the alternating sums of parabolic inductions, and state Theorem \ref{IntroductionTheorem}.

In Subsection \ref{MetastableSubSect}, we precisely define the symplectic and orthogonal
metastable ranges.
In Subsections \ref{ExtendedEtaSubsect}
and \ref{ExtendedZetaSubsect}, we describe the constructions of the extended eta and zeta correspondences
in terms of classification data. The constructions are entirely the same as the constructions
of the eta and zeta correspondences in the symplectic and orthogonal stable ranges,
when they can be applied. For input represenations where the construction cannot be applied
(specifically, where the step of concatenating a new coordinate
to a symbol factor in the unipotent part of the classification fails
to give a legal symbol), we set the extended eta
and zeta correspondences to output $0$.
In Subsection \ref{AlternatingSumSubsect}, we define the alternating sums of Harish-Chandra
induced modules that play the role of coefficients for the eta and zeta correspondence terms
in the decomposition of the restricted oscillator representation in Theorem
\ref{IntroductionTheorem}.
Finally, having prepared all the necessary notation, we restate Theorem \ref{IntroductionTheorem} precisely in
Subsection \ref{ExplicitTheoremSubsect} as
Theorems \ref{ExplicitTheoremSymp} and \ref{ExplicitTheoremOrtho}.

\subsection{The metastable ranges}\label{MetastableSubSect}
First, for a general unstable choice of symplectic and orthogonal spaces $V$ and $(W,B)$,
we must still choose whether it is ``closer" to the symplectic or orthogonal stable range.

We separate the pairs $(V, (W,B))$ which do not lie in the symplectic stable range or
the orthogonal stable range
into ``metastable ranges" to indicate whether we intend
to approach the decomposition of the restriction of $\omega [ V\otimes W]$
by extending the eta or zeta correspondence.
We consider the different cases of $O(W,B)$ individually.

\begin{definition}\label{MetaStablDefn}
Consider a choice of symplectic and orthogonal spaces $V$ and $(W,B)$. Write $dim (V) = 2N$.

\begin{itemize}

\item If $W$ is of odd dimension $dim (W) = 2m+1$, then we say $(V, (W,B))$ is in the {\em
symplectic metastable range} if
$$m< N<2m+1.$$
Say $(V,(W,B))$ is in the {\em orthogonal metastable range} if
$$m<2N \leq 2m$$

\vspace{1mm}

\item If $W$ is of even dimension $dim (W) = 2m$ and $B$ is not completely split, then we say
$(V,(W,B))$ is in the {\em symplectic metastable range} if
$$m \leq N < 2m.$$
Say $(V,(W,B))$ is in the {\em orthogonal metastable range} if
$$m-1<2N < 2m.$$

\vspace{1mm}

\item If $W$ is of even dimension $dim (W) = 2m$ and $B$ is completely split, then we say
$(V,(W,B))$ is in the {\em symplectic metastable range} if
$$m \leq N < 2m.$$
Say $(V,(W,B))$ is in the {\em orthogonal metastable range} if
$$m<2N < 2m.$$

\end{itemize}
\end{definition}

We see that under this definition, 
every unstable choice of symplectic and orthogonal
spaces $V$ and $(W,B)$ is contained in precisely one metastable range.
More specifically, the disjoint union of the symplectic stable and metastable ranges consists
of all choices of symplectic spaces $V$ and orthogonal spaces $(W,B)$ such that
\beg{EtaCorrStMetaStRange}{\frac{dim (V)}{2} \geq \lfloor \frac{dim (W)}{2} \rfloor ,}
while the disjoint union of the orthogonal stable and metastable ranges consist
of $V$ and $(W,B)$ satisfying the complimentary condition
\beg{ZetaCorrStMetaStRange}{\frac{dim (V)}{2} < \lfloor\frac{dim(W)}{2} \rfloor .}

Broadly, the conditions \rref{EtaCorrStMetaStRange} and \rref{ZetaCorrStMetaStRange} should be thought
of as detecting whether it is more computationally viable to decompose the oscillator
representation in terms of the eta correspondence (i.e. as a sum of distinct irreducible
$\text{Sp}(V)$-representations with potentially non-irreducible $O(W,B)$-coefficients), or in terms of
the zeta correspondence (i.e. as a sum of distinct irreducible $O(W,B)$-representations
with potentially non-irreducible $\text{Sp}(V)$-coefficients).
Concretely,
in our constructions of the eta and zeta correspondence, the conditions \rref{EtaCorrStMetaStRange}
and \rref{ZetaCorrStMetaStRange} ensure that, in either case, we never
attempt to concatenate a negative coordinate to a Lusztig symbol.
It is possible to further extend the eta and zeta correspondences to all ranges
by interpreting them to output
$0$ when this occurs (Lusztig's dimension formula for symbols indeed
suggests that symbols with negative
coordinates are $0$-dimensional). However, approaching the decomposition
of the oscillator representation from the ``wrong side"
is in general less computationally efficient, and we do not consider it for the purposes
of this paper.

\subsection{The extended eta correspondence}\label{ExtendedEtaSubsect}

For a choice of symplectic and orthogonal spaces $V$, $(W,B)$ which are in the disjoint union of
the symplectic stable and metastable ranges, we define a map
$$\eta^V_{W,B} : \widehat{O(W,B)} \rightarrow \widehat{\text{Sp}(V)} \cup \{0\}$$
selecting the ``top" irreducible representation of $\text{Sp}(V)$ whose tensor product with the input
irreducible representation of $O(W,B)$ is a summand of the restricted oscillator representation
(or outputting $0$ when no such summand appears).
More explicitly, for every $\rho \in \widehat{O(W,B)}$,
$$\eta^V_{W,B} (\rho)\otimes \rho \subset Res_{\text{Sp} (V) \times O(W,B)} ( \omega
[V \otimes W])$$
and there exists some $\pi \in \widehat{\text{Sp} (V)}$ with $\pi \otimes \rho$ appearing as a
summand of the restricted oscillator
representation if and only if $\eta^V_{W,B} (\rho) \neq 0$. 

With this forumlation, we may treat $\eta^V_{W,B}$ for $(V, (W,B))$ in the symplectic
stable and metastable range at the same time.
In the case of symplectic stable choices of $V$ and $(W,B)$, this is precisely a review
of the construction of the eta correspondence given in \cite{TotalHoweII}.
To define the
map $\eta^V_{W,B}$ we treat the different cases of $(W,B)$ separately.
\vspace{3mm}

\noindent{\bf Case 1:} Suppose $W$ is of odd dimension $2m+1$,
and suppose $(V, (W,B))$ is either in the symplectic stable or metastable range. In this
case, writing $dim (V) = 2N$, this means $m < N$.

For an irreducible representation $\pi \in \widehat{O(W,B)}$, we can then write it as
\beg{SplitOfZ2SympInputOWBRep}{\pi \cong r^{\text{O}_{2m+1} (\F_q)}[(s), u]^{\pm 1}}
where $(s)$ is a conjugacy class of a semisimple element
$s \in \text{Sp}_{2m} (\F_q) = (SO_{2m+1} (\F_q))^*$, and $u$ is an irreducible unipotent
representation of $(Z_{\text{Sp}_{2m} (\F_q)} (s)^\circ)^*$. Then to define
$\eta^V_{W,B} (\pi)$,
we must either specify the classification data for $\text{Sp}_{2N} (\F_q) = \text{Sp}(V)$,
or put it to be $0$.

The first part of the classification data for $\text{Sp}_{2N} (\F_q)$ is a conjugacy class of a semisimple
element in $(\text{Sp}_{2N} (\F_q))^* = \text{SO}_{2N+1} (\F_q)$. 
Let us consider the semisimple part $s$ of the input representation's classification data
as an element of a maximal torus
\beg{sInT}{s \in T = \prod_{i=1}^k SO_{2}^\pm (\mathbb{F}_{q^{n_i}})\subseteq (SO_{2m+1} (\F_q))^* = \text{Sp}_{2m} (\F_q)}
(such that $n_1 + \dots + n_k = m$). Identifying each
$SO_2^\pm (\F_{q^{n_i}})$ with the cyclic group $\mu_{q^{n_i}\mp 1}$,
define $\epsilon (s)$ to be the product of
the quadratic character $\epsilon: \mu_{q^{n_i}\mp 1} \rightarrow \{\pm 1\}$
applied to each coordinate of $s$ in \rref{sInT}.
Taking a product with
$\sigma_{N-m}^\pm \in SO_{2(N-m)+1} (\F_q)$
(recalling that $N>m$ by the range conditions) gives two options of a conjugacy class
of a semisimple element in $SO_{2N+1} (\F_q)$.
Write
$$\phi^\pm (s) = s \oplus \sigma_{N-m}^\pm \in SO_{2N +1} (\F_q).$$
We choose the sign of $\pm $ according to the sign of the $(\pm 1)$ factor
indicating the $\mu_2$-action in
\rref{SplitOfZ2SympInputOWBRep}.
Suppose $s$ has $-1$ as an eigenvalue of multiplicity $2\ell$, and write
the 
identity component of its centralizer as
\beg{s'soriginalcentralizeretaoddinput}{Z_{\text{Sp}_{2m} (\F_q)} (s)^\circ = H \times \text{Sp}_{2\ell} (\F_q).}
Then $\phi^\pm (s)$ has $-1$ as an eigenvalue of multiplicity $2(N-m+\ell)$ and 
the 
identity component of its centralizer
can be factored  
$Z_{2O_{2N+1} (\F_q)} (\phi^\pm (s))^\circ= H^* \times SO_{2(N-m+\ell)}^\pm (\F_q)$.

Next, writing \rref{s'soriginalcentralizeretaoddinput}, and then
$(Z_{\text{Sp}_{2m} (\F_q)} (s)^\circ)^* = H^* \times SO_{2\ell+1} (\F_q)$, we may consider
the unipotent part $u$ of the classification data of the input representation as a tensor product
$ u= u_{H^*} \otimes {\lambda_1< \dots < \lambda_a \choose \mu_1< \dots < \mu_b}$
where, as the notation suggest, $u_{H^*}$ is an irreducible unipotent representation of $H^*$,
and ${\lambda_1< \dots < \lambda_a \choose \mu_1< \dots < \mu_b}$ is a symbol 
specifying a unipotent representation of $SO_{2\ell+1} (\F_q)$. Since the defect $a-b$ of the symbol
is odd, we may therefore permute the rows so that we may assume $a-b$ is $1$ mod $4$.
Let us write
\beg{N'rhoOddEtaCase}{N'_{\rho} := N-m + \frac{a+b-1}{2}.}

Now if $N'_\rho < \lambda_a$, then
${\lambda_1< \dots < \lambda_a < N'_{\rho} \choose \mu_1< \dots < \mu_b}$
describes a symbol of defect $2$ mod $4$ and rank precisely equal to $N-m+\ell$, and hence
taking
$$\phi^- (u) = \widetilde{u_{H^*}} \otimes {\lambda_1< \dots < \lambda_a < N'_{\rho} \choose
\mu_1< \dots < \mu_b}$$
gives a unipotent representation of the dual of the 
identity component of
the centralizer of the new semisimple data
$(Z_{SO_{2N+1} (\F_q)} (\phi^- (s))^\circ)^* = H \times SO_{2(N-m+\ell)}^- (\F_q)$.
In this case, say $\phi^- (u)$ is constructible.
Say $\phi^- (u)$ is {\em inconstructible} if $N'_\rho \leq \lambda_a$.

Similarly, if $N'_\rho < \mu_b$, then
${\lambda_1< \dots < \lambda_a \choose \mu_1< \dots < \mu_b< N'_\rho}$
describes a symbol of defect $0$ mod $4$ and rank precisely equal to $N-m+\ell$, and hence
taking
$$\phi^+ (u) = \widetilde{u^{H^*}} \otimes {\lambda_1< \dots < \lambda_a \choose
\mu_1< \dots < \mu_b < N'_\rho}$$
gives a unipotent representation of the dual of the 
identity component of the centralizer of the new semisimple data
$(Z_{SO_{2N+1} (\F_q)} (\phi^+ (s))^\circ)^* = H \times SO_{2(N-m+\ell)}^+ (\F_q)$.
In this case, say $\phi^+ (u)$ is constructible.
Say $\phi^+ (u)$ is {\em inconstructible} if $N'_\rho \leq \mu_b$.

\begin{definition}\label{EtaOddCaseDefn}
Assume the above notation. We define $\eta^V_{W,B} (\pi )$ to be the irreducible
$\text{Sp}(V)$-representation with classification data $[(\phi^\pm (s)), \phi^\pm (u), \epsilon (s) \cdot disc (B)]$
\beg{EtaSympOddConstructCase}{\eta^V_{W,B} (\pi) := r^{\text{Sp} (V)}[(\phi^\pm (s)), \phi^\pm (u),\epsilon (s) \cdot disc (B)],}
if $\phi^\pm (u)$ is constructible. We put
$\eta^V_{W,B} (\rho) := 0$
if $\phi^\pm (u)$ is inconstrucible.
\end{definition}

\vspace{3mm}

\noindent{\bf Case 2:} Suppose $W$ is of even dimension $2m$, and write $\alpha$ for the sign
so that $O (W,B) = O^\pm_{2m} (\F_q)$. Suppose also that $(V, (W,B))$ is either in the symplectic
stable or metastable range, meaning that if we write $dim (V) = 2N$, we have
$m \leq N$.

Let us consider an input irreducible $\text{O}^\pm_{2m} (\F_q)$-representation
$\pi$. Let us suppose that its 
$\text{O}^\pm_{2m} (\F_q)$-classification data consists of
a $\text{O}^\pm_{2m} (\F_q)$-conjugacy class $(s)$ of a semisimple element of $\text{SO}^\pm_{2m} (\F_q)$, an irreducible unipotent
representation $u$ of, say, $(Z_{\text{SO}^\pm_{2m} (\F_q)} (s)^\circ)^*$, and
possible central sign data depending on which eigenvalues appear in $s$.

As in the previous case, consider $s$ as an element of a torus
$s\in T \cong \prod_{i=1}^k SO_2^{\pm} (\F_{q^{n_i}})$.
We take a direct sum with the identity matrix $I_{2(N-m)+1}$ to obtain a semisimple element
$$\phi (s) = s \oplus I_{2(N-m)+1} \in SO_{2N+1} (\F_q) ,$$
adding only $1$ eigenvalues to $s$.
If $s$ originally does not have any $1$ eigenvalues, then the 
identity component of its centralizer is altered by
taking a product with a special orthogonal group factor
$$Z_{SO_{2N+1} (\F_q)} (\phi (s))^\circ= (Z_{SO_{2m}^\alpha (\F_q)} (s)^\circ)^* \times SO_{2(N-m)+1} (\F_q).$$
We consider the irreducible unipotent representation
$$\phi (u):= \widetilde{u} \otimes 1$$
of the dual of this group, tensoring $\widetilde{u}$ with the trivial representation of
the new factor $(SO_{2(N-m)+1} (\F_q))^*$.
On the other hand, if $s$ has $1$ as an eigenvalue of multiplicity $2p$ for $p>0$, then, writing
$$Z_{SO_{2m}^\alpha (\F_q)} (s)^\circ = H \times SO_{2p}^\pm (\F_q),$$
we then have
$$Z_{SO_{2N+1} (\F_q)} (\phi(s))^\circ = H^* \times SO_{2(N-m+p)+1} (\F_q)$$
(note that in this case $H = H^*$).
Write $(Z_{SO_{2m}^\alpha (\F_q)} (s)^\circ)^* = H^* \times SO_{2p}^\pm (\F_q)$,
and $u = u_{H^*} \otimes {\lambda_1< \dots < \lambda_a \choose \mu_1< \dots < \mu_b}$
for an irreducible unipotent representation $u_{H^*}$ of $H^*$, and a symbol
${\lambda_1< \dots < \lambda_a \choose \mu_1< \dots < \mu_b}$ of $SO_{2p}^\pm (\F_q)$.
Suppose first that ${\lambda_1< \dots < \lambda_a \choose \mu_1< \dots < \mu_b}$ is non-degenerate.
In the case where the $1$ eigenvalues correspond to a non-split factor $SO_{2p}^- (\F_q)$ of
the 
identity component of $s$'s centralizer,
then switch rows so that $a-b+1$ is $1$ mod $4$.
In the case where the $1$ eigenvalues correspond to a split factor $SO_{2p}^+ (\F_q)$ of the 
identity component of
$s$'s centralizer,
then switch rows so that either $a>b$ or, if $a=b$, for the maximal $i$ such that
$\lambda_i \neq \mu_i$, we have $\lambda_i> \mu_i$.
Write
$$N_\rho' = N-m +\frac{a+b}{2}.$$
Therefore, if $\lambda_a < N'_\rho$ or $\mu_b< N'_\rho$,
$$\phi^- (u) = \widetilde{u_{H^*}} \otimes {\lambda_1< \dots < \lambda_a < N'_\rho \choose
\mu_1< \dots < \mu_b}, \text{ or }$$
$$\phi^+ (u) = \widetilde{u_{H^*}} \otimes {\lambda_1< \dots < \lambda_a \choose
\mu_1< \dots < \mu_b < N'_\rho}$$
respectively define irreducible unipotent representations of the dual group
$(Z_{SO_{2N+1} (\F_q)}(\phi (s))^\circ)^* = H
\times \text{Sp}_{2(N-m+p)} (\F_q)$.
Say $\phi^+ (u)$, $\phi^- (u)$ are {\em inconstructible} if $\lambda_a\geq N'_\rho$, or
$\mu_b \geq N'_\rho$, respectively.
We choose the sign of $\phi^\pm (u)$ according to the central sign data $(\pm 1)$ chosen
from $s$ having $1$ eigenvalues.

Finally, to define an output irreducible $\text{Sp}_{2N} (\F_q)$-representation, we 
need to also choose output
central sign data if $\phi (s)$ has $-1$ eigenvalues. By definition, $\phi (s)$ has the same number
of $-1$ eigenvalues as $s$. Therefore, in this case, the original $s$ has $-1$ eigenvalues also,
so the $O(W,B)$-classification data supplies us with the data of one more central sign
$\pm 1$, which we use as the output central sign data.

\begin{definition}\label{EtaEvenCaseDefn}
Assume the above notation. We define $\eta^V_{W,B} (\pi)$ to be the irreducible
$\text{Sp}(V)$-representation with classification data
$[ \phi (s), \phi^\alpha (u), \beta ]$ where the sign $\alpha$ denotes
the central sign data from $s$'s $1$ eigenvalues (and is omitted if no
such data is given), the sign in $\beta$ is the central sign data
from $s$'s $-1$ eigenvalues (and is omitted if no
such data is given):
\beg{EtaDefnSympMetastEven}{\eta^V_{W,B} (\pi) := r^{\text{Sp} (V)}[(\phi (s)), \phi^\alpha (u), \beta],}
if $\phi^\alpha (u)$ is constructible. We put
$\eta^V_{W,B} (\rho) := 0$
if $\phi^\pm (u)$ is inconstructible.
\end{definition}

\subsection{The extended zeta correspondence}\label{ExtendedZetaSubsect}

Similarly as in the previous subsection,
for a choice of symplectic and orthogonal spaces $V$, $(W,B)$ which are in the disjoint union of
the orthogonal stable and metastable ranges, we define a map
$$\zeta_V^{W,B} : \widehat{\text{Sp}(V)} \rightarrow \widehat{\text{O}(W,B)} \cup \{0\}$$
analogously selecting the ``top" irreducible representation of $O(W,B)$ whose tensor
product with the input irreducible representation of $\text{Sp}(V)$ is a summand of the restricted
oscillator representation. This construction proceeds entirely similarly, still considering
the cases of $O(W,B)$ separately:

\vspace{3mm}

\noindent{\bf Case 1:} Suppose $W$ is of odd dimension $2m+1$, and
suppose $(V,(W,B))$ is either in the orthogonal stable or metastable range. In
this case, writing $dim (V) = 2N$, this means $m\geq N$.

Now we fix an irreducible representation $\rho$ of $\text{Sp}(V)=\text{Sp}_{2N} (\F_q)$. Our goal
is to construct an irreducible representation of $O(W,B)$, which 
is equivalent to specifying a
semisimple conjugacy class $(s) \in (\text{Sp}_{2N} (\F_q))^* = SO_{2N+1} (\F_q)$,
an irreducible unipotent representation of the dual of the 
identity component of its centralizer, and an extension sign.
Say that $s$ has $-1$ as an eigenvalue of multiplicity $2\ell$, for $0\leq \ell \leq N$, and
let us write
$$Z_{SO_{2N+1} (\F_q)} (s)^\circ = H \times SO_{2\ell}^\pm (\F_q).$$
Recall, as a semisimple element of $SO_{2N+1} (\F_q)$, $s$ must have at least one 
$1$ eigenvalue, arising from the embedding of any maximal torus $T$ of
the form \rref{TorusMadeOfSO2s} into
$SO_{2N+1} (\F_q)$. Therefore, considering $s$ as an element of the torus,
let us write $\widetilde{s}\in T$ (giving a $2N$ by $2N$ matrix), by removing the single ``forced"
eigenvalue $1$ from $s$. Taking a direct sum with $-I_{2(m-N)}$,
$$\psi (s) := \widetilde{s} \oplus (-I)_{2(m-N)} $$
we obtain a semisimple element of $\text{Sp}_{2m} (\F_q) = (SO_{2m+1} (\F_q))^*$,
which has $-1$ as an eigenvalue of multiplicity $2(m-N+\ell )$. The 
identity component of its centralizer is then
$$Z_{\text{Sp}_{2m}(\F_q)} (\psi (s))^\circ = H^* \times \text{Sp}_{2(m-N+\ell)} (\F_q).$$

For the unipotent part of the classification data of $\rho$,
again write 
$u= u^{H^*} \otimes {\lambda_1< \dots < \lambda_a \choose
\mu_1< \dots < \mu_b}$ for $u^{H^*} \in \widehat{(H^*)}_u$
and a symbol ${\lambda_1< \dots < \lambda_a \choose \mu_1< \dots < \mu_b}$
specificying an irreducible unipotent representation of $SO_{2\ell}^\pm (\F_q)$.
In the case where the $1$ eigenvalues correspond to a non-split factor $SO_{2p}^- (\F_q)$ of 
the 
identity component of $s$'s centralizer,
then switch rows so that $a-b+1$ is $1$ mod $4$.
In the case where the $1$ eigenvalues correspond to a split factor $SO_{2p}^+ (\F_q)$ of
the 
identity component of $s$'s centralizer,
then switch rows so that either $a>b$ or, if $a=b$, for the maximal $i$ such that
$\lambda_i \neq \mu_i$, we have $\lambda_i> \mu_i$.
Let us write
$$m'_\rho = m-N + \frac{a+b}{2}.$$
Then if $\lambda_a< m'_\rho$, $\mu_b< m'_\rho$, respectively, the symbols
\beg{mrho'addsymbol}{{\lambda_1< \dots < \lambda_a < m'_\rho\choose
\mu_1< \dots < \mu_b}, {\lambda_1< \dots < \lambda_a \choose
\mu_1< \dots < \mu_b < m'_\rho}}
have odd defect and rank precisely equal to $m-N+\ell$, and therefore specify
irreducible unipotent representations of $SO_{2(m-N+\ell)+1} (\F_q)
= (\text{Sp}_{2(m-N+\ell)} (\F_q))^*$.

In the case when $\ell = 0$, both of the symbols \rref{mrho'addsymbol} specify the
trivial representation, so let us put
$$\psi (u) = \widetilde{u_{H^*}} \otimes 1.$$
On the other hand, in the case when $\ell>0$, we are given
sign data $\pm 1$ in the 
classification data for $\rho$, which we can use to select which of the symbols \rref{mrho'addsymbol}
we attempt to use for the unipotent part of the classification data of $\zeta^{W,B}_V (\rho)$.
Specifically, if $\lambda_a< m'_\rho$, $\mu_b < m'_\rho$, respectively, we put
$$\begin{array}{c}
\displaystyle \psi^+ (u) = \widetilde{u_{H^*}} \otimes {\lambda_1< \dots < \lambda_a < m'_\rho\choose
\mu_1< \dots < \mu_b}\\
\\
\displaystyle 
\psi^- (u) = \widetilde{u_{H^*}} \otimes {\lambda_1< \dots < \lambda_a \choose
\mu_1< \dots < \mu_b < m'_\rho}.
\end{array}$$
If $\lambda_a \geq m'_\rho$, resp. $\mu_b\geq m'_\rho$, we say
$\psi^+ (u)$, resp. $\psi^- (u)$ is inconstructible.

\begin{definition}
Assume the above notation, writing $\rho = r^{\text{Sp} (V)}[(s),u]$.
In the case where the semisimple part $s$ of $\rho$'s classification data
has no $-1$ eigenvalues i.e. $\ell = 0$, 
we take $\zeta^{W,B}_V (\rho)$ to be the tensor product
of the irreducible $SO(W,B)$-representation corresponding
to classification data $[(\psi (s)), \psi(u)]$
with the sign $\epsilon (s) \cdot disc (B)$:
\beg{ell=0orthoodd}{\zeta^{W,B}_V (\rho) = r^{\text{O} (W,B)} [(\psi (s)), \psi (u)]^{\epsilon (s) \cdot disc (B)} }
In the case where $s$ has $-1$ eigenvalues i.e.,
$\ell > 0$, writing $\rho = r^{\text{Sp} (V)}[(s), u, \pm 1]$,
we take $\zeta^{W,B}_V(\rho)$ to be the tensor product of the irreducible
$SO(W,B)$-representation corresponding to classification data
$[(\psi(s)), \psi^\pm (u)]$ with the sign $\epsilon (s) \cdot disc(B)$
\beg{ell>0orthoodd}{\zeta^{W,B}_V (\rho) =  r^{\text{O} (W,B)} [(\psi (s)), \psi^\pm (u)]^{\epsilon (s) \cdot disc (B)},}
if the involved $\psi^\pm (u)$ is constructible.
Put
$\zeta^{W,B}_V (\rho) = 0$ if $\psi^\pm (u)$ is inconstructible.
\end{definition}

\vspace{3mm}

\noindent{\bf Case 2:} Suppose $W$ is of even dimension $2m$,
and write $\alpha$ for the sign so that $O(W,B) =O^\pm_{2m} (\F_q)$.
Suppose also that $(V, (W,B))$ is either in the orthogonal stable or metastable range,
meaning that

Consider an irreducible representation $\rho$ of $\text{Sp}(V) = \text{Sp}_{2N} (\F_q)$.
We want to produce $O_{2m}^\pm (\F_q)$-classification data, which we recall
consists of a semisimple conjugacy class $(\psi (s)) \in O_{2m}^\pm (\F_q)$, a unipotent representation
$\psi (u)$ of the (dual) of the 
identity component of
the centralizer of $s$ in $SO_{2m}^\alpha (\F_q)$, and central sign data depending on the eigenvalues
of $\psi (s)$.
We note that since $Res_{O(W,B)} (\omega[ V\otimes W])$ is the
permutation representation $\C W$ tensored with the representation $\epsilon (det)$
(corresponding to the sign representation of $O(W,B)/SO(W,B)$),
part of the central sign data is already
forced. Specifically, as in the case of the symplectic group, we will only need to choose
central sign data for the output representation corresponding to $-1$-eigenvalues.

Write $(s)$ with $s \in SO_{2N+1} (\F_q) = (\text{Sp}_{2N} (\F_q))^*$
for the semisimple part of the classification data for the input 
representation $\rho$ of $\text{Sp}_{2N} (\F_q)$.
Say $s$ has $1$ as an eigenvalue of multiplcity $2p+1$, and write
$Z_{SO_{2N+1} (\F_q)} (s)^\circ = H \times SO_{2p+1} (\F_q)$.
Again, we may remove a single ``forced" $1$ eigenvalue from $s$
to view it as a $2N$ by $2N$
element of the maximal torus $\widetilde{s} \in T$.
Then consider the direct sum with the $2(m-N)$ by $2(m-N)$ identity matrix
$$\psi (s) = \widetilde{s} \oplus I_{2(m-N)},$$
configured to give
a $2m$ by $2m$ matrix that can be considered as an element of $SO(W,B) \subseteq O(W,B)$.
As in Case 2 of the construction of the eta correspondence, each distinct
$SO_{2N+1}(\F_q)$-conjugacy class
$(s)$ gives a distinct $O_{2m}^\alpha (\F_q)$-conjugacy class $\psi (s)$.
We have
\beg{Beta}{Z_{SO_{2m}^\pm (\F_q)} (\psi (s))^\circ = H^* \times SO_{2(N-m+p)}^\sigma (\F_q),}
for a single determined choice of sign $\sigma$
(so that its product with the other signs appearing in $H$ agrees
with the fixed sign $\pm$ of the orthogonal group $\text{O} (W,B)$).

For the unipotent part of the $O_{2m}^\alpha (\F_q)$-classification data we want to produce,
write the unipotent representation $u$ in $\rho$'s classification
data as $u = u_{H^*} \otimes  {\lambda_1 < \dots < \lambda_a \choose
\mu_1< \dots < \mu_b}$ for $u_{H^*} \in \widehat{(H^*)}_u$ and
a symbol ${\lambda_1< \dots < \lambda_a \choose \mu_1< \dots < \mu_b}$
specifying an irreducible unipotent representation of $SO_{2p+1} (\F_q)$. 
Switch rows so that the defect $a-b$ is $1$ mod $4$ (which is
possible since this symbol has odd defect). Let us write
$$m'_\rho = m-N +\frac{a+b-1}{2}.$$
Then, if $\sigma = +$, if $\mu_b < m'_\rho$, putting
$$\psi^+ (u) = \widetilde{u_{H^*}} \otimes {\lambda_1< \dots < \lambda_a \choose
\mu_1< \dots < \mu_b < m'_\rho}$$
gives a unipotent representation of the group $H \times SO_{2(N-m+p)}^+ (\F_q) = 
(Z_{SO_{2m}^\alpha (\F_q)} (\psi (s))^\circ)^*$.
In this case, we put a sign
$\alpha (\sigma)$ to be $+1$ unless $a=b+1$, $m_\rho' \leq \lambda_a$ and
for the maximal $i$ such that $\lambda_{i+1} \neq \mu_i$, we have $\mu_i> \lambda_{i+1}$,
in which case we put $\alpha (\sigma) = -1$.
If $\sigma = +$, $\mu_b \geq m'_\rho$, then say $\psi (u)$ is inconstructible.

Similarly, if $\sigma = -$, if $\lambda_a < m'_\rho$, putting
$$\psi^- (u) = \widetilde{u_{H^*}} \otimes {\lambda_1< \dots < \lambda_a < \mu_\rho'\choose
\mu_1< \dots < \mu_b }$$
gives a unipotent representation of the group $H \times SO_{2(N-m+p)}^- (\F_q) = 
(Z_{SO_{2m}^\alpha (\F_q)} (\psi (s))^\circ)^*$.
In this case, we put a sign
$\alpha (\sigma)$ to always be $+1$.
If $\sigma = -$, $\lambda_a \geq m'_\rho$, then say $\psi (u)$ is inconstructible.

Now, as in Case 2 in the symplectic case, $(s)$ and $(\psi (s))$ have the same multiplicity of
$-1$ eigenvalues. When it does, we take the sign
of the new $\text{O}_{2m}^\pm (\F_q)$-extension
data corresponding to the $-1$ eigenvalues of $\psi (s)$
to be the same as the input representation $\rho$'s corresponding
$\text{Sp}_{2N} (\F_q)$-central sign data.

\begin{definition}
Suppose we are given the above notation.
If $s$ has no $-1$ eigenvalues, define $\zeta^{W,B}_V (\rho)$ to be the irreducible $O(W,B)$-representation
with $O(W,B)$-classification data $[(\psi (s)), \psi (u)]$ with the extension data
$(\alpha^\sigma)$ corresponding to the $1$ eigenvalues of $\psi (s)$
\beg{ZetaDefnOrthoEven}{\zeta^{W,B}_V (\rho) := r^{\text{O}(W,B)}[(\psi (s)), \psi^\sigma (u)]^{(\alpha ( \sigma))},}
for $\sigma$ denoting the sign determined by \rref{Beta}.
If $s$ has $-1$ eigenvalues, writing $\pm 1$ for the central $\text{Sp} (V)$-data of the input representation
$\rho$
\beg{SignZetaDefnOrthoEven}{\zeta^{W,B}_V (\rho) := r^{\text{O}(W,B)}[(\psi (s)), \psi^\sigma (u)]^{(\alpha ( \sigma), \pm 1)},}
again for $\sigma$ denoting the sign determined by \rref{Beta}.
If $\psi (u)$ is inconstructible, we put
$\zeta^{W,B}_{V} (\rho) = 0$.

\end{definition}

Recalling the results of \cite{TotalHoweII},
we found that in the symplectic (resp. orthogonal)
stable range, every irreducible representation $\rho$ of $O(W,B)$ (resp. $\text{Sp}(V)$)
appears in the restriction
$$Res_{O(W,B)} (Res_{\text{Sp}(V) \times O(W,B)} ( \omega [V\otimes W]))$$
(resp. $Res_{\text{Sp}(V)} (Res_{\text{Sp}(V) \times O(W,B)} ( \omega [V\otimes W]))$).
In the present paper's notation,
for $(V,(W,B))$ in the symplectic stable range, for every $\rho \in \widehat{O(W,B)}$,
$$\eta^V_{W,B} (\rho) \neq 0.$$
Similarly, for $(V,(W,B))$ in the orthogonal stable range, for every $\rho \in \widehat{\text{Sp}(V)}$,
$$\zeta_V^{W,B} (\rho) \neq 0.$$

Even in the metastable range, however, the range conditions ensure that we never
need to add a negative coordinate to a Lusztig symbol:
For $(V,(W,B))$ in the symplectic (stable or) metastable range, for every $\rho \in \widehat{O(W,B)}$,
we have
$$N_\rho' >0.$$
For $(V,(W,B))$ in the orthogonal (stable or) metastable range, for every $\rho \in \widehat{\text{Sp}(V)}$,
we have
$$m_\rho' >0.$$

\subsection{Alternating sums of parabolic inductions}\label{AlternatingSumSubsect}
Now we define the alternating sums of parabolic inductions (only needed in the metastable ranges)
which form the coefficients of
the extended eta or zeta correspondences in the restricted oscillator representation.
The description given in this subsection is somewhat technical, 
but in each case,
the principle is to preserve all of the classification data, except for the
symbol factor of the unipotent part that is altered in the construction of the extended eta
or zeta correspondence. We take the sum of the representations
obtained by replacing that symbol by those appearing in the
alternating sum of parabolic inductions of the symbols
obtained by conctenating a final coordinate
to one of the rows as in the construction of $\eta^V_{W,B}$ or $\zeta^{W,B}_V$,
and removing another coordinate in the same row to recover the original symbol's row lengths
(see \rref{TopSymbolCycle} and \rref{BottomSymbolCycle} below).

Suppose we want to consider the decomposition of the restricted oscillator representation
$Res_{\text{Sp}(V) \times O(W,B)} (\omega [ V\otimes W])$. For this subsection,
we now fix our choice of $(V, (W,B))$,
specifically fixing a case of the parity of $dim (W)$ and whether
we consider the decomposition in terms of
an extended eta correspondence (i.e. condition \rref{EtaCorrStMetaStRange} is satisfied)
or an extended zeta correspondence (i.e. condition
\rref{ZetaCorrStMetaStRange} is satisfied).

\vspace{3mm}

Now consider a symbol
\beg{GeneralSymbol}{\theta = {\lambda_1< \dots < \lambda_a \choose \mu_1< \dots < \mu_b}}
of a group $K[-k]$, where $K$ denotes a possible factor of 
the 
identity component of a centralizer of a semisimple element 
$s$ in the dual of the domain of whichever correspondence we have chosen to consider
($O(W,B)$ in the extended eta correspondence or $\text{Sp}(V)$ for the extended
zeta correspondence), which would be altered in the fixed correspondence's construction
(corresponding to $-1$ eigenvalues if $dim (W)$ is fixed to be odd,
and $1$ eigenvalues if $dim (W)$ is fixed to be even).
We use the notation $[-k]$ to refer to the group of 
the same type and subtracting $k$ from the rank (e.g.
for $K= SO(W,B)$, we write $K[-k] = SO(W[-k], B[-k])$).
Let us switch rows in \rref{GeneralSymbol} so that in the construction of
the extended eta or zeta correspondence, if $\theta$ appears as the factor of a unipotent
part $u$ of the classification data for an input representation,
$\phi^+$ or $\psi^+$ concatenates a new coordinate to the top row $\lambda_1< \dots < \lambda_a$
of \rref{GeneralSymbol}  
(this corresponds to a condition on the row lengths $a$, $b$,
which varies depending on the case of the extended eta or zeta correspondence we consider).
We do this in order to consolidate the notation and treat every case at once.

\vspace{2mm}

Fix $k$ and fix some $N' \in \N_0$ such that
\beg{lambdasalpha}{\lambda_1< \lambda_2< \dots < \lambda_a < N' .}
Consider the $1 \leq c \leq a$ such that
$$N' - \lambda_c \leq k < N' - \lambda_{c+1}.$$
Then, for $c \leq q \leq a+1$, let us consider the symbol $\theta^+ (N')_q$
with top row given by \rref{lambdasalpha}
with the $q$th coordinate removed, and unmodified bottom row. For $c\leq q \leq a$ this gives
\beg{TopSymbolCycle}{\theta^+ (N')_q:= {\lambda_1< \dots < \widehat{\lambda_q} < \dots < \lambda_a < N'
\choose
\mu_1 < \dots < \mu_b}.}
In the case of $q=a+1$, we the $(a+1)$th coordinate of \rref{lambdasalpha}
is $N'$ and thus we have $\theta (N' )_q = \theta$.
Each of these symbol has the same row lengths as $\theta$,
and $\theta (N')_q$ therefore describes a unipotent representation of $K[-(k-(N'-\lambda_q))]$.
We may hence consider the alternating sum
\beg{AlternateSymbolInd}{\begin{array}{c}
A_k^+ (\theta, N') :=\\[1ex]  
\displaystyle \bigoplus_{q=c}^{a+1} (-1)^{a+1-q} \cdot \text{Ind}_{P_{k-(N' - \lambda_q)}} (\theta^+ (N')_q)
\end{array}}
where here $P_j$ denote the maximal parabolics in the full group $K$ with Levi factors
$K[-j] \times GL_j (\F_q)$, and we take trivial $GL_j (\F_q)$-action in each induction term.

Similarly, for $N' \in \N_0$ such that $\mu_1< \dots <\mu_b < N'$, considering $1 \leq c \leq b$ such that
$$N' - \mu_c \leq k < N' - \mu_{c+1},$$
for $c \leq q \leq b$, we consider the symbol
\beg{BottomSymbolCycle}{\theta^- (N')_q:= {\lambda_1< \dots < \lambda_a
\choose \mu_1< \dots < \widehat{\mu_q} < \dots < \mu_b < N'}.}
We also put $\theta^- (N')_{b+1} = \theta$. We can then define the alternating sum
\beg{AlternateSymbolIndBottom}{\begin{array}{c}
A_k^- (\theta, N') :=\\[1ex]  
\displaystyle \bigoplus_{q=c}^{b+1} (-1)^{b+1-q} \cdot \text{Ind}_{P_{k-(N' - \mu_q)}} (\theta^- (N')_q).
\end{array}}

We find that \rref{AlternateSymbolInd} and
\rref{AlternateSymbolIndBottom}, in every case we consider, define genuine representations.
In fact, we give a concrete description of which symbols appear in their decompositions
in the Appendix.

Again, we approach the case of the extended eta correspondence first.

\begin{definition}\label{SympCaseAltSumsDefn}
Consider a choice of $V$ and $(W,B)$ in the symplectic (stable or) metastable range.
Suppose we are given $0 \leq k \leq h_W$ and $m'> 0$.
Consider an irreducible representation $\rho \in \widehat{O(W[-k],B[-k])}$. 
\begin{enumerate}
\item Suppose $dim (W) = 2m+1$ is odd. Write the
extended classification data of an irreducible $\text{O} (W[-k], B[-k])$-representation $\rho$ as
$\rho = r^{\text{O}(W[-k],B[-k])}[(s), u]^\alpha$,
and, writing $2\ell$ for the multiplicity of $-1$ as an eigenvalue of $s$, write
$Z_{\text{Sp}_{2(m-k)} (\F_q)} (s)^\circ = H \times \text{Sp}_{2\ell} (\F_q)$, and $u = u_{H^*} \otimes \theta$.
As in the construction of the eta correspondence, we can interpret $s \oplus (-I_{2k})$
as a semisimple element of $\text{Sp}_{2m} (\F_q)$ with the 
identity component of its centralizer of the form $H \times \text{Sp}_{2(\ell + k)} (\F_q)$.
Then define $\mathcal{A}_k (\rho, N')$ to be the $O(W,B)$-representation 
$$\mathcal{A}_k (\rho, N') = \bigoplus_{\chi \subset A_k^\pm (\theta, N')}
r^{\text{O}(W,B)}[(s \oplus (-I_{2k}) ), u_{H^*} \otimes \chi]^\alpha,$$
with the sum running over every distinct irreducible unipotent $\theta$ appearing in $A_k^\pm (\theta, N')$
where the superscript sign is chosen to agree with the sign of $\phi^\pm (u)$ we take when
construction $\eta^V_{W,B} (\rho)$.

\item Suppose $dim (W) = 2m$ is even. Write the
extended classification data of an irreducible $\text{O} (W[-k], B[-k])$-representation $\rho$ as
$\rho = r^{SO(W[-k],B[-k])}[(s), u]^\gamma$ (writing $\gamma$ for the extension sign data
depending on the eigenvalues of $s$).
Say $s$ has $1$ as an eigenvalue of multiplicity $2p$ and write
$Z_{\text{Sp}_{2(m-k)} (\F_q)} (s)^\circ = H \times \text{Sp}_{2p} (\F_q)$, and $u = u_{H^*} \otimes \theta$.
As in the construction of the eta correspondence, we can interpret $s \oplus I_{2k}$
as a semisimple element of $\text{Sp}_{2m} (\F_q)$ with the 
identity component of its centralizer of the form $H \times \text{Sp}_{2(p + k)} (\F_q)$.
Then define $\mathcal{A}_k (\rho, N')$ to be the $O(W,B)$-representation such that
$$\mathcal{A}_k (\rho, N') = \bigoplus_{\chi \subset A_k^\pm (\theta, N')}
r^{\text{O} (W,B)} [(s \oplus I_{2k}), u_{H^*} \otimes \chi]^{ \gamma} $$
with the sum running over every distinct irreducible unipotent $\theta$ appearing in $A_k^\pm (\theta, N')$
where the superscript sign is chosen to agree with the sign of $\phi^\pm (u)$ we take when
construction $\eta^V_{W,B} (\rho)$.

\end{enumerate}

\end{definition}

Similarly, consider the orthogonal metastable range. We note that we still need the separate
the cases of the parity of the dimension of the orthogonal space $W$, though the role of
$W$ is somewhat hidden in the notation.

\begin{definition}\label{OrthoCaseAltSumsDefn}
Consider a choice of $V$ and $(W,B)$ in the orthogonal (stable or) metastable range,
and write $dim (V) =2N$.
Suppose we are given $0 \leq k \leq N$ and $m'> 0$.
Consider an irreducible representation $\rho \in \widehat{\text{Sp}(V[-k])}$. 
\begin{itemize}
\item Suppose $dim (W) = 2m+1$ is odd. Write the
extended classification data of an irreducible $\text{Sp} (V[-k])$-representation $\rho$ as
$\rho = r^{\text{Sp} (V[-k])}[(s), u, \pm 1]$,
(omitting the central sign data $\pm 1$ if $s$ has no $-1$ eigenvalues).
Writing $2\ell$ for the multiplicity of $-1$ as an eigenvalue of $s$, say
$Z_{SO_{2N+1} (\F_q)} (s)^\circ = H \times SO_{2\ell}^\pm (\F_q)$, and $u = u_{H^*} \otimes \theta$.
Then define $\mathcal{A}_k (\rho, m')$ to be the $\text{Sp}(V)$-representation
$$\mathcal{A}_k (\rho, m') = \bigoplus_{\chi \subset A_k^\pm (\theta, m')}
r^{\text{Sp} (V)}[(s\oplus (-I)_{2k}), u_{H^*} \otimes \chi, \pm 1]$$
with the sum running over every distinct irreducible unipotent $\theta$ appearing in $A_k^\pm (\theta, m')$,
where the superscript sign is chosen to agree with the sign of $\psi^\pm (u)$ we take when
construction $\zeta^V_{W,B} (\rho)$, and
where the central sign in the classification data of each summand is the same as $\rho$'s.

\vspace{3mm}

\item Suppose $dim (W) = 2m$ is even. rite the
extended classification data of an irreducible $\text{Sp} (V[-k])$-representation $\rho$ as
$\rho = r^{\text{Sp} (V[-k])}[(s), u, \pm 1]$,
(omitting the central sign data $\pm 1$ if $s$ has no $-1$ eigenvalues).
Writing $2 p$ for the multiplicity of $1$ as an eigenvalue of $s$, write
$Z_{SO_{2N+1} (\F_q)} (s)^\circ = H \times SO_{2p+1} (\F_q)$, and $u = u_{H^*} \otimes \theta$.
Then define $\mathcal{A}_k (\rho, m')$ to be the $\text{Sp}(V)$-representation
$$\mathcal{A}_k (\rho, m') = \bigoplus_{\chi \subset A_k^\pm (\theta, m')}
r^{\text{Sp} (V)} [(s\oplus I_{2k}), u_{H^*} \otimes \chi, \pm 1] $$
with the sum running over every distinct irreducible unipotent $\theta$ appearing in $A_k^\pm
(\theta, m')$, where the superscript sign is chosen to agree with the sign of $\psi^\pm (u)$ we take when
construction $\zeta^V_{W,B} (\rho)$, 
where the central sign in the classification data of each summand is the same as $\rho$'s.

\end{itemize}

\end{definition}

\subsection{The main statement}\label{ExplicitTheoremSubsect}

Now that we have established the necessary notation to describe the terms of the decomposition
of a restricted oscillator representation claimed in Theorem \ref{IntroductionTheorem}, we
may restate it concretely.

The first part of our main result Theorem \ref{IntroductionTheorem},
extending the eta correspondence, can be explicitly restated as:

\begin{theorem}\label{ExplicitTheoremSymp}
Suppose $(V, (W,B))$ is in the symplectic metastable range
(which we recall means $\lceil \text{dim} (W)/2 \rceil \leq \text{dim} (V)/2 < \text{dim} (W)$). 
Then
\beg{EtaThmDecompFull}{\begin{array}{c}
Res_{\text{Sp}(V) \times \text{O}(W,B)} ( \omega[ V \otimes W]) = \\[1ex]
\displaystyle
\bigoplus_{k=0}^{h_{B}} \bigoplus_{\rho \in \widehat{O(W[-k], B[-k])} }
\mathcal{A}_k ( \rho, N'_\rho  )  \otimes \eta^V_{W,B} (\rho)
\end{array}
}
where $\mathcal{A} (\rho, N'_\rho)$ is considered as a representation of $O(W,B) $.
\end{theorem}

Similarly, the second part extending the zeta correspondence can be restated as
\begin{theorem}\label{ExplicitTheoremOrtho}
Suppose $(V, (W,B))$ is in the orthogonal metastable range
(which we recall means $h_W \leq \text{dim} (V) < \text{dim} (W)$
where $h_W$ denotes the dimension of a maximal isotropic subspace of $W$ with respect to $B$). 
Then
\beg{ZetaThmDecompFull}{\begin{array}{c}
Res_{\text{Sp}(V) \times \text{O}(W,B)} ( \omega[ V \otimes W]) = \\[1ex]
\displaystyle
\bigoplus_{k=0}^N \bigoplus_{\rho \in \widehat{\text{Sp}(V[-k])} }
\zeta_V^{W,B} (\rho)  \otimes
\mathcal{A}_k ( \rho, m'_\rho )
\end{array}
}
where $\mathcal{A} (\rho, m'_\rho)$ is considered as a representation of $\text{Sp}(V)$.
\end{theorem}

\section{Interpolation and proof of the metastable correspondences}\label{InterpolationSection}

In this section, we describe how the results of \cite{TotalHoweI, TotalHoweII}
for stable range reductive dual pairs
can be {\em interpolated}, which we can use to conclude Theorem \ref{IntroductionTheorem}.
We focus on the case of the eta correspondence (interpolating the symplectic stable range result).
The case of the zeta correspondence
is similar.

In Subsection \ref{InterpolatedBasicCats}, we define 
interpolated representation categories $\mathfrak{Rep} (\text{Sp}_{2t} (\F_q)$ and
$\overline{\mathfrak{Rep}} (\text{Sp}_{2t} (\F_q))$,
modelling symplectic groups of non-integer rank $t$.
$\mathfrak{Rep} (\text{Sp}_{2t} (\F_q))$ is defined to be
tensor generated by a basic object corresponding to the standard permutation representation, and
$\overline{\mathfrak{Rep}} (\text{Sp}_{2t} (\F_q))$ is taken to be
generated by basic objects corresponding to the oscillator representations, respectively. While
$\mathfrak{Rep} (\text{Sp}_{2t} (\F_q))$ is a more classical construction that
can be studied, for example, using oligomorphic groups (see \cite{HarmanSnowden}),
they do not actually contain objects corresponding to oscillator representations.
For our purposes, the
``finer" cateory $\overline{\mathfrak{Rep}} (\text{Sp}_{2t} (\F_q))$ is necessary.
These categories form semisimple abelian pre-Tannakian tensor categories for generic complex
values of $t$. However, for certain values, including $t=N \in \mathbb{N}_0$, they are not themselves semisimple, though they are ``semisimplifiable," and quotienting
out a certain class of ``negligible" morphisms does give semisimple pre-Tannakian categories.

In Subsection \ref{InterpolatedIsotyp}, we define subcategories
of $\mathscr{C}_B^{int} (t)\subseteq\overline{\mathfrak{Rep}} (\text{Sp}_{2t} (\F_q))$ isolating the ranges
of objects appearing in a certain tensor power of oscillator representations
(a ``partial pseudo-abelian envelope") corresponding to a specific choice of orthogonal
space $(W,B)$. For a fixed irreducible representation $\rho \in \widehat{O(W,B)}$, we 
further consider a subcategory $\mathscr{C}_\rho^{int} (t)$ detecting objects that
interpolate representations of $\text{Sp}_{2N} (\F_q)$ symplectic stable with $O(W,B)$
whose tensor product with $\rho$ appears in the restriction of $\omega [ \F_q^{2N} \otimes W]$.
We denote their images under semisimplification
by $\widetilde{\mathscr{C}}_B^{int} (t)$, $\widetilde{\mathscr{C}}_\rho^{int} (t)$.
In Subsection \ref{IntEtaSubsect}, we enumerate the objects of this category
in terms of ``formal Lusztig symbols" and write down their dimensions.
An interpolated decomposition of the restricted oscillator representation holds in these
categories (Theorem \ref{InterpolatedSemiSimpEtaThm} below).

In Subsection \ref{IntEtaProofSubsect}, we find that, at $t=N$ 
with $\text{Sp}_{2N} (\F_q)$ in the symplectic metastable
range with $O(W,B)$, the relationship
between an interpolated category $\widetilde{\mathscr{C}}_\rho^{int} (N)$ and
the subcategory of genuine $\text{Sp}_{2N} (\F_q)$-representations $\pi$ such that
$\rho \otimes \pi$ appears in the restricted oscillator representation, gives
that the decomposition of the restricted oscillator representation as a genuine representation
can be derived from the interpolated statement after ``cancelling terms."
Then we check that simplifying the cancelled terms gives the claimed decomposition
as a genuine representation of the symplectic group.

In Subsection \ref{InductiveConstruction}, we describe an application
of the statement to give an inductive procedure to explicitly construct
representations corresponding to classification data.

\subsection{Interpolated representation categories}\label{InterpolatedBasicCats}
First, to begin discussing interpolated representation theory,
we briefly recall P. Deligne's cosntruction of the category of representations of a general linear group $GL_c$
for $c \notin \Z$ (see \cite{DeligneSymetrique, DeligneTensor, DeligneMilne}):

This interpolation is based on the fact that in $Rep (GL_N)$, denoting the permutation representation
``basic object" of dimension $N$ by $X$, we have a stable structure of the endomorphism algebras
\beg{EndGLNIsSigman}{End_{GL_N} (X^{\otimes n}) \cong \C \Sigma_n}
when $N>> n$. 
However, in the genuine category $Rep (GL_N)$ of representations of $GL_N (\C)$, when the tensor
power degree becomes large compared to $N$, certain
Schur functors predicted to occur in $X^{\otimes n}$ have dimension $0$,
causing \rref{EndGLNIsSigman} to fail.
However, consider a formal basic generating object $X$ of dimension $c\in \C\smallsetminus \Z$,
we may construct a diagrammatic
category where the endomorphism algebra of $X^{\otimes n}$ is always the group algebra
$\C\Sigma_n$ (since the polynomial dimensions of the ``Schur functors" is never $0$
when applied to $dim (X) = c$). Formally adding direct sums and taking a pseudo-abelian envelope
gives a category we denote here by $\mathfrak{Rep} (GL_c)$, which is semisimple for $c\in \C\smallsetminus\Z$.

Attempting to apply this construction at $c= N$ gives a non-semisimple category $\mathfrak{Rep} (GL_N)$, which
for example, has simple objects of dimension $0$. However by applying a procedure
of {\em semisimplification} (see \cite{EGNOBook}), which is designed to eliminate these simple objects,
outputs a semisimple category which is precisely the genuine category of representation $Rep (GL_N)$.

One may also attempt to define an interolated category $\mathfrak{Rep} (\text{Sp}_{2t} (\F_q))$
with the standard permutation representation $\C V_t$ of dimension $q^{2t}$ as a basic
tensor-generating object, defined so that
\beg{IntHomRepSp}{\begin{array}{c}
Hom_{\mathfrak{Rep} (\text{Sp}_{2t} (\F_q)} ((\C V_t)^{\otimes m}, (\C V_t)^{\otimes n}) =\\[1ex]
Hom_{Rep (\text{Sp}_{2N} (\F_q))} ((\C V_N)^{\otimes m}, (\C V_N)^{\otimes n}),
\end{array}}
for a large enough $N>> m,n$ writing $V_N$ for the $q^{2N}$-dimensional 
underlying symplectic space
$\text{Sp}(V_N) = \text{Sp}_{2N} (\F_q)$.
The semisimplification of this category at $t=N$ is known to give
a semisimple pre-Tannakian category. However, it is not equivalent to the genuine representation
category $Rep (\text{Sp}_{2N} (\F_q))$. For example, there is no object of dimension $q^t$ corresponding
to the oscillator representation.

Another commonly considered interpolated category $\mathfrak{Rep} (GL_t (\F_q))$
is defined similarly to be tensor-generated by the standard representation of dimension $q^t$
written as $\C \mathbb{F}_q^t$, with
\beg{IntHomRepOrtho}{\begin{array}{c}
Hom_{\mathfrak{Rep} (GL_{t} (\F_q)} ((\C \F_q^t)^{\otimes m}, (\C \F_q^t)^{\otimes n}) =\\[1ex]
Hom_{Rep (GL_{N} (\F_q))} ((\C \F_q^N)^{\otimes m}, (\C \F_q^N)^{\otimes n}),
\end{array}}
for a large enough $N>> m,n$.
Similarly, we may define an interpolated category $\mathfrak{Rep} (O_t (\F_q))$
with basic object corresponding to the signed $\C W \otimes \epsilon (det)$, 
writing $\epsilon (det)$ for the sign representation on the center $\mu_2$ corresponding to $det$.
Writing the basic object
$\C W_t \otimes \epsilon (det)$ of $\mathfrak{Rep} (O_t (\F_q))$, we set
$Hom$-spaces between its tensor powers to be equal
to what they would be in an orthogonal group of high enough rank
(which may be chosen to be even or odd),
as in \rref{IntHomRepSp}, \rref{IntHomRepOrtho}.
Unlike in the symplectic case, modelling the (twisted) permutation representation
as the basic generating object in $\overline{\mathfrak{Rep}} (O_t (\F_q))$ will be enough for our purposes,
recalling that, in the case of $2$-dimensional $V$ with the standard symplectic form, we have
$$Res_{O(W,B)} (\omega [ \F_q^2 \otimes W]) = \C W^-$$
denoting the permutation representation $\C W$ tensored with $\epsilon (det)$ i.e.
the sign representation of $O(W,B)/SO(W,B)$,
(which can be seen from restricting first to $GL(W)$).

Now, both $\mathfrak{Rep} (GL_t(\F_q))$ aand $\overline{\mathfrak{Rep}} (O_{t} (\F_q))$
are known to, after semisimplification when needed at natural number values of $t$,
give semisimple categories. This can both be proved directly
(see, for examples of this method, \cite{DeligneSymetrique, DeligneTensor, VectorDelannoy}),
or be approached using
the theory of oligomorphic groups of A. Snowden and N. Harman (see \cite{HarmanSnowden}).

However, we need to define a different interpolated tensor category $\overline{\mathfrak{Rep}} (\text{Sp}_{2t} (\F_q))$
with the oscillator representations as basic objects (since the structure of endomorphism algebras
and homomorphism modules between tensor products of oscillator representation is stable
when the dimension of the underlying symplectic space is large enough compared to the tensor product
degrees).
For more details, see \cite{OscillatorRepsFull}.
When attempting to put $t = N \in \mathbb{N}$, as in the case of the interpolated
representation categories of the symmetric group, $\overline{\mathfrak{Rep}} (\text{Sp}_{2N} (\F_q))$
is not a semisimple category (since, again, it has simple objects
of dimension $0$).

Consider a category $\overline{\mathfrak{Rep}} (\text{Sp}_{2t} (\F_q))_0$
with objects $\boldsymbol \omega_a^{\otimes m} \otimes \boldsymbol \omega_b^{\otimes n}$, with $Hom$-spaces between objects  defined to be, as $\C$-vector spaces
$$\begin{array}{c}
Hom_{\overline{\mathfrak{Rep}} (\text{Sp}_{2t} (\F_q))_0} (\boldsymbol \omega_a^{\otimes  m} \otimes \boldsymbol \omega_b^{\otimes n} , \boldsymbol \omega_a^{\otimes  p} \otimes \boldsymbol \omega_b^{\otimes \ell}):=\\[1ex]
Hom_{\text{Sp}_{2N} (\F_q)} (\omega_a^{\otimes  m} \otimes \omega_b^{\otimes n} , \omega_a^{\otimes  p} \otimes \omega_b^{\otimes \ell})
\end{array}$$
for $N>> m,n,p,l$. Tensor products of morphisms and the actions of the bijections
on the coordinates of the different tensor factors are defined in the obvious way from
$Rep (\text{Sp}_{2N} (\F_q))$ for a large enough $N$.
To construct a category (and involve the constant $t$), we also need to define
a partial trace operation corresponding to matching factors of 
the generating objects $\boldsymbol \omega_a$. In this case, it suffices to define
a trace operation for endomorphisms of each $\boldsymbol \omega_a$. We do this by considering
$$\begin{array}{c}
End_{\overline{\mathfrak{Rep}} (\text{Sp}_{2t} (\F_q))_0} (\boldsymbol \omega_a) = End_{\text{Sp}_{2N} (\F_q)} (\omega_a) \\[1ex]
= Hom_{\text{Sp}_{2N} (\F_q)} (1, \C V_N) \cong (\C V_N)^{\text{Sp}_{2N}},
\end{array}$$
which has a basis consisting of $(0)$ and $\sum_{v\neq 0\in V_N} (v)$, where we put
$$tr ((0)) = q^t, \hspace{5mm} tr(\sum_{v\neq 0 \in V_N} (v)) = 0.$$
Composition can be defined with a combination of tensor product, permutation, and partial trace (for
more details, see \cite{OscillatorRepsFull}). 
This defines a $\C$-linear category $\overline{\mathfrak{Rep}} (\text{Sp}_{2t} (\F_q))_0$.

We then define the category $\overline{\mathfrak{Rep}} (\text{Sp}_{2t} (\F_q))$ by 
first formally adding direct
sums to $\overline{\mathfrak{Rep}} (\text{Sp}_{2t} (\F_q))_0$, and then applying
a pseudo-abelian envelope, adding new objects defined as the images of
idempotents in the endomorphism algebras of the objects
$\boldsymbol \omega_a^{\otimes m} \otimes \boldsymbol \omega_b^{\otimes n}$.

\vspace{3mm}

In each case, the interpolated category is constructed from the data of
a system of $Hom$-spaces between tensor powers of the basic object, with operations
of permutation action, tensor product, and partial trace. This data can in fact be
captured in the universal algebra structure of a {\em T-algebra}. To describe a
$\C$-linear additive category with strong duality and associative commutative unital tensor product
generated by a basic object $X$,
its corresponding T-algebra $\mathcal{T}$ consists of vector spaces for every pair of finite sets $S,T$
$$\mathcal{T}_{S,T} = Hom  (X^{\otimes S}, X^{\otimes T}),$$
with appropriate functoriality over the category of finite sets, 
linear partial trace operations corresponding to partial bijections
between $S$ and $T$ and tensor product operations corresponding the set disjoint union, with appropriate axioms.
Here, we denote the T-algebra corresponding to $\mathfrak{Rep} (GL_t (\F_q)$
with basic object corresponding to the standard representation
$X=\C \F_q^t$ by $\mathcal{T} [GL_{t} (\F_q)]$ and the T-algebra
corresponding to $\overline{\mathfrak{Rep}} (\text{Sp}_{2t} (\F_q)$
with basic object $X = \boldsymbol \omega_a \otimes \boldsymbol \omega_b$
by $\overline{\mathcal{T}}[\text{Sp}_{2t} (\F_q)]$.
For more details, see \cite{OscillatorRepsFull}.

\vspace{3mm}

Let us denote the {\em semisimplification} of a
$\C$-linear additive category with strong duality and associative commutative unital tensor product
$\mathscr{C}$ by $\mathscr{S} (\mathscr{C})$.
Recall that semisimplification refers to a construction
quotienting out {\em negligible morphisms}, such as simple idempotents
with trace $0$, \cite{EGNOBook}.

We recall a result of \cite{OscillatorRepsFull}.

\begin{proposition} 
In every case of $t$, the semisimplification of the category
$\overline{\mathfrak{Rep}} (\text{Sp}_{2t} (\F_q))$ is a semisimple (and, in particular,
abelian) pre-Tannakian category.
For values of $t\in \C$ such that $q^t \neq \pm q^n$ for $n \in \N_0$, the category
$\overline{\mathfrak{Rep}} (\text{Sp}_{2t} (\F_q))$ itself is semisimple.
\end{proposition}

\begin{proof}

First, there is an inclusion of T-algebras
$$\overline{\mathcal{T}} [\text{Sp}_{2t} (\F_q)] \hookrightarrow \mathcal{T} [GL_t (\F_q)],$$
since 
for every $N$,
the restriction of an oscilaltor representation $\omega_a$
to $GL_N (\F_q) \subseteq \text{Sp}_{2N} (\F_q)$ is isomorphic to
$$Res_{GL_N (\F_q)} (\omega_a) \cong  (\C V_N) \otimes \epsilon (\text{det}).$$
In particular, then for finite sets $S,T$, restriction gives an inclusion from each $Hom$-space
$$\begin{array}{c}
Hom_{\overline{\mathfrak{Rep}} (\text{Sp}_{2t} (\F_q))}
(\boldsymbol \omega_a^{\otimes S_1} \otimes \boldsymbol \omega_b^{\otimes S_2},
\boldsymbol \omega_a^{\otimes T_1} \otimes \boldsymbol \omega_b^{\otimes T_2})=\\[1ex]
Hom_{\text{Sp}_{2N} (\F_q)}
( \omega_a^{\otimes S_1} \otimes \omega_b^{\otimes S_2},
 \omega_a^{\otimes T_1} \otimes  \omega_b^{\otimes S_2})
\end{array}
$$
making up $\overline{\mathcal{T}} [\text{Sp}_{2t} (\F_q)]_{S,T}$ for $S_1 \amalg S_2 = S, T_1 \amalg T_2 = T$,
into the $GL_N (\F_q)$-equivariant $Hom$-space on the restrictions
$$Hom_{GL_N (\F_q)} ( (\C V_N \otimes \epsilon (\text{det}))^{S_1 \amalg S_2}, (\C V_N \otimes
\epsilon (\text{det}))^{T_1 \amalg T_2}),$$
which is isomorphic to $\mathcal{T} [GL_t (\F_q)]_{S,T}$.
Partial trace, tensor product, and functoriality (and therefore composition) are all
compatible.

\vspace{3mm}

We then apply

\begin{lemma}\label{SemiSimpLemma}
The semisimplification of a $\C$-linear additive category with strong duality
and an associative, commutative, unital tensor product generated by a basic object $X$
is semisimple if and only if for every endomorphism $f \in End (X^{\otimes n})$,
if the trace of $f$ is non-zero, then for every $n$, there exists an $m>n$ such that
$$tr (f^{\circ m}) \neq 0.$$
\end{lemma}

\begin{proof}[Proof of Lemma \ref{SemiSimpLemma}]

To prove sufficiency, consider an endomorphism $f$ of some tensor power $X^{\otimes n}$
is non-negligible, i.e. there exists some morphism $g\in End (X^{\otimes n})$ such that
the trace of $f \circ g$ is non-zero. The trace condition then gives that for every
$n$, there exists a $m> n$ such that
$$tr ((f \circ g)^{\circ m}) \neq 0,$$
and hence, $f$ is not an element of the Jacobian ideal of the endomorphism algebra $End (X^{\otimes n})$,
and in particular, is not nilpotent. Therefore, the semisimplification of the category is semisimple. 

Necessity follows from the general result that in a semisimple $\C$-algebra (e.g.
the endomorphism algebra of $X^{\otimes n}$ in the semisimplification),
if some general trace operation (i.e. a linear combination of trace on each factor, consider the endomorphism
algebra as a product of matrix algebras) is non-zero on an element $f$, then for every $n$, there exists an
$m>n$ such that the trace operation is non-zero on $f^{\circ m}$. (This follows, for example,
by considering the Vanermonde determinant.)
\end{proof}

In particular, if the semisimplification
of a category defined by a T-algebra $\mathcal{T}$ is semisimple,
then the semisimplifications of any category defined by a sub-T-algebra $\mathcal{T}' \subseteq \mathcal{T}$ is semisimple as well.
Therefore, since the semisimplification of $\mathfrak{Rep} (GL_t (\F_q))$
is semisimple, so is the semisimplification of $\overline{\mathfrak{Rep}} (\text{Sp}_{2t} (\F_q))$.

The second claim follows, for example, by examining the polynomial order of $\text{Sp}_{2N} (\F_q)$,
and replacing $N$ by $t$, to conclude that every indecomposable object is non-vanishing in the semisimplification.
\end{proof}

For our purposes here, we will want to consider the case of $t = N$, making
the first part of this statement more relevant. 
The second part of this statement for generic values of $t$
can be used to conclude an interpolated version of our decomposition
statement, for example, describing an interpolated eta correspondence
$$\eta^t_{W,B}: \widehat{O(W,B)} \hookrightarrow \overline{\mathfrak{Rep}} (\text{Sp}_{2t} (\F_q)),$$
sending the irreducible representations of $O(W,B)$ to simple objects of
$\overline{\mathfrak{Rep}} (\text{Sp}_{2t} (\F_q))$.

\subsection{Partial pseudo-abelian envelopes and ``isotypical" subcategories}
\label{InterpolatedIsotyp}
In this subsection, we give an argument using interpolation for
our full statement of Howe duality. From here on, we will focus on the case
of interpolating the eta correspondence with target in odd orthogonal group representations,
for simplicity. In this part of the argument, every other case is entirely similar.
Fix a choice of orthogonal space and bilinear form $(W,B)$.

Recall that in this case, for a choice of $V$ lying with $(W,B)$ in the 
symplectic stable range,
the eta correspondence is an injective map
$\eta^V_{W,B}$ sending the irreducible representations of $O(W,B)$
to irreduicble representations of $\text{Sp}(V)$.
The methods of \cite{TotalHoweI, TotalHoweII} ultimately only rely on information carried by the
endomorphism algebras of (tensor powers of) oscillator representations, whose structure is stable
under enlarging the dimension of $V$,
in the stable symplectic range, and therefore, by definition, they pass to the interpolated categories.

Write $dim (W) = n$, and suppose $B$ is, as a symmetric bilinear form, equivalent to
a diagonal matrix with entries $a_1, \dots , a_n \in \F_q^\times $. We may then consider
the object
$$\boldsymbol \omega^{\otimes B} = \boldsymbol \omega_{a_1}
\otimes \dots \otimes \boldsymbol \omega_{a_n}$$
in $\overline{\mathfrak{Rep}} (\text{Sp}_{2t} (\F_q))$.
Let us consider the ``partial pseudo-abelian envelope"
$\mathscr{C}_B^{int} (t)$ defined as the subcategory
of $\overline{\mathfrak{Rep}} (\text{Sp}_{2t} (\F_q))$ consisting of images of
idempotents of $End_{\mathfrak{Rep} (\text{Sp}_{2t} (\F_q))} (\boldsymbol \omega^{\otimes B})$.
We do not consider a tensor product on $\mathscr{C}_B^{int} (t)$, only working
with its structure as an additive $\C$-linear category.

Further, considering
$$\begin{array}{c}
End_{\overline{\mathfrak{Rep}} (\text{Sp}_{2t} (\F_q))} (\boldsymbol\omega^{\otimes B})
\cong 
End_{\text{Sp}_{2N} (\F_q)} (\omega_{a_1}[V_N] \otimes \dots \otimes \omega_{a_n} [V_N])  = \\[1ex]
End_{\text{Sp}_{2N} (\F_q)} ( \omega [ V_N \otimes W]),
\end{array}$$
for a large enough rank $N$, where on the right hand side, we consider
the restriction of $\omega [V_N \otimes W]$ along
the inclusion
$$\text{Sp}_{2N} (\F_q) = \text{Sp} (V_N) \subseteq \text{Sp}(V_N) \times O(W,B) \subseteq \text{Sp} (V_N \otimes W).$$
Therefore, there is a built in action of $O(W,B)$ on the endomorphism algebras
$End_{\overline{\mathfrak{Rep}} (\text{Sp}_{2t} (\F_q))} (\boldsymbol\omega^{\otimes B})$.

In particular, for a fixed irreducible representation $\rho$ of $O(W,B)$, we may consider
the full {\em $\rho$-isotypical} subcategory
$\mathscr{C}_\rho^{int} (t)\subseteq \mathscr{C}_B (t)$, with objects consisting of
images of idempotents $\iota$ in 
$End_{\overline{\mathfrak{Rep}} (\text{Sp}_{2t} (\F_q))} (\boldsymbol\omega^{\otimes B})$
where, considering the $O(W,B)$-fixed point algebra
$$\begin{array}{c}
End_{O(W,B) \times \overline{\mathfrak{Rep}} (\text{Sp}_{2t} (\F_q))} (\boldsymbol\omega^{\otimes B})=\\[1ex]
(End_{\overline{\mathfrak{Rep}} (\text{Sp}_{2t} (\F_q))} (\boldsymbol\omega^{\otimes B}))^{O(W,B)}
\end{array},$$
there is a simple idempotent of the form $\kappa \otimes \iota$, for
$\kappa$ an idempotent in $Rep(O(W,B))$ with image isomorphic to $\rho$.
We may also describe these objects as the images of $\iota$ lying in the $\rho$-isotypical part
of $End_{\overline{\mathfrak{Rep}} (\text{Sp}_{2t} (\F_q))} (\boldsymbol\omega^{\otimes B})$
as an $O(W,B)$-representation.
We cannot take a semisimplification of $\mathscr{C}_\rho^{int} (t)$, since we have given
up its tensor structure.
However, we may consider the images of $\mathscr{C}_B^{int} (t)$, $\mathscr{C}_{\rho}^{int} (t)$ under
the quotient semisimplification functor
$$\overline{\mathfrak{Rep}} (\text{Sp}_{2t} (\F_q)) \rightarrow \mathscr{S} (\overline{\mathfrak{Rep}} (\text{Sp}_{2t} (\F_q))).$$
Writing $\widetilde{\mathscr{C}}_B^{int} (t)$, $\widetilde{\mathscr{C}}^{int}_\rho (t)$ for these images,
they form semisimple abelian subcategories of
$\mathscr{S} (\overline{\mathfrak{Rep}} (\text{Sp}_{2t} (\F_q)))$.
Note that this is only non-trivial for a choice of $t=N$ a natural number.

\vspace{5mm}

On the other hand, we may also consider the full subcategories
$\mathscr{C}_\rho^{gen} (N)$ of $Rep (\text{Sp}_{2N} (\F_q))$
consisting of direct sums of all genuine irreducible representations $\pi\in Rep (\text{Sp}_{2N} (\F_q))$
such that
$$\pi \otimes \rho \subseteq Res_{O(W,B) \times \text{Sp}(V)} ( \omega [ V\otimes W]).$$

\subsection{Interpolating correspondences}\label{IntEtaSubsect}

The purpose of this subsection is to describe how the objects of the subcategories
$\mathscr{C}_B^{int} (t)$ of
the interpolation
$\overline{\mathfrak{Rep}} (\text{Sp}_{2t} (\F_q))$ 
constructed in tensor powers of oscillator representations of odd degree
such that the total product of central characters has a certain quadratic character $\alpha$
(i.e. corresponding to odd orthogonal spaces and symmetric bilinear
forms of a certain prescribed discriminant $\alpha$), can be
written down according to an ``formally interpolated Lusztig parametrization."
We discuss this in detail for the case of $dim (W)$ odd. All other cases are similar.

More specifically, an object of $\mathscr{C}_B^{int} (t)$
is of the form
$\eta^{t}_{W,B} (\rho)$,
for $W$ of dimension $2m+1$ for $m \in \N$, with a form $B$ of discriminant $disc(B) = \alpha$,
and an irreducible representation $\rho$ of $O(W,B)$.
Say that as a representation of $O(W,B) $,
$\rho$ is of the form $r^{\text{O} (W, B)}[(s),u]^{\pm 1}$
corresponding to classification data
with semisimple part $(s) \in \text{Sp}_{2m} (\F_q) = (SO_{2m+1} (\F_q))^*$, unipotent part
$u \in \widehat{(Z_{\text{Sp}_{2m} (\F_q)} (s)^\circ)^*}_u$, and central sign data $\pm 1$.
More concretely, further say that $s$ has $-1$ as an eigenvalue of multiplicity $2\ell$, writing
$Z_{\text{Sp}_{2m} (\F_q)} (s)^\circ = H \times \text{Sp}_{2\ell} (\F_q)$, and
$$u = u_{H^*} \otimes {\lambda_1< \dots < \lambda_a \choose \mu_1< \dots  < \mu_b}$$
for $u_{H^*} \in \widehat{H^*}_u$ and
${\lambda_a< \dots < \lambda_a \choose \mu_1< \dots < \mu_b}$ (switching rows so that $a-b$
is $1$ mod $4$) specifying a unipotent 
representation of $SO_{2\ell+1} (\F_q) = (\text{Sp}_{2\ell} (\F_q))^*$.
Then, for the sign $+$, we say that $\eta^{t}_{W,B} (r^{\text{O} (W, B)}[(s),u]^{+ 1})$
corresponds to ``interpolated classification data"
\beg{LusztigClassDataOfRho+InterpolateTot}{[\phi^+ (s), \widetilde{u_{H^*}} \otimes {\lambda_1 < \dots <\lambda_a
\choose \mu_1< \dots < \mu_b  \hspace{5mm} t'_\rho }],}
writing $t'_\rho = t-m + \frac{a+b-1}{2}$.
This is exactly the classification data of a stable range eta correspondence $\eta^{t}_{W,B} (\rho)$
for $dim (V) = 2N$, with $N$ replaced by $t$ (we omit the final $<$ symbol in the second row of the
symbol notation, since at an interpolated value of $t$, writing $\mu_b < t'_\rho $
may be false or incomparable). Again, $\text{Sp}_{2t} (\F_q)$ is not a genuine group, and writing
$\phi^+ (s)$ indicates an element with finitely many eigenvalues not equal to $-1$
(which would contribute genuine factors in ``the 
identity component of its centralizer") and has $-1$
as an eigenvalue of ``multiplicity $2(t-m+\ell)$."
Interpolating the stable formula one would obtain for $\text{Sp}_{2N} (\F_q)$, replacing $N$ by $t$,
its dimension is
\beg{InterpolatedEta+Dim}{\begin{array}{c}
\displaystyle dim (\eta^{t}_{W,B} (r^{\text{O} (W, B)}[(s),u]^{+ 1})) = 
\\
\displaystyle 
\frac{\displaystyle dim (\rho) \cdot \prod_{i=t'+1}^t (q^{2i} -1) \cdot
\prod_{i=1}^a (q^{t'_\rho} + q^{\lambda_i}) \cdot
\prod_{i=1}^b (q^{t'_\rho} - q^{\mu_i})}{2 \cdot q^{(a+b-1)(a+b+1)/4} \cdot |SO_{2m+1} (\F_q)|_{q'}}.
\end{array}
}

Similarly, at the sign $-1$, we say that $\eta^{t}_{W,B} (r^{\text{O} (W, B)}[(s),u]^{- 1})$
corresponds to ``interpolated classification data"
\beg{LusztigClassDataOfRho-InterpolateTot}{[\phi^- (s), \widetilde{u_{H^*}} \otimes {\lambda_1 < \dots <\lambda_a \hspace{5mm} t'_\rho
\choose \mu_1< \dots < \mu_b  }],}
writing $t'_\rho = t-m + (a+b-1)/2$.
Interpolating the stable formula one would obtain for $\text{Sp}_{2N} (\F_q)$, replacing $N$ by $t$,
its dimension is
\beg{InterpolatedEta-Dim}{\begin{array}{c}
\displaystyle dim (\eta^{t}_{W,B} (r^{\text{O} (W, B)}[(s),u]^{ -1})) = 
\\
\displaystyle 
\frac{\displaystyle dim (\rho) \cdot \prod_{i=t'+1}^t (q^{2i} -1) \cdot
\prod_{i=1}^a (q^{t'_\rho} - q^{\lambda_i}) \cdot
\prod_{i=1}^b (q^{t'_\rho}+ q^{\mu_i})}{2 \cdot q^{(a+b-1)(a+b+1)/4} \cdot |SO_{2m+1} (\F_q)|_{q'}}.
\end{array}
}

Now we may consider the ``restriction" functor
$$\mathbf{Res}: \mathscr{S}(\overline{\mathfrak{Rep}} (\text{Sp}_{2nt} (\F_q))) \rightarrow \mathscr{S}(\overline{\mathfrak{Rep}} (\text{Sp}_{2t} (\F_q)))
\boxtimes Rep (O(W,B)).$$
The interpolated Howe duality statement then is

\begin{theorem}\label{InterpolatedSemiSimpEtaThm}
In the semisimplification $\mathscr{S} (\overline{\mathfrak{Rep}} (\text{Sp}_{2t} (\F_q)))$,
the original decomposition of $\mathbf{Res} ( \boldsymbol \omega_1)$ as
\beg{InterpolatedEtaCorrStatement}{
\bigoplus_{k=0}^{m_W} \bigoplus_{\rho \in \widehat{O(W[-k],B[-k])}} 
  \eta^{t}_{W,B} (\rho)  \boxtimes  \text{Ind}_{P_k} (\rho \otimes \epsilon (\text{det}))
}
holds, as objects of $\mathscr{S}(\overline{\mathfrak{Rep}} (\text{Sp}_{2t} (\F_q))) \boxtimes Rep (O(W,B))$.
\end{theorem}

In summary, the objects of $\mathscr{C}_B^{int} (t)$ are precisely direct sums of
all
\beg{Etatpi}{\eta^t_{W[-k], B[-k]} (\rho)}
for irreducible representations $\rho \in \widehat{O(W[-k], B[-k])}$.
For a fixed irreducible representation $\rho$ of $O(W,B)$, the objects of
$\mathscr{C}_\rho^{int} (t)$ consist of direct sums of objects
$\eta^t_{W[-k], B[-k]} (\rho')$ corresponding to irreducible representations $\rho' \in \widehat{O(W[-k], B[-k])}$ such that $\rho$
is a summand of the parabolic induction
$$\rho \subseteq \text{Ind}_{P_k} (\rho' \otimes \epsilon (\text{det}))$$
writing $P_k\subseteq O(W,B)$ for the maximal parabolic subgroup
with Levi factor $O(W[-k],B[-k]) \times GL_k (\F_q)$, considering $\epsilon (\text{det})$
as a representation of the factor $GL_k (\F_q)$.

At $t=N$ corresponding to a reductive dual pair
$(\text{Sp}_{2N} (\F_q), O(W,B))$ in 
the symplectic metastable range, the semisimplification images $\widetilde{\mathscr{C}}_B (N)$,
$\widetilde{\mathscr{C}}_\rho (N)$ are semisimple categories
with objects consisting of direct sums of simple objects corresponding to
all formal interpolated classification, eliminating $0$-dimensional objects,
which occur precisely when $N_\rho' = \lambda_i$ or $\mu_i$ for some $i$ in
\rref{LusztigClassDataOfRho+InterpolateTot} or \rref{LusztigClassDataOfRho-InterpolateTot}.
Note that the remaining formal interpolated classification objects,
say where $\lambda_i< N'_\rho < \lambda_{i+1}$ in \rref{LusztigClassDataOfRho+InterpolateTot}
or $\mu_i< N'_\rho < \mu_{i+1}$ in \rref{LusztigClassDataOfRho-InterpolateTot}, have dimension
equal to a genuine irreducible $\text{Sp}_{2N} (\F_q)$-representation where $N'_\rho $ is inserted
in the appropriate place, multiplied by $(-1)^{a-i}$ or $(-1)^{b-i}$.
Call this the {\em true permutation} of the formal interpolated Lusztig data
as $t=N$.

\subsection{The proof of the metastable eta correspondence}\label{IntEtaProofSubsect}

Now, to approach a choice of $(V,(W,B))$ in the symplectic metastable range,
we attempt to apply Theorem \ref{InterpolatedSemiSimpEtaThm} to $t= N$, giving a decomposition in
the semisimplification $\mathscr{S} (\overline{\mathfrak{Rep}} (\text{Sp}_{2t} (\F_q)))$ in terms
of objects of
$\widetilde{\mathscr{C}}_B (N)$.
We must relate this category to genuine $\text{Sp}_{2N} (\F_q)$-representations.
In fact, we claim that replacing the formal classification data by its true permutation,
with the corresponding sign, gives a genuine decomposition of the restricted oscillator representation
in the Grothendieck group $K( Rep (\text{Sp}_{2N} (\F_q)))$.
Simplifying will precisely give the claimed decomposition in Theorem \ref{ExplicitTheoremSymp}.

\vspace{5mm}

To see this, for each $\rho \in \widehat{O(W,B)}$, consider functors
$$\Phi: \mathscr{C}^{gen}_\rho (N) \rightarrow \widetilde{\mathscr{C}}_\rho^{int} (N)$$
defined as follows: Consider a simple object $\pi$ of the source, such that
$$\pi \otimes \rho \subseteq Res_{\text{Sp} (V) \times O(W,B)} (\omega [ V\otimes W]).$$
We may consider an idempotent $\iota_\pi$ in the $\text{Sp}(V)$-equviariant
endomorphism algebra of $Res_{\text{Sp} (V)} (\omega [ V\otimes W]) \cong \omega^{\otimes B}$
whose image is one of these copies of $\pi$. By duality, since each oscillator
representation has $\omega_a \otimes (\omega_a)^\vee \cong \C V)$, we may conisder
$$\begin{array}{c}
\iota_\pi \in End_{\text{Sp}(V)} (\omega^{\otimes B}) \cong Hom_{\text{Sp}(V)} (1, (\C V)^{\otimes B})=\\[1ex]
(\C (V\otimes W))^{\text{Sp}(V)}
\end{array}$$
as a linear combination of $\text{Sp}(V)$-orbits on $V \otimes W = V^{\oplus n}$ (recall that
an $\text{Sp}(V)$-orbit consits of a set of $n$-tuples of vectors $(v_1, \dots , v_n)$ satisfying
some linear (in)dependence conditions, and equations for the values of the symplectic form
on pairs of them).
Theses orbits can also be considered as orbits of any $\text{Sp}_{2M} (\F_q)$ acting on
$(\F_q^{2M})^{\oplus n}$
for any higher $M$, and therefore $\iota_\pi$ corresponds to an interpolated endomorphism
$$End_{\overline{\mathfrak{Rep}} (\text{Sp}_{2N} (\F_q))} (\boldsymbol \omega^{\otimes B})
\cong ((\C (\F_q^{2M}\otimes W))^{\text{Sp}_{2M} (\F_q)}$$
for $M>>n$ (by the definition of $\overline{\mathfrak{Rep}} (\text{Sp}_{2N} (\F_q))$).
Since
partial trace operations 
(and therefore compositions) are computed the same
in $Rep (\text{Sp}_{2N} (\F_q))$ and $\overline{\mathfrak{Rep}} (\text{Sp}_{2N} (\F_q))$
(the difference between them arising instead from certain morphisms in the interpolated category
not existing in the genuine category), this new endomorphism
is still an idempotent, with image equal to an object in $\widetilde{\mathscr{C}}_\rho^{int} (N)$ of
the same dimension as $\pi$. We put $\Phi (\pi)$ to be this object.

On the other hand, define
\beg{PsiGrothDefn}{\Psi: K (\widetilde{\mathscr{C}}_\rho^{int} (N)) \rightarrow K (\mathscr{C}_\rho^{gen} (N))}
by sending a simple object in $\widetilde{\mathscr{C}}_\rho^{int} (N)$
which must be of the form \rref{LusztigClassDataOfRho+InterpolateTot} or \rref{LusztigClassDataOfRho-InterpolateTot} at $t= N$ 
for some choice of $s,u$, ${\lambda_1< \dots < \lambda_a
\choose \mu_1< \dots < \mu_b}$ to $(-1)^{b-i}$, where $i$ is the index so that
$\mu_i < t'_\rho< \mu_{i+1}$ or 
$\lambda_i< t'_\rho < \lambda_{i+1}$, times
the genuine irreducible representation of $\text{Sp}(V)$ whose Lusztig calssification data
is the same as \rref{LusztigClassDataOfRho+InterpolateTot} or \rref{LusztigClassDataOfRho-InterpolateTot} with the formal interpolated symbol replaced by 
$${\lambda_1< \dots < \lambda_a \choose
\mu_1< \dots < \mu_i< t'_\rho < \mu_{i+1} < \dots < \mu_b}$$
or
$${\lambda_1< \dots < \lambda_i< t'_\rho< \lambda_{i+1} < \dots < \lambda_a \choose
\mu_1< \dots < \mu_b},$$
respectively. In other words, we them precisely to their signed true permutations.

\begin{proposition}
The composition of
$$\diagram
K (\mathscr{C}_\rho^{gen} (N)) \rto^{K(\Phi)} & K (\widetilde{\mathscr{C}}_\rho^{int} (N)) \rto^{\Psi} & K (\mathscr{C}_\rho^{gen} (N))
\enddiagram$$
is $Id_{K (\mathscr{C}^{gen}_\rho (N))}$.
\end{proposition}

\begin{proof}
Note that this holds immediately when formal classification data is actually
genuine classification data defining a $\text{Sp}_{2N} (\F_q)$-representation,
since both $K (\Phi)$ and $\Psi$ act as the identity on these objects, considered in either
categories. The general statement follows since dimensions are preserved by $\Psi$ and $\Phi$, and it is not
possible for $\Psi \circ K (\Phi)$ when applied to an irreducible
representation of $\mathscr{C}^{gen}_\rho (N)$ to output a linear combination of multiple different
irreducible representations in $Rep (\text{Sp}_{2N} (\F_q))$
with integer coefficients, both by dimension and the fact
that it would violate the decomposition \rref{InterpolatedEtaCorrStatement}.
\end{proof}

Therefore, the decomposition of the restricted oscillator representation
as a genuine representation of $\text{Sp}_{2N} (\F_q)$ can be obtained from
\rref{InterpolatedEtaCorrStatement} by applying $\Psi$, and cancelling terms as in $K(\mathscr{C}_B (N))$.
It remains to reconcile this cancellation with the claimed answer.
Consider a term arising from $\rho \in \widehat{O (W[-j],B[-j])}$.
Write $\rho= r^{\text{O} (W[-j], B[-j])}[(s),u]^{\pm 1}$.
Suppose $s \in \text{Sp}_{2(m-j)} (\F_q) = (SO_{2(m-j)+1} (\F_q))^*$
has $-1$ as an eigenvalue of multiplicity $2\ell$, and write
$Z_{\text{Sp}_{2(m-j)} (\F_q)} (s)^\circ = H \times \text{Sp}_{2\ell} (\F_q)$.
Then $\phi^\pm (s)$ remains as defined in the stable range, so that
$Z_{SO_{2N+1} (\F_q)} (\phi^\pm (s))^\circ = H^* \times SO_{2(N-m+\ell)}^\pm (\F_q)$.
If $\phi^\pm (u)$ is constructible, then the interpolated eta correspondence
outputs a genuine $\text{Sp}_{2N} (\F_q)$-representation, which is our proposed 
metastable construction
of $\eta^V_{W,B}$.
In the case corresponding to when $\phi^\pm (u)$ is ``inconstructible," however,
the interpolated eta correspondence outputs a representation involving an illegal symbol
\beg{IllegalSymbolOriginal}{\phi (u) = \widetilde{u_{H^*}} \otimes {\lambda_1< \dots < \lambda_a \hspace{5mm} N'_\rho
\choose
\mu_1< \dots < \mu_b}.}
This term 
$$r^{\text{Sp} (V)}[ (\phi^\pm (s)), \phi^\pm (u) , \epsilon (s) \cdot disc(B)] \otimes \text{Ind}_{P_i} ( r^{\text{O} (W[-j], B[-j])}[(s),u]^{\pm 1})$$
must be eliminated when we semisimplify.

To see where this term cancels, suppose that
$$\lambda_1< \dots < \lambda_{c} < N_\rho' < \lambda_{c+1} < \dots < \lambda_a,$$
then, for $c+1 \leq k \leq a$, we have that the multiplicity $2p$ of $-1$ as an eigenvalue of $s$
must be greater than or equal to $2(\lambda_{k} - N'_\rho)$.
first let us write 
$$j^{(k)} = j-\lambda_{k} + N'_\rho .$$
Note that by the rank conditions of Lusztig symbols,
the multiplicity $2p$ of $-1$ as an eigenvlaue must be greater than or equal to $2(\lambda_{k} -N'_\rho)$.
Then write $s^{(k)}$ for the semisimple element of 
$$\text{Sp}_{2(m-j+\lambda_{k} -N'_\rho)} (\F_q) = \text{Sp}_{2(m-j^{(k)})} (\F_q)= (SO_{2(m-j^{(k)})+1} (\F_q))^*$$
obtained by removing
$2(\lambda_{k} - N'_\rho)$ eigenvalues $-1$.
Write
$$\begin{array}{c}
u^{(k)} = u_{H^*} \otimes  \\
\\
\displaystyle {\lambda_1< \dots < \lambda_{c}< N'_\rho < \lambda_{c+1}
< \dots < \lambda_{k-1} <\lambda_{k+1}<\dots < \lambda_a
\choose \mu_1< \dots < \mu_b}
\end{array}$$
Consider the representation corresponding to the Lusztig data
$$r^{O(W[-j^{(k)}],B[-j^{(k)}] )}[(s^{(k)}), u^{(k)}]^{\pm 1}$$
giving an irreducible representation of $O_{2(m-j^{(k)})+1} (\F_q)$.
Then $\psi^\pm (s) = \psi^\pm (s^{(k)})$, and we have
$$N'_{r^{O(W[-j^{(k)}],B[-j^{(k)}] )} [(s^{(k)}), u^{(k)}]^{\pm 1}} = \lambda_k .$$
Therefore,
$$\begin{array}{c}
\phi^\pm (u^{(k)}) = \widetilde{u_{H^*}}\otimes\\
\\
\displaystyle {\lambda_1< \dots < \lambda_{c}< N'_\rho < \lambda_{c+1}
< \dots < \lambda_{k-1} <\lambda_{k+1}<\dots < \lambda_a \hspace{5mm} \lambda_k
\choose \mu_1< \dots < \mu_b}.
\end{array}$$
Note that the interpolated symbol part has dimension $(-1)^{a-k}$, multiplied by the dimension of the 
genuine symbol
\beg{LegalSymbolFinally}{{\lambda_1< \dots < \lambda_c< N'_\rho< \lambda_{c+1} < \dots < \lambda_a
\choose \mu_1< \dots < \mu_b}.}
Therefore, at $j^{(k)}$, this contributes a term
\beg{FinalktheTermRecursiveAltIndSummands}{\begin{array}{c}
r^{\text{Sp} (V)} [(\phi^\pm (s)), \psi^\pm (u^{(k)}),  \epsilon (s) disc(B)]\otimes\\[1ex]
 \text{Ind}_{P_{j^{(k)}}} (r^{O_{2(m-j^{(k)})+1} (\F_q)}[(s^{(k)}), u^{(k)}]^{\pm 1}),
\end{array}}
which has terms that can cancel recursively (e.g. $k=1$ completely cancels \rref{IllegalSymbolOriginal}, though it introduces new summands that are illegal if $a>1$),
since switching coordinates in interpolated Lusztig symbols
gives a change in the sign of dimension.
At $k=a$, we have the legal symbol \rref{LegalSymbolFinally}.

Putting together the terms \rref{IllegalSymbolOriginal} and \rref{FinalktheTermRecursiveAltIndSummands}
for $c+1 \leq k \leq a$, this gives the final genuine term, which is
$\eta^V_{W,B} (r^{O(W[-j^{(a)}],B[-j^{(a)}] ) } [(s^{(a)}), u^{(a)}]^{\pm 1})$, 
tensored with a coefficient of the $O(W,B)$-representation given by the alternating sum
$$\bigoplus_{k=c}^{a} (-1)^{a-k} \text{Ind}_{P_{j_{(k)}}}  (r^{O(W[-j^{(k)}],B[-j^{(k)}] ) }[(s^{(k)}), u^{(k)}]^{\pm 1})$$
(writing $j^{(0)} = j$, $s^{(0)} =s$, $u^{(0)} = u$).
By definition, this $O(W,B)$-representation is precisely the proposed alternating sum 
$$\mathcal{A}_{j_{(a)}} (r^{O(W[-j^{(a)}],B[-j^{(a)}] ) }[(s^{(a)}), u^{(a)}]^{\pm 1} , N'_{r^{O(W[-j^{(a)}],B[-j^{(a)}] ) }[(s^{(a)}), u^{(a)}]^{\pm 1}}).$$
Therefore, in the final decomposition of $Res_{\text{Sp}(V)\times O(W,B)} (\boldsymbol \omega[V\otimes W])$
in the genuine representation categories $Rep (\text{Sp} (V))$, $Res (O(W,B))$, we obtain
that the illegal constructions should be taken to be $0$, in exchange for replacing $\text{Ind}_{P_{j^{(a)}}}$
in the corresponding final legal level by $\mathcal{A}_{j^{(a)}}$, exactly as claimed.

We also note that Theorems \ref{ExplicitTheoremSymp}, \ref{ExplicitTheoremOrtho}
can be checked elementarily, using the global dimension formula
calculated in \cite{TotalHoweI, TotalHoweII} and observations about how
the dimensions of endomorphism algebras of tensor powers of oscillator representations
degenerate in the degrees corresponding to metastable reductive dual pairs.

\subsection{An ``inductive" construction}\label{InductiveConstruction}
We now describe how the eta and zeta correspondence can be used
to construct families of irreducible representations of the finite symplectic and orthogonal
groups recursively. This process re-expresses any irreducible representation
to a chain of eta and (possibly signed) zeta correspondences, applied to a
certain ``terminal" representations which can not be expressed as the eta or zeta correspondence applied
to a representation of a group of lower rank.
A terminal representations is precisely one such that the semisimple component $(s)$
of its classification data has a centralizer only involving factors of type
$A$ or ${}^2 A$ (in which case, its centralizer is automatically connected).

For a fixed choice of symplectic space $V$, let us consider the
eta correspondences
$$\eta_V^{W,B}: \widehat{O(W,B)} \rightarrow \widehat{\text{Sp}(V)} \cup \{0\},$$
for every choice of orthogonal space $(W,B)$ such that $(\text{Sp}(V), O(W,B))$
forms a reductive dual pair in the symplectic metastable range. 
Every irreducible representation $\rho \in \text{Sp}_{2N} (\F_q)$
is contained as a summand of a tensor product of oscillator representations
\beg{SpsideDegreenTensor}{{ \omega}_{a_1}[V] \otimes { \omega}_{a_2}  [V] \otimes \dots \otimes { \omega}_{a_n} [V],}
of degree $n \leq 2N$,
which we can further interpret as the restriction of an oscillator representation $\omega [V\otimes W]$
for an $n$-dimensional orthogonal space $W$ with respect to a form $B$ given by the diagonal matrix
with entries $a_1, \dots, a_n$, along the inclusion
$\text{Sp}(V)\subseteq \text{Sp}(V) \times \text{O}(W,B) \subseteq \text{Sp}(V\otimes W)$.
In particular, we find that the union of the non-zero images of the eta correspondences makes up the
whole set of irreducible representation
$$\widehat{\text{Sp}(V)} = \coprod_{\tiny \begin{array}{c}
(W,B), \\ \text{dim} (W) \leq 2N\end{array}} \text{Im} (\eta^{W,B}_V) \smallsetminus \{0\}.$$

Therefore, knowing $\pi \in \widehat{O(W,B)}$ satisfying
$\eta^V_{W,B} (\pi) = \rho$, we may consider the idempotent
\beg{TrueIdempotentConcreteInd}{\frac{dim (\pi) }{|O(W,B)|} \sum_{g\in O(W,B)} \chi_{\pi} (g)^{-1} \cdot g \in \C O(W,B)}
in 
$$\text{End}_{\text{Sp}(V)} (\omega [V\otimes W]) =\text{End}_{\text{Sp} (V)} (\bigotimes_{i=1}^n { \omega} [V]_{a_i}),$$
where $\chi_\pi: O(W,B) \rightarrow \C^\times$
denotes the character corresponding to $\pi$. The image of \rref{TrueIdempotentConcreteInd}
in \rref{SpsideDegreenTensor} recovers $\rho$.
Outisde of the stable range, 
the constructions of the reflection elements are readily interpolated.
In the metastable range, if $\rho \otimes \pi$ is in the top part of
the restricted oscillator representation, then the corresponding idempotent
\rref{TrueIdempotentConcreteInd} survives and is unaltered by semisimplification.


On the other hand, we note that for every irreducible representation $\pi $ of $SO(W,B)$,
at least one irreducible representation obtained as a summand of the induction of $\pi $ from $SO(W,B)$
to $O(W,B)$ occurs in a tensor product
$$(\epsilon (det) \otimes \C W)^{\otimes N}$$
of degree $N \leq h_W$.
Again, for any irreducible representation $\rho \in \widehat{\text{Sp}(V)}$, the summand
$\zeta^{V}_{W,B} (\rho)$ can be obtained as the image of the idempotent
\beg{}{\frac{dim (\rho) }{|\text{Sp}(V)|} \sum_{g\in \text{Sp}(V)} \chi_{\rho} (g)^{-1} \cdot g \in \C \text{Sp}(V)}
in 
$$\text{End}_{O(W,B)} (\omega [V\otimes W]) = \text{End}_{O(W,B)} ((\epsilon(det) \otimes \C W)^{\otimes N}),$$
again where
$\chi_\rho: \text{Sp}(V) \rightarrow \C^\times$
denotes to character corresponding to $\rho$.
All representations of $O(W,B)$ can then be obtained by
tensoring the representations appearing in the image of the zeta correspondences
by a sign representation.

In particular, one can attempt
to attempt to produce a given representation of a symplectic group
$\text{Sp}(V)$ or an orthogonal group $O(W,B)$ as a chain of alternating eta and zeta correspondences
(and tensors with sign representations) applied to a simpler representation of a
smaller symplectic or orthogonal group. However, we note that both the eta
and zeta correspondences can never alter any input classification 
to introduce eigenvalues other than $1$ or $-1$ to the semisimple component of the data.
Therefore, we find that any irreducible representation

\begin{definition}
Consider an irreducible representation $\rho$ of a symplectic group or orthogonal group.
\begin{enumerate}
\item For $\rho \in \widehat{\text{Sp}_{2N} (\F_q)}$, we call $\rho$ {\em terminal} if in its
classification data,
the semisimple compotent $(s) \in SO_{2N+1} (\F_q) = (\text{Sp}_{2N} (\F_q))^*$ has $1$
as an eigenvalue of multiplicity one and no $-1$ eigenvalues.

\item For $\rho \in \widehat{O_{2m+1} (\F_q)}$, we call $\rho$ {\em terminal} if in
the classification data of the restriction of $\rho$ to the special orthogonal group
$SO_{2m+1} (\F_q)$,
the semisimple compotent of the data
$(s) \in \text{Sp}_{2m} (\F_q) = (SO_{2m+1} (\F_q))^*$ has no $1$ or
$-1$ eigenvalues.

\item For $\rho \in \widehat{O_{2m}^\pm (\F_q)}$, we call $\rho$ {\em terminal} if
in the classification data of the restriction of $\rho$ to the special
orthogonal group $SO_{2m}^\pm (\F_q)$,
the semisimple compotent $(s) \in SO_{2m}^\pm (\F_q) = (SO_{2m}^\pm (\F_q))^*$ 
no $1$ or $-1$ eigenvalues.

\end{enumerate}

\end{definition}

We may then produce any irreducible representation of a symplectic or orthogonal group
by applying alternating eta and zeta correspondences and possibly tensor products with sign representations
during steps of the construction in $\widehat{O(W,B)}$.
Suppose we are given a representation whose Lusztig data has semisimple component
involving $1$ as an eigenvalue contributing a factor of the centralizer of rank $p$
and $-1$ as an eigenvalue contributing a factor of the centralizer of rank $\ell$.
Say the corresponding factors of the unipotent component of the classification data
are symbols of the form
$${\lambda_1< \dots < \lambda_a \choose \mu_1< \dots < \mu_b} \otimes {\lambda_1'< \dots < \lambda_{a'}' \choose \mu_1'< \dots < \mu_{b'}'}.$$
Then (disregarding steps where we tensor with sign representations
to obtain $O(W,B)$-representations which are only obtained up to sign
in the zeta correspondence), $a+b+a'+b'$ steps of applying eta and zeta correspondences
are needed.


\subsection{Character computations}\label{CharaAlgorithmSubSect}
In this subsection, we use the ``inductive" construction to describe an algorithm
for computing characters of an arbitrary given irreducible representation.
There are three major steps to be considered: The calculation of the character
of the top part of an oscillator representation, how to use this top part's character
to obtain the character of the representation obtained from
applying the eta or zeta correspondence to a given input representation,
and the calculation of the characters of the terminal representations.

The action of the symplectic group generators on our model of the oscillator rerpesentation are
as follows:
\beg{ActionOfGeneratorsCharaaSubS}{\begin{array}{c}
\displaystyle \omega_a \begin{pmatrix} 1 & A \\ 0 & 1\end{pmatrix} (v) = \psi (\frac{a}{2} A(v,v)) \cdot (v)\vspace{1.5mm}\\
\displaystyle \omega_a \begin{pmatrix} 0 & B \\ -B^{-1} & 0\end{pmatrix} (v) = \sum_{w \in V_-}
\frac{\psi (aB(v,w))}{\displaystyle \sum_{u \in V_-} \psi (\frac{a}{2}B(u,u))} \cdot (w)\\
\omega_a \begin{pmatrix} (C^T)^{-1} & 0 \\ 0 & C\end{pmatrix} (v) = \epsilon_{\F_q} (-1) \cdot (Cv).
\end{array}}
The action of any symplectic group element can then be deduced by expressing it as a
product of these standard generators.
In particular, in principle, we can form a closed expression for 
the character of the oscillator representation at any symplectic group element.
For example, in the case of the standard generators, we find character values
$$\chi_{\omega_a} \begin{pmatrix} 1 & A \\ 0 & 1\end{pmatrix} = (-1)^{n (\ell+1)} \text{disc}(\frac{A}{2}) (\epsilon_{\F_q} (-1) q)^{n/2}$$
$$\chi_{\omega_a} \begin{pmatrix} 0 & B \\ -B^{-1} & 0\end{pmatrix} = (\epsilon_{\F_q} (-2))^n $$
$$\chi_{\omega_a} \begin{pmatrix} (C^T)^{-1} & 0 \\ 0 & C\end{pmatrix} = \epsilon_{\F_q} (\text{det}(C)) \cdot q^{\text{dim} (\text{ker} (C-I))}.$$

\vspace{3mm}

Now, while the expression for the restricted
oscillator representation in terms of the extended eta correspondence and
the alternating sums given in
Theorems \ref{ExplicitTheoremSymp} and \ref{ExplicitTheoremOrtho} are more ideal
from the point of view of a genuine decomposition statement, it is
inefficient to organize the characters of the alternating sums when doing concrete
computations.

Instead, for the purpose of computing characters,
let us introduce the {\em virtual eta} and {\em virtual zeta correspondences}:
For any reductive dual pair $(\text{Sp}(V), \text{O}(W,B))$ in the symplectic metastable range,
we define the virtual eta correspondence as the function from the set of irreducible
$O(W,B)$-representations to the Grothendieck group on the category of representations of $\text{Sp}(V)$
$$\boldsymbol\eta^{W,B}_V : \widehat{O(W,B)} \rightarrow K(Rep( \text{Sp}(V)))$$
so that the restriction of the oscillator representation to $\text{Sp}(V) \times O(W,B)$ decomposes as
\beg{VirtualDecompCharaSubsEta}{ \bigoplus_{k=0}^{h_W}
\bigoplus_{\pi \in \widehat{O(W[-k],B[-k])}} \boldsymbol\eta^{W[-k],B[-k]}_V (\pi) \otimes Ind_{P^k_{O(W,B)}} (\pi^-).s}
Similarly, for any reductive dual pair $(\text{Sp}(V), O(W,B))$ in the orthogonal metastable range,
we define the virtual zera correspondence as the function from the set of irreducible
$\text{Sp}(V)$-representations to the Grothendieck group on the category of $O(W,B)$-representations
$$\boldsymbol\zeta_{W,B}^V : \widehat{\text{Sp}(V)} \rightarrow K(Rep( O(W,B)))$$
so that the restriction of the oscillator representation to $\text{Sp}(V) \times O(W,B)$ decomposes as
\beg{VirtualDecompCharaSubsZeta}{\bigoplus_{k=0}^{dim(V)/2}
\bigoplus_{\rho \in \widehat{\text{Sp}(V[-k])}} Ind_{P^k_{\text{Sp}(V)}} (\rho^-) \otimes \boldsymbol\zeta_{W,B}^{V[-k]} (\rho).}

Specifically, if $(\text{Sp}(V), O(W,B))$ is in the symplectic (resp. orthogonal)
stable range, we take $\boldsymbol\eta^{W,B}_V$ (resp. $\boldsymbol \zeta^V_{W,B}$)
to be equal to $\eta^{W,B}_V$ (resp. $\zeta^V_{W,B}$).
If $(\text{Sp}(V), O(W,B))$ is in the symplectic metastable range,
we take $\boldsymbol\eta^{W,B}_V = \Psi \circ \eta^{t=N}_{W,B}$,
where $\eta^{t=N}_{W,B}$ denotes the interpolated eta correspondence applied at $t=N$,
and $\Psi$ is defined as in \rref{PsiGrothDefn}. 
Similarly, if $(\text{Sp}(V), O(W,B))$ is in the orthogonal metastable range,
we take
$\boldsymbol\zeta_{W,B}^V = \Psi \circ \zeta^{t=N}_{W,B}$,
where $\zeta^{t=N}_{W,B}$ denotes the interpolated eta correspondence applied at $t=N$.

Concretely, for an input representation $\rho$ of $\boldsymbol \eta^{W,B}_V$ or $\boldsymbol \zeta^V_{W,B}$,
we compute the output by attempting the appropriate case of the construction
described in Section \ref{StatementSection}. Say the
symbol ${\lambda_1< \dots < \lambda_a \choose \mu_1< \dots < \mu_b}$ corresponds
to the alterable factor in the unipotent
component of $\rho$'s classification data. 
If the output is ``constructible," meaning that the genuine eta or zeta correspondence
$\eta^{W,B}_V (\rho)$ or $\zeta^V_{W,B} (\rho)$ is a non-zero representation, we put
$$\boldsymbol \eta^{W,B}_V (\rho) := \eta^{W,B}_V (\rho) \text{ or } \boldsymbol \zeta^{W,B}_V (\rho) := \zeta^{W,B}_V (\rho) .$$
Suppose now that it is not construcible, meaning that (possibly after switching rows so that the original
construction proposes adding $N_\rho'$ to the first row of the alterable symbol) for some
$i\leq a$, we have $\lambda_{i-1} < N_\rho' \leq \lambda_{i}$ (considering $\lambda_0 = 0$).
If $N_\rho' = \lambda_{i}$, then we put $\boldsymbol \eta^{W,B}_V (\rho)$
or $\boldsymbol \eta^{W,B}_V (\rho)$ to be the zero representation.
Then
we complete the construction with the factor
$${\lambda_1< \dots < \lambda_{i-1} < N_\rho'< \lambda_{i} < \dots < \lambda_a \choose \mu_1< \dots < \mu_b}$$
in place of where we would have used the symbol with $N_\rho'$ concatenated at the end
of the first row of the alterable symbol. This constructs a genuine representation,
which we multiply by the coefficient $(-1)^{a-i}$
in $K(Rep (\text{Sp}(V))$ or $K (Rep (O(W,B)))$. 
We set the resulting (possibly virtual) representation
to be $\boldsymbol \eta^{W,B}_V (\rho)$ or $\boldsymbol \zeta^V_{W,B} (\rho)$.
In particular, the virtual eta and zeta correspondences
always output a sign times an irreducible representation.
We note that for a representation which is sent to zero by the original
eta or zeta correspondence, the virtual eta or zeta correspondence outputs a sign
of an irreducible representation already in the image of  the genuine eta or zeta correspondence
for a smaller orthogonal or symplectic
group (respectively). 
We consider the character of
a virtual representation to be the linear combination (with possible negative coefficients)
of the characters of its genuine terms.

Now let us consider the virtual representations
\beg{VirtualEtaTop}{\widetilde{\omega^{\text{top}}} [V\otimes W] = \bigoplus_{\pi \in \widehat{O(W,B)}}
\boldsymbol\eta^{W,B}_V (\pi) \otimes \pi}
of $\text{Sp}(V) \times O(W,B)$ for every reductive dual pair in the symplectic metastable range,
and the virtual representations
\beg{VirtualZetaTop}{\widetilde{\omega^{\text{top}'}} [V\otimes W] = \bigoplus_{\rho \in \widehat{\text{Sp}(V)}}
\rho \otimes \boldsymbol\zeta_{W,B}^V (\rho) }
of $\text{Sp}(V) \times O(W,B)$ for every reductive dual pair in the orthogonal metastable range.
These virtual ``top parts" are key to understanding the relationship of an irreducibe representation's
character and the character of the representation obtained by applying the (virtual) eta or zeta correspondence.

First, we must compute the character values of \rref{VirtualEtaTop} and \rref{VirtualZetaTop}.
This can be done recursively by comparing the top parts to the full oscillator representations at each level,
similarly to how the dimensions of the stable top parts were computed in
\cite{TotalHoweI}
In the symplectic metastable case, we may re-write the decomposition of the restriction
of $\omega [V\otimes W]$ to $\text{Sp}(V) \times O(W,B)$ described in \rref{VirtualDecompCharaSubsEta} as 
$$\bigoplus_{k=0}^{h_W} \text{Ind}_{\text{Sp}(V) \times P_k^{O(W,B)}}^{\text{Sp}(V) \times O(W,B)} (\widetilde{\omega^\text{top}} [V\otimes W[-k]]^-)$$
and similarly, in the orthogonal metastable case, we may re-write the decomposition restriction of
$\omega [V\otimes W]$ to $\text{Sp}(V) \times O(W,B)$ described in \rref{VirtualDecompCharaSubsZeta} as 
$$\bigoplus_{k=0}^{dim (V)/2} \text{Ind}_{P_k^{\text{Sp}(V)}\times O(W,B)}^{\text{Sp}(V) \times O(W,B)} (\widetilde{\omega^{\text{top}'}} [V[-k]\otimes W]^-).$$

Now we recall that to compute the character of an induction $\text{Ind}_H^G (\rho)$
for a subgroup $H \subseteq G$ and an $H$-representation, for a given $g\in G$, we have
$$\chi_{\text{Ind}_H^G (\rho)} (s) = \frac{1}{|H|}\sum_{\tiny \begin{array}{c} x \in G\\[-0.5ex] xgx^{-1} \in H\end{array}} \chi_\rho (xgx^{-1}).$$
Hence, we obtain that for every choice of $g \in \text{Sp}(V)$, $h \in O(W,B)$ for $(\text{Sp}(V), O(W,B))$ in the
symplectic metastable range,
the sum
$$\sum_{k=0}^{h_W} \frac{1}{|\text{Sp}(V) \times P_k^{O(W,B)}|} \hspace{-7.5mm}
\sum_{\tiny \begin{array}{c} x \in \text{Sp}(V) \times O(W,B) \\[-0.5ex] x(g\otimes h)x^{-1} \in \text{Sp}(V) \times  P_k^{O(W,B)}
\end{array}} \hspace{-7.5mm} \chi_{\widetilde{\omega^{\text{top}}} [V\otimes W[-k]]^-} (x (g\otimes h)x^{-1})$$
is equal to the character value $\chi_{\omega [V\otimes W]} (g\otimes h)$, which for
any explicit choice of $g$ and $h$ could be calculable using \rref{ActionOfGeneratorsCharaaSubS}.
Similarly, for every $g \in \text{Sp}(V)$, $h \in O(W,B)$, for $(\text{Sp}(V), O(W,B))$ in the
orthogonal metastable range, the sum
$$\sum_{k=0}^{dim(V)/2} \frac{1}{|O(W,B) \times P_k^{\text{Sp}(V)}|} \hspace{-9.5mm}
\sum_{\tiny \begin{array}{c} x \in \text{Sp}(V) \times O(W,B) \\[-0.5ex] x(g\otimes h)x^{-1}\hspace{-0.5mm} \in P_k^{\text{Sp}(V)}\hspace{-0.5mm}\times O(W,B)
\end{array}} \hspace{-10mm} \chi_{\widetilde{\omega^{\text{top}'}} [V[-k]\otimes W]^-} (x (g\otimes h)x^{-1})$$
is equal to the character value $\chi_{\omega [V\otimes W]} (g\otimes h)$.
This gives a linear system of equations that can then be used to recursively calculate the virtual characters
of the representations $\widetilde{\omega^{\text{top}}} [V\otimes W]$ and
$\widetilde{\omega^{\text{top}'}} [V\otimes W]$.

Suppose we have calculated the virtual characters
of $\widetilde{\omega^{\text{top}}} [V\otimes W]$ (resp.
$\widetilde{\omega^{\text{top}'}} [V\otimes W]$) for
pairs $(\text{Sp}(V), O(W,B))$ in the symplectic (resp. orthogonal) metastable range.
We then use their values to deduce the effect of the extended eta (resp.
zeta) correspondence on irreducible representations:
\begin{proposition}\label{CharacterEtaZetaRecipeProp}
Fix a reductive dual pair $(\text{Sp}(V), O(W,B))$. 
\begin{enumerate}
\item Suppose the pair is
in the symplectic metastable range, and fix an irreducible representation $\pi$
of $O(W,B)$ such that $\eta^{W,B}_V (\pi)$ is non-zero. Then, for any $g\in \text{Sp}(V)$,
the value of the character associate to $\eta^{W,B}_V (\pi)$ as $g$ can be computed as
\beg{EtaCorrMetastabCharaRecipe}{\chi_{\eta^{W,B}_V (\pi)} (g) = \sum_{(h) \in O(W,B)} \frac{|(h)|}{|O(W,B)|} \cdot
\overline{\chi_\pi (h)} \cdot \chi_{\widetilde{\omega^{\text{top}}} [V\otimes W]} (g\otimes h).}
\vspace{2mm}

\item Suppose the pair is in the orthogonal metastabl range, and fix an irreducible representation $\rho$
of $\text{Sp}(V)$ such that $\zeta_{W,B}^V (\rho)$ is non-zero. Then, for any $h\in O(W,B)$,
the value of the character associate to $\zeta_{W,B}^V (\pi)$ as $h$ can be computed as
\beg{ZetaCorrMetastabCharaRecipe}{\chi_{\zeta_{W,B}^V (\rho)} (h) = \sum_{(g) \in \text{Sp}(V)} \frac{|(g)|}{|\text{Sp}(V)|} \cdot
\overline{\chi_\rho (g)} \cdot \chi_{\widetilde{\omega^{\text{top}'}} [V\otimes W]} (g\otimes h).}
\end{enumerate}
\end{proposition}

\noindent In fact, out proof of this result also shows that \rref{EtaCorrMetastabCharaRecipe} is equal to 
\beg{EtaCorrCharaRecipeGenuineTopPart}{\sum_{(h) \in O(W,B)} \frac{|(h)|}{|O(W,B)|} \cdot
\overline{\chi_\pi (h)} \cdot \chi_{\omega^{\text{top}} [V\otimes W]} (g\otimes h)}
for choices of $\pi$ which survive the extended eta correspondence.
The expressions \rref{EtaCorrMetastabCharaRecipe}
and \rref{EtaCorrCharaRecipeGenuineTopPart} are different only for irreducible representations
$\pi \in \widehat{\text{O}(W,B)}$ with $\eta^V_{W,B} (\pi) = 0$, in which case
\rref{EtaCorrMetastabCharaRecipe} gives the virtual character
$\chi_{\boldsymbol\eta^{W,B}_V}(g)$ (which will be a sign times a genuine irreducible charcter),
while the sum
\rref{EtaCorrCharaRecipeGenuineTopPart} gives $0$.

\begin{proof}[Proof of Proposition \ref{CharacterEtaZetaRecipeProp}]
This statement follows from elementary manipulations of character theory.
Let us fix an order of the conjugacy classes of $O(W,B)$ and an order of
its irreducible characters. We may then consider the square matrix
obtained from the character table of $O(W,B)$ written according to these orderings.
Denote this matrix by $\text{ct}(O(W,B))$. By orthogonality of characters, recall that
the inverse of $\text{ct} (O(W,B))$ can be expressed as a diagonal matrix consisting of
the fractions $|(h)|/|O(W,B)|$ of the order of a conjugacy class $(h)$ of $O(W,B)$ divided by the
group order, multiplied on the right by the conjugate of the transpose of $\text{ct} O(W,B)$
\beg{CtInverse}{ \text{ct} (O(W,B))^{-1} = \begin{pmatrix}
\frac{|(h_1)|}{|O(W,B)|} & 0 & \dots \\
0 & \frac{|(h_2)|}{|O(W,B)|} & \\
\vdots & & \ddots
\end{pmatrix} 
\cdot \overline{\text{ct} (O(W,B))}^T.}

Now fix a group element $g \in \text{Sp}(V)$. Then, for any $h \in O(W,B)$, by definition \rref{VirtualEtaTop},
we find that the virtual character of the representation $\widetilde{\omega^\text{top}}$
applied to $g\otimes h \in \text{Sp}(V) \times O(W,B)$ is
$$\chi_{\widetilde{\omega^\text{top}}[V\otimes W]} ( g\otimes h) =
\sum_{\pi \in \widehat{O(W,B)}} \chi_{\boldsymbol \eta^{W,B}_V (\pi)} (g) \cdot \chi_{\pi} (h).$$
In terms of matrices, this can be interpreted as the statement that the column
of character values $\chi_{\widetilde{\omega^\text{top}}[V\otimes W]} (g\otimes h)$
(varying conjugacy classes $(h)$) is obtained by multiplying the transpose of the character
table $\text{ct} O(W,B)$ by the column of character values $\chi_{\boldsymbol \eta^{W,B}_V (\pi)} (g)$
(varying $\pi \in \widehat{O(W,B)}$):
$$ \begin{pmatrix}
\chi_{\widetilde{\omega^\text{top}}[V\otimes W]} ( g\otimes h_1)\\
\chi_{\widetilde{\omega^\text{top}}[V\otimes W]} ( g\otimes h_2) \\
\vdots \\
\end{pmatrix} = \text{ct} (O(W,B))^T \cdot \begin{pmatrix}
\chi_{\boldsymbol \eta^{W,B}_V (\pi_1)} (g)\\
\chi_{\boldsymbol \eta^{W,B}_V (\pi_2)} (g)\\
\vdots \\
\end{pmatrix}.
$$
Hence, by applying \rref{CtInverse}, we can calculate the column of character values
$\chi_{\boldsymbol \eta^{W,B}_V (\pi)} (g)$
as
$$\overline{\text{ct} (O(W,B))} \cdot \begin{pmatrix}
\frac{|(h_1)|}{|O(W,B)|} & 0 & \dots \\
0 & \frac{|(h_2)|}{|O(W,B)|} & \\
\vdots & & \ddots
\end{pmatrix} \cdot \begin{pmatrix}
\chi_{\widetilde{\omega^\text{top}}[V\otimes W]} ( g\otimes h_1)\\
\chi_{\widetilde{\omega^\text{top}}[V\otimes W]} ( g\otimes h_2) \\
\vdots \\
\end{pmatrix}$$
Multiplying this out, we find that each character value $\chi_{\boldsymbol\eta_V^{W,B} (\pi)} (g)$
can be calculated as the sum
$$
\chi_{\boldsymbol\eta^{W,B}_V (\pi)} (g) = \sum_{(h) \in O(W,B)} \frac{|(h)|}{|O(W,B)|} \cdot
\overline{\chi_\pi (h)} \cdot \chi_{\widetilde{\omega^{\text{top}}} [V\otimes W]} (g\otimes h),$$
matching the right hand side of \rref{EtaCorrMetastabCharaRecipe}
The claim then follows, since for $\pi \in \widehat{O(W,B)}$ such that
$\eta^{W,B}_V (\pi)$ is non-zero, we have
$\boldsymbol \eta^{W,B}_V (\pi) = \eta^{W,B}_V (\pi)$. 
\end{proof}


\section{The Gurevich-Howe rank conjecture}\label{GurevichHoweSection}

Finally, as advertised in the introduction, the purpose of this section is
to apply our calculation of the eta correspondence to prove Theorem \ref{MainMatchRes},
verifying the rank conjecture of S. Gurevich and R. Howe, which predicts
the equality of $U$-rank and tensor-rank
for every representation of a symplectic group $\text{Sp}_{2N} (\F_q)$ not attaining top possible
$U$-rank $N$.

First, in Subsection \ref{BackgroundSect}, we recall the definitions of $U$-rank and tensor rank in more detail
and recall the results of Gurevich and Howe in \cite{HoweGurevich, HoweGurevichBook}
which reduce Theorem \ref{MainMatchRes}
to a statement that all irreducible representations with tensor rank larger than $N$
have $U$-rank equal to $N$ (see Proposition \ref{AllUnstableTopProp}).
Next, in Subsection \ref{IndSect}, we use our explicit description of the 
extended eta correspondence 
to obtain an induction relation (see Proposition \ref{IndEtaProp})
between tensor ranks from an analogue of the Pieri rule,
which reduces the claimed statement to
lower tensor rank, again.

\subsection{$U$-rank and the eta correspondence}\label{BackgroundSect}
First we recall the definitions of $U$-rank and tensor rank
given in \cite{HoweGurevich, HoweGurevichBook}. We begin with $U$-rank.
Consider a symplectic group $\text{Sp}_{2N }(\F_q)$.
The Siegel unipotent subgroup is defined as
$$U_N = \{
\begin{pmatrix}
I & A\\
0 & I 
\end{pmatrix}
\mid A \in M_{N \times N} (\F_q) \text{ symmetric} \} \subseteq \text{Sp}_{2N} (\F_q).$$
Note that, as a group, $U_N$ is isomorphic to the abelian group of 
symmetric $N\times N$ matrices, with respect to addition.
In particular, we may fix an identification of $U_N$ with its Pontrjagin dual $U_N^*$
in the standard way, i.e. by fixing a non-trivial additive character $\chi_0: \F_q \rightarrow \C^\times$
and identifying each element of $U_N$ corresponding to a symmetric matrix $A$
with the character
$$\begin{array}{c}
 U_N \rightarrow \C^\times\\[1ex]
\begin{pmatrix}
I & B\\
0 & I 
\end{pmatrix} \mapsto \chi_0 (tr(AB)).
\end{array}$$

Gurevich and Howe \cite{HoweGurevich, HoweGurevichBook}
then define the {\em $U$-rank} of an $\text{Sp}_{2N} (\F_q)$-representation
$\rho$ as the maximal rank of a character appearing in its restriction to $U_N$:
\beg{UrkDefn}{rk_U (\rho) := \text{max}\{ rk (\chi) \mid \chi \in  U_N^* \text{ and } \chi \subseteq
Res_{U_N} (\rho) \},}
where for a character $\chi \in U_N^*$, its rank is defined to be the matrix rank of the symmetric
$N \times N$-matrix specifying its
corresponding element of $U_N$.

\vspace{3mm}

On the other hand, the tensor rank of a $\text{Sp}_{2N} (\F_q)$-representation is defined
according to the oscillator representations.
Recall that each oscilaltor representation of a symplectic group $\text{Sp}(V)$
decomposes into two irreducible summands
$$\omega_a [V] = \omega_a^+ [V] \oplus \omega_a^- [V]$$
(from the perspective of the eta correspondence, $\omega_a^\pm [V]$ is obtained by
applying the eta correspondence to the representations $(\pm 1)$ of $\mu_2 = O_1(\F_q)$.
These pieces of the oscillator representation are the smallest non-trivial
irreducible representations of $\text{Sp}(V)$, and they each have $U$-rank $1$.
The {\em tensor rank} of a representation $\rho$ is then defined as the minimal degree $n$
such that every
irreducible component of $\rho$ appears in a tensor product of less than or equal to $n$ oscillator
representations:
$$\begin{array}{c}
rk_\otimes (\rho ):= min\{n \mid \text{for }  \pi \in \widehat{\text{Sp} (V)}, \; \pi \subseteq \rho ,
\text{ there exists } \\[1ex]
 m \leq n, \; a_1, \dots , a_m \in \F_q^\times \text{ with }
\pi \subseteq \omega_{a_1} [ V] \otimes \dots \otimes \omega_{a_m} [V]\}
\end{array}$$

\vspace{1mm}

Now recalling again that a degree $n$ tensor product
$$\omega_{a_1} [V] \otimes \dots \otimes \omega_{a_n} [V]$$
can be considered as the restirction of the oscillator representation of 
a larger symplectic group $\text{Sp}(V\otimes W)$
along the inclusion
$$\text{Sp}(V) \hookrightarrow \text{Sp}(V) \times O(W,B) \hookrightarrow \text{Sp}(V\otimes W),$$
where $(W,B)$ denotes an $n$-dimensional $\F_q$-space with a non-degenerate symplectic form $B$
corresponding to the diagonal matrix with entries $a_1 , \dots , a_n$, we see that understanding
the restricted oscillator representation $Res_{\text{Sp}(V)\times O(W,B)} (\omega [V \otimes W])$
is key to. In particular, the results of this paper for pairs $(\text{Sp}(V), O(W,B))$
in the symplectic stable or metastable range explicitly classify the
irreducible representations of each tensor rank $0 \leq rk_\otimes \leq 2N$.

\vspace{1mm}

In \cite{HoweGurevich, HoweGurevichBook}, Gurevich and Howe
also found a connection between the restricted oscillator representations
$Res_{\text{Sp}(V)\times O(W,B)} (\omega [V \otimes W])$ and U-rank,
which was one of their original motiviation for defining the eta correspondence:

The original statement describing the eta correspondence given
in, say, Theorem 4.3.3 of \cite{HoweGurevichBook}, is that for choices of $V$ and $(W,B)$
in the symplectic stable range (which we recall means $dim (W) \leq dim (V)/2$,
there is a system of injections
$$\eta^V_{W,B}: \widehat{O(W,B)} \hookrightarrow \widehat{\text{Sp}(V)}$$
(we omit the subscript when the source is determined)
such that for every irreducible representations $\rho \in \widehat{O(W,B)}$,
the tensor product $\rho \otimes \eta^V_{W,B} (\rho)$ is a summand of $Res_{\text{Sp}(V)\times O(W,B)} (\omega [V \otimes W])$, and 
\beg{EtaAttainsURank}{rk_U (\eta^V_{W,B}(\rho)) = dim (W).}
Further, every other $\pi \in \widehat{\text{Sp}(V)}$
such that $\rho \otimes \pi$ appears in the restricted oscillator representation has strictly lower
$U$-rank.
(Note that though Theorem 4.3.3 of \cite{HoweGurevichBook} does
not include the case of $dim (W) = dim (V)/2$, 
the result still applies to this case as described in
Remark 4.3.6.)

\vspace{3mm}

First, we note that the results of \cite{HoweGurevich, HoweGurevichBook}
immediately imply the agreement of tensor- and $U$-rank in cases
covered by the symplectic stable range:

\begin{corollary}\label{StableTrivCase}
For irreducible $\text{Sp}_{2N} (\F_q)$-representations $\rho$ of tensor rank
$\leq N$, the notions of rank coincide:
$$rk_\otimes (\rho) = rk_U (\rho).$$ 
\end{corollary}

\vspace{3mm}

Therefore, it only remains to prove the following

\begin{proposition}\label{AllUnstableTopProp}
Consider an irreducible representation $\rho$ of a symplectic group $\text{Sp}_{2N} (\F_q)$ obtained
first in the restriction of an oscillator representation to an
unstable reductive dual pair, meaning
$$N < rk_\otimes (\rho) \leq 2N.$$ 
Then $\rho$ attains top $U$-rank
$$rk_U (\rho) = N.$$
\end{proposition}

\subsection{An induction rule and concluding Theorem \ref{MainMatchRes}}\label{IndSect}

As described in the previous subsection,
to prove Proposition \ref{AllUnstableTopProp}, we need more explicit information about the
symplectic group representations of each tensor rank, which can be obtained from
Theorem \ref{ExplicitTheoremSymp}.
In particular, in the decomposition \rref{EtaThmDecompFull},
we may further restrict to $\text{Sp}(V)$-represenations by treating the 
coefficient $O(W,B)$-representations
as multiplicity spaces, obtaining a classification of the irreducible $\text{Sp}(V)$-represntations
of tensor rank $rk_\otimes = r$ for each $0 \leq r \leq 2N$
as precisely those constructed
in the image of an eta correspondence
$$\eta^V_{W,B} : \widehat{O(W,B)} \rightarrow \widehat{\text{Sp}(V)} \cup \{0\},$$
for one of the two non-equivalent choices of $(W,B)$ with dimension $r$.

The key step we use to conclude Propostion \ref{AllUnstableTopProp} and Theorem \ref{MainMatchRes} is the following
result, which gives a relationship between the different rank layers of the eta correspondence,
according to parabolic induction:

\begin{proposition}\label{IndEtaProp}
Fix a $\F_q$-vector space $W$ with symmetric bilinear form $B$.
Consider symplectic spaces $V$, $U$ of dimension $2N\leq 2M$ respectively, such that
both reductive dual pairs $(\text{Sp}(V), O(W,B))$ and $(\text{Sp}(U), O(W,B))$ are in the symplectic stable
or metastable ranges. Then we may consider the eta correspondences
$$\eta^V : \widehat{O(W,B)} \rightarrow \widehat{\text{Sp}(V)} \cup \{0\}$$
$$\eta^{U} : \widehat{O(W,B)} \rightarrow \widehat{\text{Sp}(U)} \cup \{0\}.$$
For an irreducible representation $\pi \in \widehat{O(W,B)}$ such that
$\eta^V (W) \neq 0$, we have
\beg{NeededParabEta}{\eta^{U} (\pi ) \subseteq Ind_{P_{M-N}^U} (\eta^{V} (\pi)^\pm),}
where the $\pm$ denotes whether we consider a sign character on the factor $GL_{M-N} (\F_q)$
of the Levi subgroup before inflating to $P_{M-N}^U$ and applying the induction.
The sign is $+$ when $W$ is even dimensional and is $-$ when $W$ is odd dimensional.

\end{proposition}

To prove this, we now recall briefly the analogue of the Pieri rule for
symbols. We give a more concrete statement in Subsection \ref{PieriRuleSubsection}
below with a more detailed explanation on how it can be derived from the results
of \cite{LusztigSymbols}. For the purposes of proving Theorem \ref{MainMatchRes}, we only
need one case of it, so we do not give the full general statement in this subsection.

Consider a unipotent representation of $\text{Sp}_{2N}(\F_q)$ corresponding to a symbol
${\lambda_1< \dots < \lambda_a \choose \mu_1< \dots < \mu_b}$.
Write $P_{1}$ for the maximal parabolic subgroup of $\text{Sp}_{2(N+1)} (\F_q)$ with
Levi factor $\text{Sp}_{2N} (\F_q) \times GL_1 (\F_q)$, and consider ${\lambda_1< \dots < \lambda_a
\choose \mu_1< \dots < \mu_b}$ as a
its representation by letting the $GL_1 (\F_q)$ factor of the Levi
subgroup act trivially and
inflating on the unipotent radical trivially. Then its parabolic induction
to a $\text{Sp}_{2(N+1)} (\F_q)$-representation
$Ind_{P_1}^{\text{Sp}_{2(N+1)} (\F_q)} ({\lambda_1< \dots < \lambda_a \choose \mu_1< \dots < \mu_b})$
is a direct sum of unipotent representations corresponding to symbols
\beg{NoRowChange}{\begin{array}{c}
\displaystyle {\lambda_1< \dots < \lambda_{i-1} < \lambda_i+1< \lambda_{i+1} < \dots < \lambda_a
\choose \mu_1< \dots < \mu_b},\vspace{1mm}\\
\displaystyle { \lambda_1< \dots < \lambda_a \choose
\mu_1< \dots < \mu_{i-1} < \mu_i+1< \mu_{j+1} < \dots < \mu_a}
\end{array}}
when possible, i.e. for $1 \leq i \leq a$ or $1 \leq j \leq b$ where $\lambda_i +1 < \lambda_{i+1}$ or $\mu_j+1 < \mu_{j+1}$, respectively, and the unipotent representations
\beg{}{\begin{array}{c}
\displaystyle { 1 < \lambda_1+1 < \lambda_2 + 1< \dots < \lambda_a +1 \choose
0< \mu_1 +1 < \dots < \mu_b +1}\vspace{1mm}\\
\displaystyle {0< \lambda_1+1 <  \dots < \lambda_a +1 \choose 1 < \mu_1 +1 < \mu_2 +1 < \dots < \mu_b+1}
\end{array}}
when possible, i.e. when $\lambda_1 > 0$ or $\mu_1 > 0$, respectively.
This is a full description of the ``one step" Pieri rule.

More generally, for the ``$r$ step" Pieri rules, describing
the parabolic induction from a maximal parabolic $P_{r}$
with Levi subgroup $\text{Sp}_{2N} (\F_q) \times GL_r (\F_q)$ to $\text{Sp}_{2(N+r)} (\F_q)$
(still taking $GL_r (\F_q)$ and the unipotent
radical to act trivially on the input representation),
instead of adding a single ``box" to the underlying Young diagrams corresponding to
a symbol ${\lambda_1< \dots < \lambda_a \choose \mu_1< \dots < \mu_b}$, we must add a
``row of $r$ boxes."
More specifically, one must undo
the procedures described in Proposition 3.2 and Subsection 4.6 of \cite{LusztigSymbols},
then apply the classical Pieri rule adding a ``row of
$r$ boxes" as a Weyl group representation,
before re-applying the procedures of \cite{LusztigSymbols} to recover the original defect and the
new rank $N+r$.
This rule can be derived directly from the definition of the symbols (see \cite{LusztigSymbols},
Subsection 4.8). 

In particular, summands that always
appear in $Ind_{P_r} ({\lambda_1< \dots < \lambda_a\choose
\mu_1< \dots < \mu_b})$ are symbols obtained by adding $r$ to the final coordinate in a row:
\beg{AddToTheFinalCoordPieri}{\begin{array}{c}
\displaystyle {\lambda_1< \dots <\lambda_{a-1}< \lambda_a + r \choose \mu_1< \dots < \mu_b}\vspace{1mm}\\
\displaystyle {\lambda_1< \dots < \lambda_a \choose \mu_1< \dots < \mu_{b-1} < \mu_b +r}.
\end{array}
}

\vspace{2mm}

To apply such
a parabolic induction $Ind_{P_r}$ to a general representation $\rho$ of $\text{Sp}_{2N}(\F_q)$,
the resulting $\text{Sp}_{2(N+r)} (\F_q)$ representation
consists of summands which add $1$'s to the semisimple part of $\rho$'s classification data
and have unipotent part consisting of the input unipotent part with the factor corresponding
to the 
identity component of the centralizer of $1$ eigenvalues
replaced by the possible pieces of its $r$ step parabolic induction.
We also consider the ``signed parabolic induction" $Ind_{P_r} (\rho^-)$, by which we denote the
$\text{Sp}_{2(N+r)} (\F_q)$-representation obtained by tensoring
$\rho$ with the sign character of the $GL_r (\F_q)$ factor of the Levi subgroup
of $P_r$ before inflating and inducing. The procedure on classification data giving
the signed parabolic induction is completely similar to the unsigned case, except that $-1$'s
are added to the semisimple part of the data
corresponding to the input representation (instead of $1$'s) and
the symbol corresponding to this factor of the 
identity component of its centralizer is altered, instead.

In particular, by combining the symbol Pieri rule with our description of the eta
correspondence given in Definitions \ref{EtaOddCaseDefn} and \ref{EtaEvenCaseDefn}
we are able to conclude Proposition \ref{IndEtaProp}:

\begin{proof}[Proof of Proposition \ref{IndEtaProp}]
The choice of sign in \rref{NeededParabEta} precisely specifies whether the induction operation
will add $1$'s or $-1$'s to the classification data of the input representation.
Since it is chosen according to the partiy of $dim (W)$,
the semisimple part of the classification data of $\eta^U (\pi)$ agrees with
that of the irreducible summands of $Ind_{P^U_{M-n}} (\eta^V (\pi)^\pm )$,
reducing the claim to the fact that in the ``altered factor" of the unipotent parts of
$\eta^U (\rho)$ and $\eta^V (\rho)$, 
$${\lambda_1< \dots < \lambda_a < M'_\rho
\choose \mu_1< \dots < \mu_b} \subseteq Ind_{P_{M-N}} ({\lambda_1< \dots < \lambda_a < N'_\rho
\choose \mu_1< \dots < \mu_b} ),$$
which follows from the symbol Pieri rule.
\end{proof}

Considering the effect of parabolic induction on $U$-rank then allows us
to reduce Proposition \ref{AllUnstableTopProp} to Corollary \ref{StableTrivCase}, and conclude Theorem \ref{MainMatchRes}:

\begin{proof}[Proof of Proposition \ref{AllUnstableTopProp} and Theorem \ref{MainMatchRes}]
Consider a representation $\rho$ of $\text{Sp} (V)$ of tensor rank
$$N< rk_\otimes (\rho) \leq 2N.$$
Then by Theorem \ref{ExplicitTheoremSymp},
there exsits a choice of $(W,B)$ with $dim (W) = rk_\otimes (\rho)>N$,
and an irreducible representation
$\pi \in \widehat{O(W,B)}$ such that
$\eta^V_{W,B} (\pi) = \rho $.

Let us denote by $V'$ the symplectic space of dimension $dim (V') = 2 \cdot dim (W)$, i.e.
the maximal dimensional symplectic space such that $(\text{Sp}(V'), O(W,B))$ is a reducive dual pair
in the symplectic stable range. Consider the eta correpsondence
$$\eta^{V'} : \widehat{O(W,B)} \hookrightarrow \widehat{\text{Sp}(V')},$$
and its image of $\pi$. Let us write $\rho' = \eta^{V'} (\pi)  \in \widehat{\text{Sp}(V')}$.
Applying Corollary \ref{StableTrivCase}, we know that as a
representation of $\text{Sp}(V') = \text{Sp}_{2 \cdot dim (W)} (\F_q)$, its $U$-rank is
$rk_U (\rho' ) = rk_\otimes (\rho ') =  dim (W)$.

Now applying Proposition \ref{IndEtaProp} gives that $\rho '$ appears as a summand
of a (possibly signed) parabolic induction of $\rho $ from a parabolic subgroup with Levi factor
$\text{Sp} (V) \times GL_{dim (W)-N} (\F_q) \subseteq \text{Sp}(V')$,
which is an operation that can only increase $U$-rank by at most the difference $dim (W) -N$.
In other words, the $U$-rank of $\rho $ is at least
$$\begin{array}{c}
rk_U (\rho) \geq rk_U (\rho ') - ( dim (W) -N) =\\[1ex]
 rk_\otimes (\rho') - (dim (W)- N) = N,
\end{array}$$
and therefore we must have equality $rk_U (\rho) = N$, obtaining \rref{NeededMatchingThm}.
\end{proof}

\section{Resolving the alternating sums}\label{AlternatingSumSection}

In Theorem \ref{ExplicitTheoremSymp} and \ref{ExplicitTheoremOrtho},
we decompose the restricted oscillator representation
$Res_{\text{Sp}(V) \times O(W,B)} (\omega [V\otimes W])$ in terms of
the eta correspondence (in the case of the symplectic stable
or metastable range, corresponding to the condition \rref{EtaCorrStMetaStRange})
or the zeta correspondence (in the orthogonal
metastable or stable ranges, corresponding to the complementary condition
\rref{ZetaCorrStMetaStRange}).
Further, we described the eta and zeta correspondences
in terms of classification data, allowing us to directly compute
the $\text{Sp}(V)$-representation and $O(W,B)$-representation summands occuring in the restriction
of $\omega [V\otimes W]$.

However, in \rref{EtaThmDecompFull} and \rref{ZetaThmDecompFull}, the eta and zeta correspondence terms appear with coefficients
$\mathcal{A}_k (\rho, N_\rho')$ (giving an
$O(W,B)$-representation, for $\rho$ an irreducible $O(W[-k], B[-k])$-representation)
or $\mathcal{i} (\rho, m_\rho')$ (giving a $\text{Sp}(V)$-representations, for
$\rho$ an irreducible $\text{Sp}(V[-k])$, respectively.  
Our description of these terms for the purpose of proving the Theorem
was as certain alternating sums of parabolic inductions.
The purpose of this section is to simplify these sums
$\mathcal{A}_k(\rho , N_\rho' )$ (resp. $\mathcal{A}_i (\rho , m_\rho')$)
and describe their irreducible $O(W,B)$- (resp. $\text{Sp}(V)$-) representation summands
in a way that can be used for concrete computations.

\subsection{The main statement}

Our main result is

\begin{theorem}\label{AlternatingSumIntroPropStatement}
Fix a reductive dual pair $(\text{Sp}(V), O(W,B))$ in the symplectic stable 
or metastable range. 
For an irreducible representation $\rho$ of $O(W[-k], B[-k])$, consider the factor of
the unipotent part of its classification data writable as a symbol
\beg{PropMainUnipAlt}{{\lambda_1< \dots < \lambda_a \choose \mu_1< \dots < \mu_b},}
such that it is replaced by a symbol 
$${\lambda_1< \dots < \lambda_a < N_\rho' \choose \mu_1< \dots < \mu_b}$$
in the construction of $\eta^V_{W[-k], B[-k]} (\rho)$ (switch rows if necessary).
Then the $O(W,B)$-representation $\mathcal{A}_{k} (\rho , N_\rho')$ consists of the irreducible summands
appearing in the parabolic induction $Ind_{P_k} (\rho \otimes \epsilon (det))$ such that, when performing
the Pieri rule (see Proposition \ref{PieriRuleProp}) on the row
$\lambda_1< \dots < \lambda_a$ of \rref{PropMainUnipAlt},
the highest coordinate $\lambda_{a'}'$ of the corresponding row of the new symbol
satisfies
$$\lambda_{a'}' < N_\rho' + (a'-a).$$
There is a similar description in the case of $(\text{Sp}(V), O(W,B))$ in the orthogonal stable 
or metastable range, of the $\text{Sp}(V)$-representations $\mathcal{A}_i (\rho, m_\rho')$
for $\text{Sp}(V[-i])$-representations
$\rho$.
\end{theorem}

To prove this, we use the Pieri
rule for Lusztig symbols, which we state in Proposition \ref{PieriRuleProp} below. 
Recall that the combinatorial
data of a symbol (classifying the irreducible unipotent representations)
is equivalent to the data of its defect (which gives the information
of the underlying cuspidal representation) and a pair of Young diagrams corresponding
to the irreducible representation of the remaining Weyl group specifiying which piece of the induced
cuspidal representation the symbol corresponds to (this data comes from undoing the
procedures given in Proposition 3.2 and Subsection 4.6 or 4.7 of \cite{LusztigSymbols}; we give more details
below).

\vspace{1mm}

First, as we discussed in Subsection \ref{AlternatingSumSubsect},
in each case, in every summand
of the alternating sums of parabolic inductions all factor through precisely to the symbol altered by
the eta or zeta correspondence
(the symbol of the
factor of the unipotent part corresponding to $(-1)^{dim (W)}$ eigenvalues),
with the other classification data being preserved.
Therefore, the problem of simplifying $\mathcal{A}_k (\rho, N_\rho')$
(resp. $\mathcal{A}_k (\rho, m_\rho')$)
reduces to simplifying an alternating sum of parabolic inductions of symbols.
As in Definitions \ref{SympCaseAltSumsDefn} and \ref{OrthoCaseAltSumsDefn},
write $\theta = {\lambda_1< \dots < \lambda_a \choose
\mu_1< \dots < \mu_b}$ for the symbol corresponding to the factor
of the unipotent part of $\rho$'s classification data which is altered in the
extended eta correspondence
(resp. the extended zeta correspondence). Arrange
the rows so that the eta correspondence alters the top row.
Write
\beg{}{k^{(i)} = k-N_\rho' + \lambda_i ,}
so that we have $k^{(a)} \geq k^{(a-1)}  \geq \dots \geq k^{(1)}$.
Pick the minimal $i_0$ such that $k^{(i_0)}\geq 0$
We then need to
find the symbols $\chi$
appearing in the alternating sum $A_k^\pm (\theta, N_\rho')$
(with superscript sign agreeing with the sign of $\phi^\pm (u)$ or $\psi^\pm (u)$
appearing in the the construction of $\eta^V_{W,B} (\rho)$ or $\zeta^{W,B}_V (\rho)$),
which can be written out as
\beg{AlternatingSumOfSymbols}{
\bigoplus_{i=i_0}^{a+1} (-1)^{a+1-i}\cdot Ind_{P_{k^{(i)}}} (\text{${\lambda_1< \dots < \widehat{\lambda_i} < \dots < \lambda_a < N_\rho' \choose \mu_1< \dots <\mu_b}$})
}
(for the case of the extended zeta correspondence, replace $N_\rho' $ by $m_\rho '$).
The irreducible components of $\mathcal{A}_k  (\rho , N_\rho' )$ will then consist of
of $O(W,B)$-representations obtained with classification data obtained by adding $(-I)_{2k}$
to $\rho$'s original semisimple part, and taking the unipotent part obtained by tensoring a symbol
in \rref{AlternatingSumOfSymbols} (and choosing the same central sign data as $\rho$).

In other words, Theorem \ref{AlternatingSumIntroPropStatement} can be more precisely stated as
\begin{theorem}\label{MorePreciseAlternatingSumProp}
Assume the above notation. The alternating sum 
\rref{AlternatingSumOfSymbols} simplifies as the sum of all symbols
$${\lambda_1 '< \dots< \lambda_{a'}' \choose \mu_1< \dots< \mu_{b'}'}\subseteq Ind_{P_k} ({\lambda_1< \dots < \lambda_a \choose \mu_1< \dots < \mu_b})$$
such that
$$\lambda_{a'} < N_\rho' + (a'-a).$$
\end{theorem}

This statement can now be proved directly by considering the induction rule on symbols,
which can be done by translating the symbols into Young diagram data.

Consider a symbol of type $B$, $C$, $D$ or ${}^2 D$-type
\beg{GenericSymbol}{{\lambda_1< \dots < \lambda_a \choose \mu_1< \dots < \mu_b}.}
We want to consider the ``$k$-step induction"
\beg{GenerickIndOfSymbol}{Ind_{P_k} ({\lambda_1< \dots < \lambda_a \choose \mu_1< \dots < \mu_b}),}
where $P_k$ denotes the standard maximal parabolic $P_k^{W,B}$ if we take
\rref{GenericSymbol} to correspond to a unipotent representation of $O(W[-k], B[-k])$ for some orthogonal
space $(W,B)$,
or $P_k^{V}$ if we take \rref{GenericSymbol} to correspond to a unipotent representation of $\text{Sp}(V[-k])$
for some symplectic space $V$. To consider a symbol \rref{GenericSymbol} as a representation of
$P_k$ in either of these cases, let the factor $GL_k (\F_q)$ of the Levi subgroup of $P_k$ act trivially,
and inflate by letting. Then \rref{GenericSymbol} denotes the induction of the resulting
representation to $O(W,B)$ or $\text{Sp}(V)$.
The decomposition of \rref{GenerickIndOfSymbol} is according to a {\em Pieri rule}, which
we concretely state in Proposition \ref{PieriRuleProp} below.

\subsection{The Pieri rule}\label{PieriRuleSubsection}

First, let us briefly recall the role of symbols as representations. In (4.6.2), (4.7.1) of \cite{LusztigSymbols},
Lusztig described how the combinatorial data of a pair of increasing sequences \rref{GenericSymbol}
corresponds to an irreducible representation of a certain Hecke algebra $\mathscr{H}_n (q, q^d)$
defined according to certain relations (see Subection 4.1 of \cite{LusztigSymbols}) which
are equivalent to the classical Iwahori relations and recover the group algebra of the Weyl
group (see \cite{CurtisReiner}, $\mathsection$68A).
In Subsection 4.8 of \cite{LusztigSymbols},
Lusztig also describes how induction is preserved by these correspondences.
For the Weyl groups of the groups we consider here, the irreducible representations in each
case are classified by pairs of Young diagrams whose total numbers of boxes add
up to the rank.
Therefore, the induction \rref{GenerickIndOfSymbol} can be computed by applying the Pieri rule
to these Young diagrams i.e., by considering
all choices of $k_1+k_2 = k$, and adding a row of length $k_1$ to the top row's corresponding
Young diagram and a row of length $k_2$ to the bottom row's corresponding Young diagram.

To find the Weyl group representation corresponding to a symbol \rref{GenericSymbol},
without loss of generality, switch rows so that $a\geq b$, and denote the defect by
$d = a-b$. 

First suppose $d$ is strictly positive.
The symbol notation then indicates that the unipotent representation
$\lambda_1< \dots < \lambda_a \choose \mu_1< \dots < \mu_b$ is constructed in an induction of
the cuspidal unipotent representation corresponding to the symbol
$0< 1<2< \dots< d-1 \choose \emptyset $
(the minimal rank symbol of defect $d$).
The first step of Lusztig's procedure is to ``remove" this cuspidal representation from the symbol
(i.e. by undoing the bijection $j$ of Proposition 3.2 of \cite{LusztigSymbols}), 
to obtain a defect one symbol 
\beg{NewDefect1Symbol}{{\lambda_1< \dots <\lambda_a\choose 0<1< \dots < d-2< \mu_1 + (d-1)< \dots < \mu_b + (d-1)},}
(using the convention of \cite{LusztigSymbols} describing how to reduce a symbol if 
the coordinate of its first two rows is $0$).

The next step is to
undo the procedure described in Subsection 4.6 of \cite{LusztigSymbols} to obtain Young diagrams.
In the case of \rref{NewDefect1Symbol}, we obtain a Young diagram
\beg{TopYoungDiagram}{(\lambda_1 \leq \lambda_2 -1 \leq \dots \leq \lambda_a - (a-1))}
where the $i$th row has length $\lambda_{a-i+1} -(a-i)$ corresponding to the top row,
and a Young diagram
\beg{BottomYoungDiagram}{(\mu_1\leq \mu_2 -1 \leq \dots \leq \mu_b - (b-1))}
where the $i$th row has length $\mu_{b-i+1} - (b-i)$ corresponding to the bottom row
(not writing the rows with length $0$ corresponding to the coordinates
$0<1< \dots < d-2$ concatenated onto the bottom row in \rref{NewDefect1Symbol}).

In the case of defect $d=0$, we undo the procedure in Subsection 4.7 of
\cite{LusztigSymbols} to obtain this same answer, of a Weyl group representation
corresponding to Young diagrams \rref{TopYoungDiagram}, \rref{BottomYoungDiagram}.

We will denote the Young diagrams \rref{TopYoungDiagram}, \rref{BottomYoungDiagram},
by $\alpha$, $\beta$, denoting the $i$th row lengths by
\beg{AlphaBeta}{\begin{array}{c}
\alpha_i := \lambda_{a-i+1} - (a-i)\\[1ex]
\beta_i := \mu_{b-i+1} - (b-i).
\end{array}}
(We use the convention that the first row of the Young diagram is the longest.)

We use the terminology that, for a Young diagram $(\gamma_n \leq \dots \leq \gamma_1)$
and for a natural number $k$,
we say Young diagrams of the form
$$(k_{n+1} \leq \gamma_n + k_n \leq \gamma_n + k_{n-1} \leq \dots \leq \gamma_1+ k_1)$$
where $k_i$ are natural numbers satisfying $k_1+ \dots + k_{n+1} = k$ and, for every $i =1 ,\dots , n$,
we have $k_{i+1}\leq \gamma_{i} - \gamma_{i+1}$ (putting $\gamma_{n+1} = 0$ 
are the Young diagrams {\em obtained by adding a row of length $k$ to $\gamma$}.

\vspace{3mm}

The Pieri rule for symbols can then be stated as follows:

\begin{proposition}\label{PieriRuleProp}
Fix an orthogonal space $(W,B)$ (resp. a symplectic space $V$)
and consider a symbol ${\lambda_1< \dots < \lambda_a \choose\mu_1< \dots < \mu_b}$
defining a unipotent representation of $O(W[-k], B[-k])$ (resp. $\text{Sp}(V[-k])$). 
Recall that we denote by $P_k$ the standard maximal parabolic with Levi subgroup
$O(W[-k], B[-k])\times GL_k (\F_q)$ (resp. $\text{Sp}(V[-k]) \times GL_k (\F_q)$), and consider
the symbol as a $P_k$-representation by letting $GL_k (\F_q)$ act trivially and inflating by letting
the unipotent part of the parabolic act trivially. Then its parabolic induction to an $O(W,B)$-representation
decomposes as a sum of symbols
$$\text{Ind}_{P_k} ({\lambda_1< \dots < \lambda_a \choose \mu_1< \dots < \mu_b}) =
\bigoplus_{\mathcal{S}_k[{\lambda_1 < \dots < \lambda_{a} \choose \mu_1 < \dots < \mu_{b}}]} {\lambda_1' < \dots < \lambda_{a'}' \choose \mu_1' < \dots < \mu_{b'}'}$$
where the sum runs over the set of symbols $\mathcal{S}_k[{\lambda_1 < \dots < \lambda_{a} \choose \mu_1 < \dots < \mu_{b}}]$ consisting of 
${\lambda_1' < \dots < \lambda_{a'}' \choose \mu_1' < \dots < \mu_{b'}'}$ where, for some
$k_1 + k_2 = k$,
the Young diagram 
$$(\lambda_1' \leq \lambda_2' -1 \leq \dots \leq \lambda_{a'}' -(a'-1))$$
is obtained by adding a row of length
$k_1$ to 
$$(\lambda_1 \leq \lambda_2 -1 \leq \dots \leq \lambda_{a} -(a-1)),$$
and the Young diagram
$$(\mu_1' \leq \mu_2' -1 \leq \dots \leq \mu_{b'}' -(b'-1))$$
is obtained by adding a row of length $k_2$ to
$$(\mu_1 \leq \mu_2 -1 \leq \dots \leq \mu_{b} -(b-1)).$$
\end{proposition}

\noindent (The awkwardness of this statement is in order
to accomodate all cases of input symbols,
including those where one of the rows of the original symbol
begins with a $0$ coordinate, and to properly address
the case when $a'$ and $b'$ are larger than $a$ and $b$.)

\subsection{The proof of the main statement}

We then conclude Theorem \ref{MorePreciseAlternatingSumProp}
by considering it in terms of Young diagrams, and applying this Pieri rule.

\begin{proof}[Proof of Theorem \ref{MorePreciseAlternatingSumProp}]
Let us use the above notation \rref{AlphaBeta}
for an alterable unipotent part of a representation $\rho$ for which we want
to find the coefficient $\mathcal{A}_k (\rho ,N_\rho' )$ appearing with $\eta^V_{W,B} (\rho)$.
Since the resulting Young diagrams \rref{TopYoungDiagram}, \rref{BottomYoungDiagram}
ultimately do no depend on which row was on top or longer,
we may say without loss of generality that the top row is the one where we add the coordinate
$N_\rho'$ when constructing $\eta^V_{W,B} (\rho)$.

Applying this to the context of Theorem \ref{MorePreciseAlternatingSumProp}, the symbol 
\beg{AlteredMissingiNewCoord}{\lambda_1< \dots < \lambda_{i-1} < \lambda_{i+1}< \dots < \lambda_{a} < N_\rho'\choose
\mu_1< \dots < \mu_b}
appearing in the $i$th term
of the alternating sum \rref{AlternatingSumOfSymbols} 
corresponds to the same cuspidal representation
$0 <1< \dots< d-1 \choose \emptyset$ as the original symbol (since it has the same row lengths)
and the Young diagrams
\beg{ithYoungDiagramTerm}{\begin{array}{c}
(\alpha_a \leq \dots \leq \alpha_{a-i+2} \leq \alpha_{a-i}+1 \leq \dots \leq \alpha_1 +1 \leq N_\rho' - a) \\[1ex]
(\beta_b \leq \dots \leq \beta_1).
\end{array}
}
(See Figure 1 for an example.)

\begin{figure}
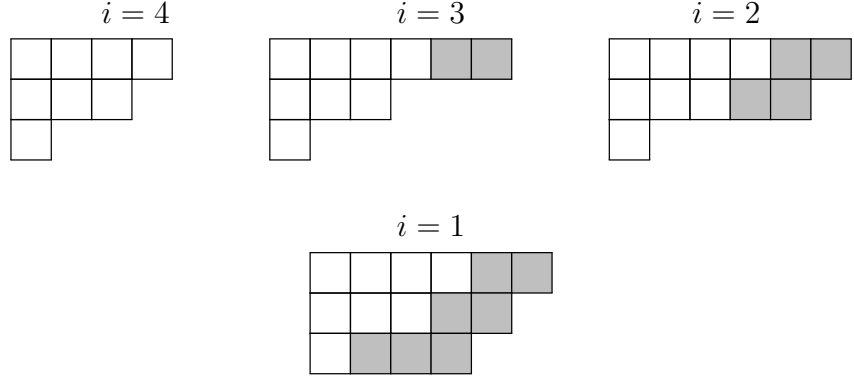

\ytableausetup
{boxsize=1.25em}
\ytableausetup
{aligntableaux=top}

$$i=4 \hspace{30mm} i = 3 \hspace{30mm} i = 2$$
\begin{ytableau}
*(gray!0) &  &  &  \\
  &  &  \\
 \\
\end{ytableau}
\hspace{10mm}
\begin{ytableau}
*(gray!0) &  &  & & *(gray!50) & *(gray!50)\\
  &  &  \\
 \\
\end{ytableau}
\hspace{10mm}
\begin{ytableau}
*(gray!0) &  &  & & *(gray!50) & *(gray!50)\\
  &  &  &*(gray!50) &*(gray!50)   \\
 \\
\end{ytableau}

\vspace{3mm}
$$i=1$$
\begin{ytableau}
*(gray!0) &  &  & & *(gray!50) & *(gray!50)\\
  &  &  &*(gray!50) &*(gray!50)   \\
 &*(gray!50) & *(gray!50)  & *(gray!50)  \\
\end{ytableau}

\caption{\tiny These are the top row Young diagrams corresponding to the symbols \rref{AlteredMissingiNewCoord}, in the case
of $\alpha = (1 \leq 3 \leq 4)$, $N_\rho' -a = 6$. The boxes highlighted gray show the skew semistandard
tableau which, after removal from each Young diagram recovers the original $\alpha$.
At each $i$, just enough boxes are added to a row
of the $(i+1)$th Young diagram to not be obtainable from the $(i+2)$th.}
\end{figure}

\vspace{1mm}

The claimed reduction of the alternating sum consists of symbols
corresponding to the same underlying cuspidal representation and
a pair of Young diagrams $(\alpha', \beta')$ obtained from applying the $k$-step Pieri rule
to $(\alpha ,\beta)$ such that the first row length of $\alpha'$ is strictly bounded
\beg{YoungDiagramVersionOfTopRowBound}{\alpha'_1  < N_\rho' - a.}
First we note that the summands of the initial parabolic induction satisfying this condition
must survive in the alternating sum, since the condition \rref{YoungDiagramVersionOfTopRowBound}
guarantees that such an $(\alpha', \beta')$ cannot appear in the parabolic induction of any
of the other terms for $i\leq a$, since the first row is shorter than the first row of any top Young diagram
in \rref{ithYoungDiagramTerm}.

\vspace{1mm}

It remains to see that every other term in the alternating sum vanishes.
First, for summands of the initial parabolic induction term $Ind_{P_k} ({\lambda_1<\dots < \lambda_a
\choose \mu_1< \dots < \mu_b})$ which fail the condition
\rref{YoungDiagramVersionOfTopRowBound} are cancelled in the next term at $i=a$, since its corresponding
top Young diagram $(\alpha_a \leq \dots \leq \alpha_2 \leq N_\rho' - a)$.
Similarly, we may proceed inductively to see that every pair of Young diagrams
appearing from the $i$th term of the alternating sum which was not used to cancel the $(i+1)$th
term is cancelled at the $(i-1)$th term. This follows
since, at each step $i$, the condition on a pair of Young
diagrams to have already appeared
in the $(i+1)$th term (and therefore be cancelled in the previous step) is that the length of the
$(a-i+2)$th row does not exceed what can be attained from the previous step without adding boxes directly
above each other, i.e.
\beg{RowLegalityCondPieri}{\alpha_{a-i+2}' \leq \alpha_{a-i+1}.}
On the other hand, the condition on $\alpha '$ to appear in the $(i-1)$th term is complementary
$$\alpha_{a-i+2}' \geq \alpha_{a-i+1} +1,$$
since the condition is that
the $(a-i+2)$th row is at least as long as the corresponding
row of the top Young diagram in the $(i-1)$th term,
so all the cases are covered.

\vspace{3mm}

It remains therefore to check that every term
from the final term (corresponding to the lowest $i$) in the alternating sum cancels.
Recall that the sum \rref{AlternatingSumOfSymbols} has final term at $i =i_0$, which is the smallest.
First, suppose $i_0 >1$. Then by definition
\beg{NextI0Negative}{k^{(i_0 -1)} = k-N_\rho' +  \lambda_{i_0-1} < 0,}
and thus, when considering the final induction
$$\text{Ind}_{P_{k^{(i_0)}}} ({\lambda_1< \dots < \widehat{\lambda_{i_0}} < \dots < \lambda_a \choose \mu_1< \dots < \mu_b}),$$
on the Young diagram level, every top Young diagram appearing satisfies the condition
\rref{RowLegalityCondPieri}, since it we would need to add
at least $\lambda_{i_0} -\lambda_{i_0 -1}$ boxes, which is strictly more than $k^{(i_0)}$ by
\rref{NextI0Negative}. Therefore, all terms are used up in the previous step, giving the claim.
In the case when $i_0 = 1$, we again know that \rref{RowLegalityCondPieri} must
also follow, by the metastable dimension condition, since the final number of boxes to be added is
$$k^{(1)} = k-N_\rho' + \lambda_1 <\lambda_1$$
(since $k \leq m$ and $N_\rho' > m$ because we assumed that
$V$ and $(W,B)$ lie in the symplectic stable or metastable range).
\end{proof}

\noindent {\bf Example:} We emphasize the point that this alternating sum,
even in extreme cases, need not by irreducible.
The largest symbol (giving the unipotent irreducible representation of maximal
dimension) for $SO_{2m+1} (\F_q)$ is
$${0<1< 2< \dots < m \choose 1< 2< \dots < m},$$
which gives the Steinberg representation $St_m$, of dimension $q^{m^2}$.
In this case, its corresponding pair of Young diagrams consists of a column of $m$ boxes (corresponding
to the top row) and the empty Young diagram (corresponding to the bottom row
In this case, the Pieri rule adding a single box gives
$$
\begin{array}{c} 
\displaystyle \text{Ind}_{P_1} (St_{m})= {0<1 <2< \dots< m-1< m+1 \choose 1< 2< \dots < m
} \oplus \\
\\
\displaystyle {0<1<2 <\dots < m \choose 1<2< \dots < m-1< m+1 } \oplus St_{m+1}
\end{array}$$
Then, for example, the alternating sum of symbols $A_1^\pm (St_m, m+1)$ with
sign chosen to alter the top row outputs
$$\begin{array}{c}
\displaystyle \text{Ind}_{P_1} (St_{m}) - {0<1 <2< \dots< m-1< m+1 \choose 1< 2< \dots < m } =\\
\\
\displaystyle  {0<1<2 <\dots < m \choose 1<2< \dots < m-1< m+1} \oplus St_{m+1}
\end{array}$$

\appendix
\section{Dictionary with the notation of S.-Y. Pan}\label{PanAppendix}

The purpose of this Appendix is to provide a dictionary between the notation used in this paper,
and the notation used by S.-Y. Pan in \cite{Pan1, Pan2}.

First note that Pan uses
a modified description of the unipotent irreducible
representations of the symplectic and orthogonal groups.
We consider the classical description of unipotent irreducible representations
of a connected group $G$, say $SO_{2m+1} (\F_q)$, $\text{Sp}_{2N} (\F_q)$, or $SO_{2m}^\pm (\F_q)$
in terms of Lusztig symbols with unordered rows,
so that
\beg{UnorderedSymbol}{{\lambda_1< \dots < \lambda_a \choose \mu_1< \dots < \mu_b} = {\mu_1< \dots < \mu_b \choose
\lambda_1< \dots < \lambda_a}}
(requiring, in the case when $G$ is of $B$- or $C$-type, simply that the defect is odd).
This is sufficient for our discussion of the irreducible unipotent representations of symplectic
groups and the odd orthogonal groups (where the center splits off $O_{2m+1}(\F_q) = \mu_2 \times SO_{2m+1} (\F_q)$). 
In the case of $D$-type, the combinatorial data of symbols
\rref{UnorderedSymbol}, requiring defect to be $0$ or $2$ mod $4$, correspond
to the irreducible unipotent representations of $SO_{2m}^+ (\F_q)$ and $SO_{2m}^- (\F_q)$.

However, performing an induction from $SO_{2m}^\pm (\F_q)$ to $O_{2m}^\pm (\F_q)$
(for $m>0$), the resulting unipotent representation of $O_{2m}^\pm (\F_q)$ of twice the dimension
of ${\lambda_1< \dots < \lambda_a \choose \mu_1< \dots < \mu_b}$ splits into two
non-isomorphic equidimensional pieces.
In our notation, we identify these pieces by labelling these unipotent irreducible
$O_{2m}^\pm (\F_q)$-representations according to their underlying symbol
\rref{UnorderedSymbol} and the data of a central sign $\pm$ indicating the action of
$\mu_2 = O_{2m}^\pm (\F_q)/ SO_{2m}^\pm (\F_q)$, determining the relevant piece
of the underlying symbol's induction:
$$\begin{array}{c}
\displaystyle Ind_{SO_{2m}^\pm (\F_q)}^{O_{2m}^\pm (\F_q)} {\lambda_1< \dots < \lambda_a\choose
\mu_1< \dots < \mu_b} = \\
\\
r^{O_{2m}^\pm (\F_q)}[(1),{\lambda_1< \dots < \lambda_a\choose
\mu_1< \dots < \mu_b}]^{(+1)} \oplus r^{O_{2m}^\pm (\F_q)}[(1),
{\lambda_1< \dots < \lambda_a\choose
\mu_1< \dots < \mu_b}]^{(-1)}.
\end{array}$$

On the other hand, Pan labels these two unipotent irreducible $O_{2m}^\pm (\F_q)$-representations
by enforcing an ordering on the rows of a symbol (in this paper, we distinguish this notation by
using square brackets), so that
$$Ind_{SO_{2m}^\pm (\F_q)}^{O_{2m}^\pm (\F_q)} {\lambda_1< \dots < \lambda_a\choose
\mu_1< \dots < \mu_b} = \begin{bmatrix} 
\lambda_1< \dots < \lambda_a\\
\mu_1< \dots < \mu_b
\end{bmatrix} \oplus
\begin{bmatrix} 
\mu_1< \dots < \mu_b\\
\lambda_1< \dots < \lambda_a
\end{bmatrix}$$
for non-isomorphic
$$ \begin{bmatrix} 
\lambda_1< \dots < \lambda_a\\
\mu_1< \dots < \mu_b
\end{bmatrix} \neq
\begin{bmatrix} 
\mu_1< \dots < \mu_b\\
\lambda_1< \dots < \lambda_a
\end{bmatrix}.$$
(We note Pan also changes the order of a symbol's entries in his notation
in \cite{Pan1},
writing each row in strictly descreasing order).
Pan also imposes an ordering on the symbols ccorresponding to the unipotent
irreducible representations of $\text{Sp}_{2N} (\F_q)$ or $SO_{2m+1} (\F_q)$, demanding then that
$$ \begin{bmatrix} 
\lambda_1< \dots < \lambda_a\\
\mu_1< \dots < \mu_b
\end{bmatrix}$$
has defect $a-b$ exactly $1$ mod $4$ (so that there is exactly one
ordered symbol corresponding to an odd defect unordered symbol \rref{UnorderedSymbol}).

\vspace{1mm}

Now, Pan describes in \cite{Pan2} the decomposition of the unipotent part of
\beg{AppenResdOsc}{Res_{\text{Sp}(V) \times O(W,B)} (\omega [V\otimes W])}
for even-dimensional $W$ as a direct sum of tensor products $\rho \otimes \pi$ for certain
pairs of irreducible unipotent representations $\rho \in \widehat{\text{Sp}(V)}$ and
$\pi \in \widehat{O(W,B)}$. Specfically, Pan describes that in the split case
$O(W,B) = O_{2m}^+ (\F_q)$, a tensor product of an ordered
symbol of $\text{Sp}(V)$ with an ordered symbol of $O(W,B)$
\beg{TensorOfTwoOrderedSymbols}{\begin{bmatrix} 
\lambda_1< \dots < \lambda_a\\
\mu_1< \dots < \mu_b
\end{bmatrix} \otimes \begin{bmatrix} 
\lambda_1 '< \dots < \lambda_{a'}'\\
\mu_1 ' < \dots < \mu_{b'}'
\end{bmatrix}}
is a summand of the unipotent part of \rref{AppenResdOsc} precisely when
\begin{itemize}
\item The Young diagram with row lengths
$$(\lambda_{a'}' - (a'-1), \lambda_{a'-1} - (a'-2), \dots , \lambda_1')$$ can be obtained by adding
a row (we discuss this notion in more detail in Subsection \ref{PieriRuleSubsection})
to the Young diagram with row lengths
$$(\mu_b - (b-1), \mu_{b-1} - (b-2), \dots , \mu_1).$$

\vspace{2mm}

\item The Young diagram with row lengths
$$(\lambda_{a} - (a-1), \lambda_{a-1} - (a-2), \dots , \lambda_1)$$ can be obtained by adding
a row to the Young diagram with row lengths
$$(\mu_{b'}' - (b'-1), \mu_{b'-1}' - (b'-2), \dots , \mu_1').$$

\vspace{2mm}

\item The defects precisely satisfy $a'-b' = -(a-b) +1$.

\end{itemize}
Similarly, Pan proves that in the non-split case
$O(W,B) = O_{2m}^- (\F_q)$, a tensor product \rref{TensorOfTwoOrderedSymbols}
of an ordered
symbol of $\text{Sp}(V)$ with an ordered symbol of $O(W,B)$ 
is a summand of the unipotent part of \rref{AppenResdOsc} precisely when
\begin{itemize}
\item The Young diagram with row lengths
$$(\mu_{b'}' - (b'-1), \mu_{b'-1} - (b'-2), \dots , \mu_1')$$ can be obtained by adding
a row to the Young diagram with row lengths
$$(\lambda_a - (a-1), \lambda_{a-1} - (a-2), \dots , \lambda_1).$$

\vspace{2mm}

\item The Young diagram with row lengths
$$(\mu_{b} - (b-1), \mu_{b-1} - (b-2), \dots , \mu_1)$$ can be obtained by adding
a row to the Young diagram with row lengths
$$(\lambda_{a'}' - (a'-1), \lambda_{a'-1}' - (a'-2), \dots , \lambda_1').$$

\vspace{2mm}

\item The defects precisely satisfy $a'-b' = -(a-b) -1$.

\end{itemize}
Therefore, Pan decomposes the unipotent part of the restriction of an oscillator representation
to $\text{Sp}(V) \times O(W,B)$, for even-dimensional $W$. (Of course, in the case of odd-dimensional $W$,
there is no unipotent part of the restriction of $\omega [V\otimes W]$.)

\vspace{3mm}

In \cite{Pan2}, Pan approaches the question of which pairs $\rho \otimes \pi$ appear with non-zero
multiplicity in the restricted oscillator representation for general irreducible representations
$\rho \in \widehat{\text{Sp}(V)}$ and $\pi \in \widehat{O(W,B)}$ by claiming a ``commutation with the
Lusztig correspondence." In essence, this means that a pair of irreducible representations appears
with non-zero multiplicity precisely when a pair of factors (corresponding to a certain eigenvalue
of the semisimple part of the classification data) of the unipotent part of their
classification data appears in the unipotent part of the appropriate restricted oscillator representation.

\vspace{2mm}

We now explain how our decomposition recovers Pan's classification of
the occurring summands.

First, we note that our constructions of the pieces of the restricted oscillator representation
always only involve ``altering" the Lusztig data of an input representation associated
to a certain specific eigenvalue of the semisimple data (the eigenvalue $1$ for even orthogonal groups
and the eigenvalue $-1$ for odd orthogonal groups). In particular, we can see the effect
of "commuting with the Lusztig correspondence" in the sense which Pan uses to pass from 
his decomposition of the unipotent part 
of the restricted oscillator representation
to information about the irreducible pairs which could appear in the full representation.
This reduces us to needing to see that the unipotent part of our decomposition
in each case matches the description of Pan's unipotent summands.

We begin with considering $(V, (W,B)$ in the symplectic stable or metastable
range. We use the eta correspondence and its extension in this case. From the decomposition
given in Theorem \ref{ExplicitTheoremSymp}, we find the unipotent
part of the restriction of $\omega [V\otimes W]$ to $\text{Sp}(V) \times O(W,B)$ is the sum
over $k = 0, \dots , h_B$ and over every unipotent irreducible representation $\pi$
of $O(W[-k], B[-k])$
of summands of the form
$$\eta^V_{W[-k],B[-k]} (\pi) \otimes \mathcal{A}_k (\pi , N_\pi' ).$$
Say the restriction of $\pi$ to $SO (W[-k], B[-k])$ corresponds to the symbol
${\alpha_1< \dots < \alpha_a \choose \beta_1< \dots < \beta_b}$ (we use different
notation for the entries here to avoid confusion with Pan's notation), so that
$N_\pi' = N-m + \frac{a+b}{2}$. Say $\eta^V_{W,B} (\pi)$ is constructible. Then,
depending on the central sign, $\eta^V_{W,B} (\pi)$ is the unipotent irreducible representation
of $\text{Sp}(V)$ corresponding to one of the symbols
\beg{AlphBetaWhereIsN'}{{\alpha_1< \dots < \alpha_a \choose \beta_1< \dots < \beta_b< N_\pi' } \text{ or }
{\alpha_1< \dots < \alpha_a < N_\pi' \choose \beta_1< \dots < \beta_b }.}
Now, for the orthogonal groups representation factor,
the parabolic induction of this symbol is a sum of
symbols
\beg{Alph'Beta'}{{\alpha_1'< \dots < \alpha_{a'}' \choose \beta_1' < \dots < \beta_{b'}'}}
such that the Young diagram 
with row lengths 
$$(\alpha_{a'}' - (a'-1), \alpha_{a'-1}' - (a'-2), \dots , \alpha_1')$$
can be obtained by adding a row to the Young diagram with row lengths
$$(\alpha_{a} - (a-1), \alpha_{a-1} - (a-2), \dots , \alpha_1),$$
and similarly, the Young diagram 
with row lengths 
$$(\beta_{b'}' - (b'-1), \beta_{b'-1}' - (b'-2), \dots , \beta_1')$$
can be obtained by adding a row to the Young diagram with row lengths
\beg{2conditionsapppendix}{(\beta_{b} - (b-1), \beta_{b-1} - (b-2), \dots , \beta_1),}
according to the Pieri rule.
The summands which contribute and survive in the alternating sum $\mathcal{A}_k (\pi, N_\pi')$
are these symbols \rref{Alph'Beta'} such that 
$$\beta'_{b'} - (b'-1) \leq N_\pi' - b \text{ or }\alpha'_{a'} - (a'-1) \leq N_\pi' -a ,$$
respectively (corresponding to \rref{AlphBetaWhereIsN'}). To understand
the relation with Pan's description, we can rephrase the conditions \rref{2conditionsapppendix}
as demanding
that the Young diagram 
$$(N_\pi' - a, \alpha_{a} - (a-1), \alpha_{a-1} - (a-2), \dots , \alpha_1)$$ 
can be obtained by adding a row to
$(\alpha_{a'}' - (a'-1), \alpha_{a'-1}' - (a'-2), \dots , \alpha_1')$ or, respectively, that
the Young diagram 
$$(N_\pi' - b, \beta_{b} - (b-1), \beta_{b-1} - (b-2), \dots , \beta_1)$$
can be obtained by adding a row to
$(\beta_{b'}' - (b'-1), \beta_{b'-1}' - (b'-2), \dots , \beta_1')$.

Let us now suppose we are specifically working in the split case,
i.e. $SO(W, B) = SO_{2m}^+ (\F_q)$, so that $a-b$ is $2$ mod $4$. 
The summand we have identified, in the ordered symbol notation, is
$$\begin{bmatrix} 
\beta_1< \dots < \beta_b < N_\pi'\\
\alpha_1< \dots < \alpha_a
\end{bmatrix} 
\otimes \begin{bmatrix} 
\alpha_1'< \dots < \alpha_{a'}'\\
\beta_1'< \dots < \beta_{b'}'
\end{bmatrix} 
$$ 
or
$$\begin{bmatrix} 
\alpha_1< \dots < \alpha_a< N_\pi'\\
\beta_1< \dots < \beta_b 
\end{bmatrix} 
\otimes \begin{bmatrix} 
\beta_1'< \dots < \beta_{b'}'\\
\alpha_1'< \dots < \alpha_{a'}'
\end{bmatrix} 
$$
(with the above corresponding restrictions on $\alpha_i'$ and $\beta_j'$), respectively.
We can therefore see that by re-labelling the top rows of the symbols as $\lambda$ and $\lambda'$
and the bottom rows as $\mu$ and $\mu'$, we exactly recover the conditions Pan described.
The non-split case of $O_{2m}^- (\F_q)$ proceeds similarly.

\vspace{3mm}

Now let us consider $(V, (W,B)$ in the orthogonal stable or metastable range.
We use the zeta correspondence and its extension in this case. From the decomposition
given in Theorem \ref{ExplicitTheoremOrtho}, we find the unipotent
part of the restriction of $\omega [V\otimes W]$ to $\text{Sp}(V) \times O(W,B)$ is the sum
over $k = 0, \dots , N$ and over every unipotent irreducible representation $\rho$
of $\text{Sp}(V[-k])$
of summands of the form
$$\mathcal{A}_k (\rho , N_\rho' ) \otimes \zeta^{W,B}_{V[-k]} (\rho).$$
Say $\rho$ corresponds to a symbol ${\alpha_1< \dots< \alpha_a \choose \mu_1< \dots < \mu_b}$,
and switch rows so that $a-b$ is $1$ mod $4$. Then
$N_\rho' = N-m+ \frac{a+b-1}{2}$. Say $\zeta^{W,B}_V (\rho)$ is constructible.
Then, in the split case $O(W,B) = O_{2m}^+ (\F_q)$, we get that the underlying $SO_{2m}^\pm (\F_q)$
symbol is
$${\alpha_1< \dots < \alpha_a < N_\rho' \choose \mu_1< \dots < \mu_b},$$
with determined central sign data.
Recalling the above description of the Pieri rule, the summands surviving
in $\mathcal{A}_k (\rho, N_\rho')$ consist of symbols
${\alpha_1'< \dots< \alpha_{a'}'\choose \mu_1< \dots< \mu_{b'}'}$ 
such that the Young diagram with row lengths
$(\alpha_{a'}' - (a'-1), \alpha_{a'-1}'-(a'-2), \dots , \alpha_{1}')$
can be obtained from adding a row to $(\alpha_a - (a-1), \alpha_{a-1} - (a-2), \dots , \alpha_1)$
such that $\alpha_{a'}-(a'-1) \leq N_\rho'- a$ and such that the Young diagram
with row lengths $(\beta_{b'}' - (b'-1), \beta_{b'-1}'-(b'-2), \dots , \beta_{1}')$
can be obtained from adding a row to $(\beta_b- (b-1), \beta_{b-1} - (b-2), \dots , \beta_1)$.
Again, the condition on the top row is equivalent to requiring that
the Young diagram 
$$(N_\rho' -a , \alpha_a - (a-1),\alpha_{a-1} - (a-2), \dots , \alpha_1)$$
can be obtained by adding a row to
$(\alpha_{a'}' - (a'-1), \alpha_{a'-1}'-(a'-2), \dots , \alpha_{1}')$.

In the ordered Lusztig symbol notation, these term we have identified are of the form
$$\begin{bmatrix} 
\alpha_1' < \dots < \alpha_{a'}' \\
\beta_1' < \dots < \beta_{b'}'
\end{bmatrix} \otimes
\begin{bmatrix} 
\beta_1 < \dots < \beta_b \\
\alpha_1 < \dots < \alpha_a < N_\rho'
\end{bmatrix}
$$
and we can see that by re-labelling the entries of the top row as $\lambda$ and $\lambda'$
and the bottom row as $\mu$ and $\mu'$, our conditions for $\alpha_i'$ and $\beta_j'$
are precisely those stated by Pan.
The non-split case of $O(W,B) = O_{2m}^- (\F_q)$ proceeds similarly.

\end{document}